\documentclass[10pt]{amsart}
\usepackage{mathtext}
\usepackage[T1,T2A]{fontenc}
\usepackage[cp1251]{inputenc}
\usepackage[english]{babel}
\usepackage{amsmath}
\usepackage{amsfonts}
\usepackage{amssymb}
\usepackage{amsthm}
\usepackage{graphicx}
\usepackage{wrapfig}
\usepackage{color}

\newtheorem{Th}{Theorem}[section]
\newtheorem{Def}[Th]{Definition}

\newtheorem{Cor}[Th]{Corollary}
\newtheorem{Le}[Th]{Lemma}
\newtheorem{St}[Th]{Proposition}

\theoremstyle{remark}
\newtheorem{Rem}[Th]{Remark}

\DeclareMathOperator{\clos}{clos}
\DeclareMathOperator{\Ker}{Ker}
\DeclareMathOperator{\supp}{supp}
\DeclareMathOperator{\HDR}{HDR}
\DeclareMathOperator{\Ann}{Ann}
\DeclareMathOperator{\Ext}{Ext}
\DeclareMathOperator{\id}{id}

\DeclareMathOperator{\loc}{loc}
\DeclareMathOperator{\SI}{SI}

\newcommand{\eps}{\varepsilon}
\newcommand{\scalprod}[2]{\langle{#1},{#2}\rangle}
\newcommand{\Scalprod}[2]{\Big\langle{#1},{#2}\Big\rangle}

\newcommand{\Si}{_{_\Sigma}\!}
\newcommand{\R}{\mathcal{R}}
\newcommand{\J}{\mathrm{J}}
\newcommand{\Rw}{\mathcal{R}_{\mathrm{w}}}
\newcommand{\normal}{\mathrm{n}}
\newcommand{\HDRL}{\HDR_{\mathrm{loc}}}
\newcommand{\B}{\boldsymbol{B}}
\newcommand{\K}{\boldsymbol{K}}
\newcommand{\Conv}{\mathcal{C}}
\newcommand{\CONV}{\boldsymbol{C}}
\newcommand{\T}{\mathcal{T}}

\begin{document}
\title{Restrictions of higher derivatives of the Fourier transform}
\author{Michael Goldberg \and Dmitriy Stolyarov}
\date{September 10, 2018}
\address{Mathematical Sciences Department, University of Cincinnati, Cincinnati, OH 45221-0025, USA}
\email{goldbeml@ucmail.uc.edu}
\address{Chebyshev Lab, St. Petersburg State Univeristy, 14th line 29b, Vasilyevsky Island, St. Petersburg 199178, Russia}
\address{St. Petersburg Department of Steklov Mathematical Institute, Fontanka 27, St. Petersburg 191023, Russia}
\email{d.m.stolyarov@spbu.ru}
\thanks{The authors would like to thank Tony Carbery for introducing them to Y. Domar's papers.}
\thanks{M.~G.~received support from Simons Foundation grant \#281057, D.~S.~received support from Russian Foundation for Basic Research grant \#17-01-00607.}

\begin{abstract}
We consider several problems related to the restriction of $(\nabla^k) \hat{f}$ to a surface $\Sigma \subset \mathbb R^d$ with  nonvanishing Gauss curvature.  While such restrictions clearly exist if $f$ is a Schwartz function, there are few bounds available that enable one to take limits with respect to the $L_p(\mathbb R^d)$ norm of $f$.  We establish three scenarios where it is possible to do so:
\begin{itemize}
\item When the restriction is measured according to a Sobolev space $H^{-s}(\Sigma)$ of negative index.  We determine the complete range of indices $(k, s, p)$ for which such a bound exists.
\item Among functions where $\hat{f}$ vanishes on $\Sigma$ to order $k-1$, the restriction of $(\nabla^k) \hat{f}$ defines a bounded operator from (this subspace of) $L_p(\mathbb R^d)$ to $L_2(\Sigma)$ provided $1 \leq p \leq \frac{2d+2}{d+3+4k}$.
\item When there is {\it a priori} control of $\hat{f}|_\Sigma$ in a space $H^{\ell}(\Sigma)$, $\ell > 0$, this implies improved regularity for the restrictions of $(\nabla^k)\hat{f}$.  If $\ell$ is large enough then even $\|\nabla \hat{f}\|_{L_2(\Sigma)}$ can be controlled in terms of $\|\hat{f}\|_{H^\ell(\Sigma)}$ and $\|f\|_{L_p(\mathbb R^d)}$ alone.
\end{itemize}
The techniques underlying these results are inspired by the spectral synthesis work of Y.\;Domar, which provides a mechanism for $L_p$ approximation by ``convolving along surfaces", and the Stein--Tomas restriction theorem.  Our main inequality is a bilinear form bound with similar structure to the Stein--Tomas $T^*T$ operator, generalized to accommodate smoothing along $\Sigma$ and derivatives transverse to it. It is used both to establish basic $H^{-s}(\Sigma)$ bounds for derivatives of $\hat{f}$ and to bootstrap from surface regularity of $\hat{f}$ to regularity of its higher derivatives.
\end{abstract}

\maketitle

\tableofcontents

\section{Introduction}\label{S1}

\subsection{Overview of the Derivative Restriction Problem}\label{s11}

Questions regarding the fine properties of the Fourier transform of a function in $L_p(\mathbb{R}^d)$ have long played a central role in the development of classical harmonic analysis.  While the Hausdorff--Young theorem guarantees that for $1 \leq p \leq 2$, the Fourier transform of $f \in L_p$ belongs to its dual space $L_{p/(p-1)}$, it does not provide guidance on whether $\hat{f}$ may be defined on a given measure-zero subset $\Sigma \subset \mathbb{R}^d$.  The canonical question of this type, originating in the work of Stein circa 1967, is to find the complete range of pairs~$(p,q)$ for which the inequality
\begin{equation}\label{Restriction_Conjecture}
\|\hat{f}|_{S^{d-1}}\|_{L_q(S^{d-1})} \lesssim \|f\|_{L_p(\mathbb{R}^d)}
\end{equation}
holds true.   The problem was solved in the case~$d=2$ in~\cite{Fefferman} and remains an active subject of research in higher dimensions (e.g.~\cite{BourgainGuth, Guth1, Guth2}).

In this paper we investigate the possibility of defining the surface trace 
of higher order gradients of the Fourier transform of an~$L_p$ function, with a focus on uniform estimates in the style of~\eqref{Restriction_Conjecture}.  
Let~$\Sigma$ be a closed smooth embedded~$(d-1)$-dimensional submanifold of~$\mathbb{R}^d$. Assume that the principal curvatures of~$\Sigma$ are non-zero at any point. Let~$K$ be a compact subset of~$\Sigma$ and~$k$ be a natural number. We consider as a model problem the inequality
\begin{equation}\label{First_restriction_for_higher_derivatives}
\big\|(\nabla^k\hat{f})\big|_{\Sigma}\big\|_{L_2(K)} \lesssim_K \|f\|_{L_p(\mathbb{R}^d)}.
\end{equation}
Here and in what follows the Fourier transform has priority over differentiation: we first compute the Fourier transform and then differentiate it. We choose the standard Hausdorff measure $d\sigma$ on~$\Sigma$ to define the~$L_2$-space on the left hand-side. The notation~``$\lesssim_K$'' signifies that the constant in the inequality may depend on the choice of~$K$, but should not depend on~$f$. We restrict our study to the case of~$L_2$ instead of~$L_q$ with arbitrary~$q$ on the left hand-side, because the Hilbert space properties of $L_2$ make this case more tractable. In fact, the range of all possible~$p$ in~\eqref{Restriction_Conjecture} when~$q=2$ is described by the classical Stein--Tomas Theorem (established in~\cite{Tomas} and~\cite{Stein}).

Unfortunately, inequality~\eqref{First_restriction_for_higher_derivatives} cannot hold true unless $k=0$. To see that, consider the  shifts of a function~$f$, in other words~$f_N(x) = f(x+Ny)$, where~$y\ne 0$ is a fixed point in~$\mathbb{R}^d$. If we plug~$f_N$ into~\eqref{First_restriction_for_higher_derivatives} instead of~$f$, the norm on the left hand-side will be of the order~$N^k$, whereas the quantity on right will not depend on~$N$. 

The next question along these lines is: what modifications can be made so that~\eqref{First_restriction_for_higher_derivatives} becomes a true statement for $k \geq 1$?
Since the original inequality~\eqref{Restriction_Conjecture} is shift-invariant, we seek translation invariant conditions for $f$. This rules out natural candidates such as requiring $(1+|x|)^k f \in L_p$.

One possibility is to relax the desired local regularity from $L_2(K)$ to a Sobolev space of negative order.  Consider the inequality
\begin{equation}\label{RestrictionWithSobolev}
\big\|\phi (\nabla^k\hat{f})\big|_{\Sigma}\big\|_{H^{-s}(\Sigma)} \lesssim_\phi \|f\|_{L_p(\mathbb{R}^d)}.
\end{equation}
Here~$\phi \in C_0^{\infty}(\Sigma)$ is an arbitrary compactly supported smooth function (the constant in the inequality may depend on it). The parameter~$s$ is a non-negative real, and~$H^{-s}$ is the~$L_2$-based Bessel potential space.
Whenever~\eqref{RestrictionWithSobolev} holds, there is a trace value for $\nabla^k \hat{f}$ in $H^{-s}_{\loc}(\Sigma)$ for all $f \in L_p(\mathbb {R}^d)$.

One might guess that the inequality~\eqref{RestrictionWithSobolev} gets weaker as we increase~$s$, opening the way to define the trace of~$\nabla^k \hat{f}$ on~$\Sigma$ with an increasingly large range of~$p$. This is indeed the case. The case~$k=0$ in~\eqref{RestrictionWithSobolev} was considered by Cho, Guo, and Lee in~\cite{ChoGuoLee}.  They observed Sobolev-space trace values of $\hat{f}$ for $f \in L_p$ with $p$ going up to the sharp exponent dictated by the Fourier transform of a surface measure.  
As we will see below, the general case $k \geq 0 $ of~\eqref{RestrictionWithSobolev} requires only one more large-$k$ endpoint estimate and a routine interpolation argument.

There are two parameters that appear frequently as bounds in our arguments:
\begin{align}
\sigma_p &= \frac{d}{p} - \frac{d+1}{2};\label{Sigma_p}\\
\kappa_p &= \frac{d+1}{p} - \frac{d+3}{2}.\label{Kappa_p}
\end{align}
Where it occurs later on, we also use the standard notation $p' = \frac{p}{p-1}$ for the dual exponent to $L_p$.
\begin{Th}[Corollary of Theorem $1.1$ in~\cite{ChoGuoLee}]\label{CorollaryOfChoGuoLee}
Let~$p > 1$. The inequality~\eqref{RestrictionWithSobolev} is true if and only if
\begin{align}
k &\leq s;\label{RestrictionWithSobolevShift}\\
k &< \sigma_p;\label{RestrictionWithSobolevSurface}\\ 
2k - s &\leq \kappa_p\label{RestrictionWithSobolevKnapp}.
\end{align}
For fixed $k$ and $p$ with $k < \sigma_p$, that means $s \geq \max(k, 2k-\kappa_p)$.  In the case~$p=1$, the case~$k=\sigma_1$ is also permitted if~$s > k$.
\end{Th}
The parameter~$\sigma_p$  is related to the ``surface measure extremizer.'' When condition~\eqref{RestrictionWithSobolevSurface} does not hold, Theorem~\ref{CorollaryOfChoGuoLee} fails by testing its dual statement against a surface measure on $\Sigma$. The parameter~$\kappa_p$, and its role in condition~\eqref{RestrictionWithSobolevKnapp} are similarly associated with Knapp examples.

In odd dimensions there is an endpoint case $p=1$, $k= \sigma_1 = \frac{d-1}{2} \in \mathbb{N}$ where inequality~\eqref{RestrictionWithSobolev} is true for $s > \frac{d-1}{2}$. This is stated more precisely in Corollary~\ref{EndpointCaseWithoutWeight} below. The proof of that bound is more direct than most of our other arguments (in fact it is nearly equivalent to the dispersive bound for the Schr\"odinger equation) and it is completely independent.  

The paper contains two proofs of Theorem~\ref{CorollaryOfChoGuoLee}.  First, it is a special case of Theorem~\ref{StubbornTheorem}, whose proof is presented as Section~\ref{S3}.  Then we show in Subsection~\ref{s63} how to derive Theorem~\ref{CorollaryOfChoGuoLee} from the results of~\cite{ChoGuoLee}. To be more specific, one can interpolate between the results of~\cite{ChoGuoLee} for $k=0$ and the Besov-space bound in Proposition~\ref{WeightedBesov} for $p=1$, $k=\frac{d-1}{2}$ to obtain the full range of Theorem~\ref{CorollaryOfChoGuoLee}. 

If one is determined not to weaken the $L_2(K)$ norm in~\eqref{First_restriction_for_higher_derivatives}, it is necessary to consider $f$ belonging to an {\em a priori} narrower space than $L_p(\mathbb{R}^d)$.   We introduce the main character.

\begin{Def}\label{LSigma}
Let~$\Sigma$ be a closed smooth embedded~$(d-1)$-dimensional submanifold of~$\mathbb{R}^d$,~$p \in [1,\infty)$, and~$k \in \mathbb{N}$. Define the space~$\Si L_p^k$ by the formula
\begin{equation*}
\Si L_p^k = \clos_{L_p}\Big(\Big\{f \in \mathcal{S}(\mathbb{R}^d)\,\Big|\; \forall l =0,1,2,\ldots,k-1  \quad \nabla^l \hat{f} = 0 \quad \hbox{\textup{on}} \quad \Sigma\Big\}\Big).
\end{equation*}
Define~$\Si L_p^0$ to be simply~$L_p(\mathbb{R}^d)$. The first non-trivial space~$\Si L_p^1$ will often be denoted by~$\Si L_p$.
\end{Def}
The symbol~$\mathcal{S}$ denotes the Schwartz class of test functions. We note that in the definition above, we do not need any information about~$\Sigma$. In fact,~$\Sigma$ may be an arbitrary closed set. The restriction~$p < \infty$ is taken so that the Schwartz class is dense in $L_p$, though one could replace closure with weak closure in the case~$p=\infty$ if needed. 
These generalities will not arise in the present paper. \emph{From now on we assume that~$\Sigma$ is a closed smooth embedded~$(d-1)$-dimensional submanifold of~$\mathbb{R}^d$ with non-vanishing principal curvatures}.

It will turn out (See Theorem~\ref{CoincidenceTheorem} below) that for a certain range of $p$ and $k$, the space $\Si L_p^k$ contains precisely the functions $f \in L_p$ whose Fourier transform vanishes on $\Sigma$ to order $k-1$.  We take advantage of the additional structure of the domain to formulate a second adaptation of inequality~\eqref{First_restriction_for_higher_derivatives}, this time with the trace of $\nabla^k\hat{f}$ still belonging to $L^2_{\loc}(\Sigma)$.
\begin{equation}\label{Restriction_for_higher_derivatives}
\big\| (\nabla^k\hat{f})\big|_{\Sigma}\big\|_{L_2(K)} \lesssim_K \|f\|_{L_p} \text{ for all } f \in \!\,\Si L_p^k.
\end{equation}

One might expect that a similar statement with the $L_2$ norm replaced by a weaker Sobolev norm will admit a larger range of $p$, that is:
\begin{equation}\label{Sobolev_Restriction_for_higher_derivatives}
\big\|(\phi \nabla^k\hat{f}) \big|_{\Sigma}\big\|_{H^{-s}(\Sigma)} \lesssim_\phi \|f\|_{L_p(\mathbb{R}^d)} \text{ for all } f \in \!\,\Si L_p^k.
\end{equation}

However at this point in the discussion it is not clear why~\eqref{Sobolev_Restriction_for_higher_derivatives} should be true outside the range established in Theorem~\ref{CorollaryOfChoGuoLee}, or why~\eqref{Restriction_for_higher_derivatives} should be true at all.  Given a generic function $f \in L_p(\mathbb R^d)$, its Fourier transform $\hat{f}$ is not differentiable even to fractional order.  We have reduced the obstruction somewhat by seeking derivatives of $\hat{f}$ only at the points $\xi \in \Sigma$, and by specifying a substantial number of its partial derivatives
via the assumption $f \in \!\, \Si L_p^k$.  Never the less, values of $\hat{f}|_\Sigma$ alone do not uniquely determine $f \in L^p$, nor are they known to shed much light on the behavior of $\hat{f}$ in a neighborhood of $\Sigma$.

 Theorem~\ref{First_restriction_for_higher_derivatives_Theorem} below finds the complete range of $p$ for which an $L_2$ gradient restriction~\eqref{Restriction_for_higher_derivatives} is true.  In particular the range is nonempty when $d \geq 4k+1$.  This result follows a clear pattern from the Stein--Tomas restriction theorem, which is the $k=0$ case.  The range of $p$ permitted in~\eqref{Sobolev_Restriction_for_higher_derivatives} is also sharp in the same way as Theorem~\ref{CorollaryOfChoGuoLee} and the results in~\cite{ChoGuoLee}.  The range of $s$ we obtain here is much larger than what is true in the context of Theorem~\ref{CorollaryOfChoGuoLee}, but most likely not optimal due to some complications with linear programming over the integers.

The $k=1$ case of Theorem~\ref{First_restriction_for_higher_derivatives_Theorem} shows that an {\em a priori} assumption $\hat{f}|_\Sigma = 0$ leads to nontrivial bounds on $\nabla \hat{f} |_\Sigma$.  In fact there is a larger family of bounds for trace values of $\nabla^k \hat{f}$, and one can begin the bootstrapping process with a much milder assumption $\hat{f}|_\Sigma \in H^\ell(\Sigma)$ instead of requiring it to vanish.  We explore these generalizations in Proposition~\ref{SufficientConditionsHDR}, Theorem~\ref{HDRExample}, and the related discussion.  The inequality which takes the place of~\eqref{Sobolev_Restriction_for_higher_derivatives} has the form
\begin{equation}
\big\|(\phi \nabla^k \hat{f})\big|_{\Sigma}\big\|_{H^{-s}(\Sigma)} \lesssim_{\phi}\Big(\|f\|_{L_p(\mathbb{R}^d)} + \|\phi\hat{f}\|_{H^\ell(\Sigma)}\Big) \text{ for all } f \in L_p
\end{equation}
(the right hand-side may be infinite).
Remarkably, there are cases where this statement holds with only an $L_2(\Sigma)$ norm on the left side. In Corollary~\ref{HDRExample_s=0} we find a sizable range of indices $(d, p, \ell)$ that admit a local-$L_2$ bound on the gradient of $\hat{f}$,
\begin{equation}
\big\| \phi \nabla \hat{f} \big\|_{L_2(\Sigma)} \lesssim_\phi \Big(\|f\|_{L_p(\mathbb{R}^d)} + \| \phi \hat{f}\|_{H^\ell(\Sigma)}\Big) \text{ for all } f \in L_p.
\end{equation}

The spaces~$\Si L_p^k$ that arise in Definition~\ref{LSigma} are not a new construction. They appeared in~\cite{Goldberg_Schlag} (see Proposition $12$ in that paper) and~\cite{Goldberg} where the authors investigated the action of Bochner--Riesz operators of negative order on these spaces. They arose in~\cite{Stolyarov} in connection with Sobolev type embedding theorems. We describe this development in Subsection~\ref{s14}. 

In fact, the spaces~$\Si L_1^k$ played the central role in the study of the spectral synthesis problem in 60s and 70s. We stress the work of Domar here (e.g.~\cite{Domar}) and will rely upon it in Section~\ref{S2}.

It is worth noting that the main inequality used to derive~\eqref{Restriction_for_higher_derivatives} and~\eqref{Sobolev_Restriction_for_higher_derivatives} is valid for all functions in $L_p$, not just those whose Fourier restriction vanishes on $\Sigma$.  Essentially it is the Stein--Tomas bilinear $T^*T$ bound modified by a smoothing operator within the surface~$\Sigma$ and partial derivatives transverse to it.  The formulation of this inequality, which may be of independent interest, is given in~\eqref{StubbornInequality} below and the sharp range of $p$ for which it holds is found in Theorem~\ref{StubbornTheorem}.

\subsection{Statement of results}\label{s12}
It follows from Definition~\ref{LSigma} that the spaces~$\Si L_p^k$ get more narrow as we increase~$k$: 
\begin{equation*}
L_p \supseteq\,  \Si L_p^1 \supseteq\,  \Si L_p^2 \supseteq \ldots \supseteq\, \Si L_p^k \supseteq \ldots \supseteq\, \Si L_p^{\infty}.
\end{equation*}
The final space can be defined as the closure in~$L_p$ of the set of Schwartz functions whose Fourier transform vanishes in a neighborhood of~$\Sigma$. We claim that~$\Si L_p^k =\, \Si L_p^{\infty}$ when~$k$ is sufficiently large (i.e. the chain of spaces stabilizes). Here is the precise formulation.
\begin{Th}\label{Stability}
We have~$\Si L_p^k =\, \Si L_p^{k+1} = \,\!\Si L_p^{\infty}$ provided~$k \geq \sigma_p = \frac{d}{p} - \frac{d+1}{2}$ and~$p > 1$. If $p=1$, this is true provided $k > \frac{d-1}{2} = \sigma_1$.
\end{Th}
For the case~$p=1$, this theorem was proved in~\cite{Domar}, and the proof works for arbitrary~$p$ (except for, possibly,~$p=\infty$, which we do not consider here). The theorem is sharp in the sense that~$\Si L_p^k \ne\, \Si L_p^{k+1}$ provided~$k < \sigma_p$. 

\begin{Th}\label{First_restriction_for_higher_derivatives_Theorem}
The inequality~\eqref{Restriction_for_higher_derivatives} is true if and only if~$p \in [1, \frac{2d+2}{d+3+4k}]$, or equivalently $2k \leq \kappa_p$.

More generally, inequality~\eqref{Sobolev_Restriction_for_higher_derivatives} is true for $p \in [1, \frac{2d}{d+1+2k})$ and
$s \geq \max(0, k +1- \lceil \sigma_p-k\rceil, 2k-\kappa_p)$, where the notation $\lceil \cdot \rceil$ indicates the smallest integer greater than or equal to the enclosed value.  This covers the entire range $k < \sigma_p$.  When $p=1$ and $\sigma_1 = \kappa_1 = \frac{d-1}{2} \in \mathbb{N}$ the value $s = \max(0, 2k - \frac{d-1}{2})$ is also permitted.
\end{Th}

\begin{Rem}
 The $p=1$, $k = \sigma_1 = \frac{d-1}{2} \in \mathbb{N}$ endpoint case is handled in Corollary~\ref{EndpointCaseWithoutWeight} below, with inequality~\eqref{Sobolev_Restriction_for_higher_derivatives} holding for all $s > k$.
\end{Rem}
The first claim in the theorem above is an ``iff'' statement. Usually, the ``if'' part is much more involved than the ``only if'' one. In fact, the ``only if'' part of Theorem~\ref{First_restriction_for_higher_derivatives_Theorem} is proved with the standard Knapp example. Some of other theorems in the paper will have richer collection of ``extremizers''. Moreover, one and the same ``extremizer'' may prove sharpness of several related estimates. We collect the descriptions of such type ``extremizers'' (and thus, the proofs of the ``only if'' parts) in Section~\ref{S5}.

Theorem~\ref{First_restriction_for_higher_derivatives_Theorem} says that the operator
\begin{equation*}
\R^k_K\colon f\mapsto (\nabla^k f)\big|_{\Sigma}
\end{equation*}
acts continuously from the space~$\Si L_p^k$ to~$L_2(K)$ when~$p \in [1,\frac{2d+2}{d+3+4k}]$, or from $\Si L_p^k$ to $H^{-s}(K)$ for some combinations of $(p,s)$ with $p \in [1, \frac{2d}{d+1+2k})$.   This allows us to define a new space
\begin{equation}\label{IntersectionOfKernels}
\Ker \R^{k} = \bigcap_{K \subset \Sigma} \Ker \R^{k}_K,
\end{equation}
which consists of all~$L_p$ functions for which the~($L_2$ or $ H^{-s}$) traces of all partial derivatives of order~$k$ vanish on~$\Sigma$. Note that~$\R^{k-1}$ is formally defined on~$\Si L^{k-1}_p \supseteq \!\,\Si L^{k}_p$, and so on, thus we have vanishing of lower order derivatives as well. We also note that in the case when~$\Sigma$ is compact, one does not need to use the intersection in~\eqref{IntersectionOfKernels} and may simply write~$\Ker \R^{k} = \Ker \R^{k}_{\Sigma}$. It follows from definitions that~$\Si L_p^{k+1} \subset \Ker \R^{k}$. In fact, the two spaces must coincide. This looks like a trivial approximation statement, however we do not know a straightforward proof. 
\begin{Th}\label{CoincidenceTheorem}
For any~$p \in [1,\frac{2d+2}{d+3+4k}]$, the spaces~$\Si L_p^{k+1}$ and~$\Ker \R^{k}$ coincide with $\R^k$ being regarded as a map from $\Si L_p^k$ to $L^2_{\loc}(\Sigma)$. This occurs when $2k \leq \kappa_p$.

 For any~$p \in [1, \frac{2d}{d+1+2k})$, the spaces~$\Si L_p^{k+1}$ and~$\Ker \R^{k}$ coincide with $\R^k$ being regarded as a map from $\Si L_p^k$ to $H^{-s}_{\loc}(\Sigma)$ for the same range of $s$ as in Theorem~\ref{First_restriction_for_higher_derivatives_Theorem}.  This occurs when $k < \sigma_p$, or $k \leq \sigma_1$ when $p=1$.

Since $\R^k$ acts nontrivially on the Schwartz functions contained in $\Si L_p^k$, it follows that $\Si L_p^k \supsetneq \!\, \Si L_p^{k+1}$ in this range of~$k$.
\end{Th}
\begin{Rem}
Theorems~\ref{Stability} and~\ref{CoincidenceTheorem} completely classify the spaces $\Si L_p^k$, modulo some details about the optimal target space for $\R^k$.  
\end{Rem}
\begin{Rem}
In the papers~\cite{Goldberg} and~\cite{Goldberg_Schlag}, the condition ``$\hat{f} = 0$ on the unit sphere'' was understood in the sense of~$L_2$ traces.
\end{Rem}

Theorem~\ref{First_restriction_for_higher_derivatives_Theorem} covers many combinations $k > s \geq 0$ that are forbidden  in Theorem~\ref{CorollaryOfChoGuoLee} by demanding that $\hat{f}$ vanishes to order $k-1$ on $\Sigma$.  We now introduce a family of statements which assume only smoothness of $\hat{f}|_\Sigma$ instead of vanishing.  Bessel spaces already appear on the left hand-side of inequality~\eqref{RestrictionWithSobolev}, so it is reasonable to use the same scale to describe the smoothness of $\hat{f}|_\Sigma$.

In effect this is a bootstrapping claim, that regularity of $\hat{f}$ in the $d-1$ directions tangent to $\Sigma$ implies a certain regularity in the transverse direction as well.  It is notable that the inequalities hold even though $f \in L_p$ (and hence $\hat{f}$) is not uniquely determined by $\hat{f}|_\Sigma$.

\begin{Def}\label{HDR}
Let~$k$ be a natural number, let~$\ell$ and~$s$ be non-negative reals, let~$p \in [1,\infty)$. We say that the higher derivative restriction property~$\HDR(\Sigma,k,s,\ell,p)$ holds true if for any smooth compactly supported function~$\phi$ in~$d$ variables, the estimate
\begin{equation}\label{HDRinequality}
\big\|(\phi \nabla^k \hat{f})\big|_{\Sigma}\big\|_{H^{-s}(\Sigma)} \lesssim_{\phi}\Big(\|f\|_{L_p(\mathbb{R}^d)} + \|\phi\hat{f}\|_{H^\ell(\Sigma)}\Big)
\end{equation}
holds true for any Schwartz function~$f$.  
\end{Def}

\begin{Rem}
A complete generalization of Theorem~\ref{First_restriction_for_higher_derivatives_Theorem} would include {\em a priori} estimates on $\|\phi \nabla^j \hat{f}\|_{H^{\ell_j}(\Sigma)}$  for $j = 0, 1, \ldots , J \leq k-1$.  We consider only the $J=0$ case above for relative simplicity of notation.
\end{Rem}

\begin{St}\label{SufficientConditionsHDR}
If~$\HDR(\Sigma,k,s,\ell,p)$ holds true and~$p > 1$, then
\begin{align}
\label{HDRShiftCondition} k &\leq s + \ell;\\
\label{HDRSigmaLpShiftCondition} k&\leq s + 1;\\
\label{HDRSurfaceCondition} k &< \sigma_p;\\
\label{HDRShiftedKnapp} \frac{k\ell}{s + \ell - k} &\leq \kappa_p\quad \hbox{when } k> s;\\
\label{HDRKnapp} 2k -s &\leq \kappa_p,
\end{align}
where the numbers~$\sigma_p$ and~$\kappa_p$ are defined by~\eqref{Sigma_p} and~\eqref{Kappa_p} respectively. In the case~$p=1$, equality in~\eqref{HDRSurfaceCondition} may also occur.
\end{St}
The sufficient conditions we are able to provide for the~$\HDR$ inequalities do not always coincide with the necessary ones listed above. In fact, they get close to necessary conditions when~$\ell$ is relatively small and there is a gap if~$\ell$ is large. By ``getting close for small $\ell$'' we mean that the non-sharpness comes only from our inability to work with non-integer~$k$. The sufficient conditions we are able to obtain are rather bulky (this is again due to ``integer arithmetic''). They are formulated in terms of certain convex hulls of finite collections of points in the plane. Since we need to introduce more notation before formulating the strongest available statement, we refer the reader to Theorem~\ref{BigTableTheorem} in Section~\ref{S4} for the details and state a representative subset of the results here.
\begin{Th}\label{HDRExample}
Let~$p > 1$ and~$\kappa_p \in\mathbb{N}$. If
\begin{equation}\label{UglyCondition}
2\Big\lceil\frac{\ell-1}{\ell}\kappa_p\Big\rceil \leq \kappa_p,
\end{equation}
then~$\HDR(\Sigma,k,s,\ell,p)$ holds true provided \eqref{HDRShiftCondition},~\eqref{HDRSigmaLpShiftCondition},~\eqref{HDRSurfaceCondition},~\eqref{HDRShiftedKnapp}, and~\eqref{HDRKnapp} are satisfied.  If~\eqref{UglyCondition} does not hold, then~$\HDR(\Sigma,k,s,\ell,p)$ holds true provided \eqref{HDRShiftCondition} -- \eqref{HDRKnapp} are satisfied as well as the inequality  
\begin{equation} \label{UglyCondition2}
s \geq k - \frac{\kappa_p-k}{\kappa_p - \big[\frac{\kappa_p}{2}\big]}.
\end{equation}
\end{Th}
Here and in what follows,~$\lceil \cdot\rceil$ is the upper integer part of a number, i.e. the smallest integer that is greater or equal to the number; the notation~$[\cdot]$ denotes the lower integer part of a number, i.e. the largest integer that does not exceed the number:
\begin{equation}\label{IntegerPart}
[x] = \sup\{z \in \mathbb{Z}\mid z\leq x\};\quad \lceil x\rceil = \inf \{z\in\mathbb{Z}\mid z \geq x\}.
\end{equation}

The $s=0$, $k=1$ cases of Theorem~\ref{HDRExample} illustrate its ability to extract derivatives of $\hat{f}$ in all directions when  only regularity along $\Sigma$ is assumed.
\begin{Cor} \label{HDRExample_s=0}
Suppose $p = \frac{2d+2}{d+3+2m}$ for some integer $2 \leq m < \frac{d-1}{2}$, and let $\ell \geq \frac{2d+2 - p(d+3)}{2d+2 - p(d+5)} = \frac{m}{m-1}$.  Then
\begin{equation} \label{HDRInequality_s=0}
\big\|  \nabla \hat{f} \big\|_{L_2(K)} \lesssim_{K, \phi} \Big(\|f\|_{L_p(\mathbb{R}^d)} + \| \phi \hat{f}\|_{H^\ell(\Sigma)}\Big).
\end{equation}
for any Schwartz function $f$, compact subset $K \subset \Sigma$, and smooth cutoff $\phi$ that is identically 1 on $K$. 
\end{Cor}
In Section~\ref{s54} we construct a translated Knapp example to show that the lower bound for $\ell$ is sharp.

The property $\HDR(\Sigma, k, s, \ell, p)$ has a dual formulation in terms of the Fourier extension operator. We denote the Lebesgue measure on~$\Sigma$ by~$d \sigma$.
\begin{Cor} \label{HDRDual}
Suppose $\HDR(\Sigma, k, s, \ell, p)$ holds true and $p < \frac{2d}{d+1}$ (e.g. if the conditions of Theorem~\ref{HDRExample} or Theorem~\ref{BigTableTheorem} are satisfied).  Then for each $g \in H^s(\Sigma)$, multi-index $\alpha$ with $|\alpha| \leq k$, and smooth compactly supported $\phi$, there exist $F_\alpha \in L_{p'}(\mathbb{R}^d)$ and $g_\alpha \in H^{-\ell}(\Sigma)$ such that
\begin{equation}\label{DualEquation}
F_\alpha + (\phi g_\alpha\, d\sigma)\check{\phantom{i}} = x^\alpha(\phi g\, d\sigma)\check{\phantom{i}},
\end{equation}
and furthermore
\begin{equation}\label{DualEstimate}
\|F_\alpha\|_{L_{p'}(\mathbb{R}^d)} + \|g_\alpha\|_{H^{-\ell}(\Sigma)} \lesssim_\phi \|g\|_{H^s(\Sigma)}.
\end{equation}
Conversly, if for any compactly supported smooth function~$\phi$, for any~$g$, and for any~$\alpha$ there exist~$F_\alpha$ and~$g_\alpha$ such that~\eqref{DualEquation} and~\eqref{DualEstimate}, then $\HDR(\Sigma, k, s, \ell, p)$ holds true (we still assume~$p < \frac{2d}{d+1}$).
\end{Cor}

Finally, we present the main analytic tool used in our proofs of $\HDR$ inequalities. We will formulate it in local form: now~$\Sigma$ is a graph of a function on~$\mathbb{R}^{d-1}$ rather than an arbitrary submanifold. 

Let~$U$ be a neighborhood of the origin in~$\mathbb{R}^{d-1}$. Let~$h$ be a~$C^{\infty}$-smooth function on~$U$ such that~$h(0) = 0$ and~$\nabla h(0) = 0$. We also assume that the Hessian of~$h$ at zero does not vanish,
\begin{equation*}
\det \frac{\partial^2 h}{\partial \zeta^2}(0) \ne 0.
\end{equation*}
Moreover, we assume that the gradient of~$h$ is sufficiently close to zero and the second differential is sufficiently close to~$\frac{\partial^2 h}{\partial \zeta^2}(0)$:
\begin{equation}\label{HessianSmoothness}
\forall \zeta \in U\quad \bigg\|\frac{\partial h}{\partial \zeta}(\zeta)\bigg\| \leq \frac{1}{10d},\quad \bigg\|\frac{\partial^2 h}{\partial \zeta^2}(\zeta) - \frac{\partial^2h}{\partial \zeta^2}(0)\bigg\| \leq \frac{1}{10}\Big|\det \frac{\partial^2 h}{\partial \zeta^2}(0)\Big|.
\end{equation}
The function~$h$ naturally defines the family of surfaces
\begin{equation*}
\Sigma_r = \big\{(\zeta,h(\zeta)+r)\,\big|\; \zeta \in U\big\},\qquad r \in (-\infty,\infty).
\end{equation*}  
We also take some small number~$\eps > 0$ and consider the set~$V = U \times (-\eps,\eps)$.

We will be using Bessel potential spaces adjusted to these surfaces. Now we will need the precise quantity defining the Bessel norm. It is convenient to parametrize everything with~$U$. For~$\gamma \in \mathbb{R}$ and a compactly supported function~$\phi$ on~$\Sigma_r$ (for some fixed~$r$), define its~$H^{-\gamma}$-norm by the formula
\begin{equation}\label{SobolevDef}
\|\phi\|_{H^{-\gamma}(\Sigma_r)}^2 = \int\limits_{\mathbb{R}^{d-1}}\Big|\mathcal{F}_{\zeta\to z}\big[\phi(\zeta,h_r(\zeta))\big](z)\Big|^2\;\big(1+|z|\big)^{-2\gamma}\,dz.
\end{equation}
The symbol~$\mathcal{F}$ denotes the Fourier transform in~$(d-1)$ variables, and we have used the notation~$h_r(\zeta) = h(\zeta) + r$. We will also use the homogeneous norm
\begin{equation*}\label{HomogeneousSobolevDef}
\|\phi\|_{\dot{H}^{-\gamma}(\Sigma_r)}^2 = \int\limits_{\mathbb{R}^{d-1}}\Big|\mathcal{F}_{\zeta\to z}\big[\phi(\zeta,h_r(\zeta))\big](z)\Big|^2|z|^{-2\gamma}\,dz, \quad \gamma\in \Big(0,\frac{d-1}{2}\Big).
\end{equation*}
Since all our functions are supported on~$U$, this norm is equivalent to the inhomogeneous norm~\eqref{SobolevDef} when~$\gamma \in (0,\frac{d-1}{2})$. 
We will often use another formula for the homogeneous norm:
\begin{equation}\label{OurSobolev}
\|\phi\|_{\dot{H}^{-\gamma}(\Sigma_r)}^2 = C_{d,\gamma}\iint\limits_{U\times U}\phi(\zeta,h_r(\zeta))\overline{\phi(\eta,h_r(\eta))}|\zeta - \eta|^{2\gamma-d+1}\,d\zeta\,d\eta, \quad \gamma\in \Big(0,\frac{d-1}{2}\Big).
\end{equation} 
The constant~$C_{d,\gamma}$ may be computed explicitly, however, we do not need the sharp expression for it.

Let~$\alpha$ and~$\beta$ be integers between~$0$ and~$\frac{d-1}{2}$, let~$\gamma \in [0,\frac{d-1}{2})$ be real, let~$p \in [1,\infty]$. Let also~$\psi$ be an arbitrary~$C_0^{\infty}$ function supported in~$U$. We are interested in the ``surface inequality"
\begin{equation}\label{StubbornInequality}
\bigg|\Big(\frac{\partial}{\partial r}\Big)^{\beta}\Big\|\frac{\partial^{\alpha} \hat{f}}{\partial \xi_d^{\alpha}}\psi\Big\|^2_{\dot{H}^{-\gamma}(\Sigma_r)}\bigg|_{r=0}\bigg| \lesssim \|f\|_{L_p(\mathbb{R}^d)}^2.
\end{equation}
So we compute the Fourier transform of an~$L_p$ function, calculate its derivative with respect to~$d$-th coordinate, compute the~$\dot{H}^{-\gamma}$-norms of traces of this derivative on the surfaces~$\Sigma_r$, and then differentiate~$\beta$ times with respect to~$r$. We use the variable~$\xi$ for points in~$\mathbb{R}^d$ on the spectral side and~$\zeta$ for points in~$\mathbb{R}^{d-1}$ decoding points  on~$\Sigma_r$ (for example,~$\xi$ is quite often equal to~$(\zeta,h(\zeta))$).
\begin{Def}
Let~$h$ satisfy the assumptions imposed on it above. We say that the statement~$\SI(h,\alpha,\beta,\gamma,p)$ holds true if~\eqref{StubbornInequality} is true.
\end{Def}
\begin{Th}\label{StubbornTheorem}
Let~$\gamma\in [0,\frac{d-1}{2})$. The statement~$\SI(h,\alpha,\beta,\gamma,p)$ is true if and only if
\begin{align}
\label{SISpectralShift} \alpha &\leq \gamma,\\
\label{SIGaussian} \alpha + \beta &\leq \sigma_p,\\
\label{SIKnapp} 2\alpha - \gamma +\beta &\leq \kappa_p,
\end{align}
and the inequality~\eqref{SIGaussian} is strict if~$p > 1$.
\end{Th}
Note that Theorem~\ref{CorollaryOfChoGuoLee}, except the endpoint case~$p=1$,~$k=\frac{d-1}{2}$, follows from Theorem~\ref{StubbornTheorem} (pick~$\beta = 0$) modulo a localization argument (see Subsection~\ref{s61} below). We also show in Subsection~\ref{s63} that  Theorem~\ref{CorollaryOfChoGuoLee} has a direct proof by interpolating between~\cite{ChoGuoLee} and the $p=1$,~$k=\frac{d-1}{2}$ endpoint.  It is not clear to the authors whether one can derive the full statement of Theorem~\ref{StubbornTheorem} from the results of~\cite{ChoGuoLee} or from Theorem~\ref{CorollaryOfChoGuoLee} (which correspond to the cases~$\alpha = \beta = 0$ and $\beta=0$ respectively). Our proof seems to use a completely different strategy than the one in~\cite{ChoGuoLee}. Our method also allows us to work with Strichartz estimates, i.e. consider the larger scale of mixed-norm Lebesgue spaces on the right hand-side of~\eqref{StubbornInequality}.

At last, we want to emphasize that the passage from Theorem~\ref{StubbornTheorem} to~$\HDR$ inequalities is not immediate. The reader is invited to read a preview of this argument in the next section, where we give an overview of the paper.

\subsection{Overview of proofs}\label{s13}
We will now sketch the proof Theorem~\ref{First_restriction_for_higher_derivatives_Theorem} in the simplest non-trivial case~$k=1$. It will show the main ideas behind the proofs of Theorems~\ref{First_restriction_for_higher_derivatives_Theorem} and~\ref{StubbornTheorem}, and also give some hints to the proof of Theorem~\ref{HDRExample}.

\begin{proof}[Sketch of the proof of Theorem~\ref{First_restriction_for_higher_derivatives_Theorem},~$k=1$,~$s=0$.]
We start with localization. We use the notation introduced before Theorem~\ref{StubbornTheorem} (the sets~$U,V$, the functions~$h, \psi$, etc.). 

\begin{St}\label{StubbornBaby}
The inequality
\begin{equation*}
\int\limits_{U}\Big|\frac{\partial \hat{f}}{\partial \xi_d}(\zeta,h(\zeta))\psi(\zeta)\Big|^2 \,d\zeta \lesssim_{\psi} \|f\|_{\Si L_p}^2
\end{equation*}
holds true provided~$p\in [1,\frac{2d+2}{d+7}]$ and~$\{(\zeta,h(\zeta))\mid \zeta \in U\}\subset \Sigma$.
\end{St}
One can reduce the case~$k=1$,~$s=0$ in Theorem~\ref{First_restriction_for_higher_derivatives_Theorem} to Proposition~\ref{StubbornBaby} via a standard partition of unity argument (a small part of~$\Sigma$ might be represented as the graph~$\{(\zeta,h(\zeta))\mid \zeta \in U\}$). There is a technicality that the full gradient may be replaced with the partial derivative in the direction of the normal at zero. We will present the argument in Subsection~\ref{s61} below.

\paragraph{\it Proof of Proposition~\ref{StubbornBaby}}
We may assume that~$f$ is Schwartz by definition of the space~$\Si L_p^k$. We use the condition~$\hat{f}(\zeta,h(\zeta)) = 0$ to express the integrand in another form:
\begin{equation*}
\Big|\frac{\partial \hat{f}}{\partial \xi_d}(\zeta,h(\zeta))\Big|^2 = \frac12\bigg( \frac{\partial^2}{\partial r^2} \Big|\hat{f}(\zeta,h(\zeta) + r)\Big|^2\bigg)\bigg|_{r=0}
\end{equation*}
(simply apply the product rule to the right hand-side and note that all but one summands are zeros).
Now, it suffices to prove the inequality
\begin{equation*}
\bigg|\int\limits_{U}\bigg(\frac{\partial^2}{\partial r^2} \Big|\hat{f}(\zeta,h(\zeta) + r)\Big|^2\bigg)\bigg|_{r=0}\psi(\zeta)\,d\zeta\bigg|  \lesssim_{\psi} \|f\|_{L_p}^2.
\end{equation*}
It is convenient to bilinearize it:
\begin{equation}\label{BilinearizedBaby}
\bigg|\int\limits_{U}\bigg(\frac{\partial^2}{\partial r^2} \Big[\hat{f}(\zeta,h(\zeta) + r)\hat{\bar{g}}(\zeta,h(\zeta) + r)\Big]\bigg)\bigg|_{r=0}\psi(\zeta)\,d\zeta\bigg|  \lesssim_{\psi} \|f\|_{L_p}\|g\|_{L_p}.
\end{equation}
We consider the complex measure~$\sigma$ on~$\mathbb{R}^d$ supported on~$\Sigma$:
\begin{equation*}
\int F d\sigma = \int\limits_{U} F(\zeta, h(\zeta))\psi(\zeta)\,d\zeta \quad \hbox{for any} \quad F\in C(U\times (-\eps,\eps)).
\end{equation*}
We may use the notion of a distributional derivative to express the left hand-side of~\eqref{BilinearizedBaby} as
\begin{equation*}
\Big|\Scalprod{\frac{\partial^2 \sigma}{\partial \xi_d^2}}{\ \hat{f}\hat{\bar{g}}}\Big|,
\end{equation*}
which, by the Plancherel theorem, is equal to
\begin{equation*}
\bigg|\int\limits_{\mathbb{R}^2} \Big(f*\mathcal{F}^{-1}\Big[\frac{\partial^2 \sigma}{\partial \xi_d^2}\Big]\Big) \bar{g}\bigg|,
\end{equation*}
where~$\mathcal{F}$ is the Fourier transform in~$d$ variables. So, the problem has reduced to the question whether convolution with~$\mathcal{F}^{-1}\Big[\frac{\partial^2 \sigma}{\partial \xi_d^2}\Big]$ is~$L_p\to L_{p'}$ bounded:
\begin{equation*}
\Big\|f*\mathcal{F}^{-1}\Big[\frac{\partial^2 \sigma}{\partial \xi_d^2}\Big]\Big\|_{L_{p'}}\lesssim \|f\|_{L_p}, \quad p\in \Big[1,\frac{2d+2}{d+7}\Big].
\end{equation*}

This inequality may be proved with the help of the standard Stein--Tomas method. In our further arguments, it is more convenient to use the fractional integration approach (see \cite{MuscaluSchlag}, $11.2.2$). 
\end{proof}
So, the proof is naturally split into three parts: localization, algebraic tricks that allow to use the vanishing condition, and the estimate of an operator between Lebesgue spaces. The proofs of the ``if'' parts of Theorems~\ref{First_restriction_for_higher_derivatives_Theorem} and~\ref{HDRExample} follow similar scheme. The localization argument is universal, and we place it in Subsection~\ref{s61}.

Theorem~\ref{StubbornTheorem} plays the role of the estimate on Lebesgue spaces. Its proof is situated in Section~\ref{S3}. It follows the general strategy of the fractional integration approach to the Stein--Tomas Theorem. First, we prove Theorem~\ref{StubbornTheorem} in the case~$p=1$ and also provide sharp estimates on the absolute value of the kernel. This step, which is a simple consequence of the Van der Corput Lemma in the classical setting, becomes technically involved in our case. It is presented in Subsection~\ref{s31}. Then, we need to interpolate it with~$L_2$ estimates. This is done with the help of a weighted version of the Stein--Weiss inequality. Though similar generalizations of the Stein--Weiss inequality appear in the literature (see, e.g.,~\cite{Kerman}), we did not manage to find the specific version we need. We present the proof in Subsection~\ref{s62} and prove Theorem~\ref{StubbornTheorem} in Subsection~\ref{s32}. Subsection~\ref{s33} is devoted to generalizations of Theorem~\ref{StubbornTheorem} in the spirit of Strichartz estimates.

The ``algebraic'' part needed to prove Theorem~\ref{First_restriction_for_higher_derivatives_Theorem} is a direct generalization of what was presented above. However, the~$\HDR$ inequalities require additional effort. The derivation of Theorem~\ref{HDRExample} from Theorem~\ref{StubbornTheorem} occupies Section~\ref{S4}. We will have to consider the quantity on the left of~\eqref{HDRinequality} as a function of~$k$ and~$s$ and study its convexity properties. The sufficient conditions listed in Proposition~\ref{SufficientConditionsHDR} then provide the domain of the function, and the condition~$\hat{f} \in H^{\ell}(\Sigma)$ defines some boundary behavior. The difficulty that we cannot overcome in the case of large~$\ell$ comes from lack of convexity of the domain.

Section~\ref{S5} collects the  ``only if'' parts of Theorems~\ref{First_restriction_for_higher_derivatives_Theorem} and~\ref{StubbornTheorem} as well as the proof of Proposition~\ref{SufficientConditionsHDR}. We split the ``extremizers'' into three groups: those which originate from the surface measure conditions, those which are related to Knapp examples, and those which correspond to shifts in the real space. For example, conditions~\eqref{HDRShiftCondition} and~\eqref{HDRSigmaLpShiftCondition} come from shifting functions,~\eqref{HDRSurfaceCondition} comes from plugging the surface measure into~\eqref{HDRinequality}, and conditions~\eqref{HDRKnapp} and~\eqref{HDRShiftedKnapp} come from Knapp examples. Condition~\eqref{HDRShiftedKnapp}  is dictated by a specific shift of a Knapp-type function in the real space. This condition and its sharp numerology are still quite surprising to the authors.

\subsection{Fredholm conditions and Sobolev embeddings}\label{s14}
\paragraph{\bf Fredholm conditions.}
Functions whose Fourier transform vanish on a compact surface in $\mathbb{R}^d$, and in particular on a sphere, arise in the study of spectral theory of Schr\"odinger operators $H = -\Delta + V$.  It is well known that the Laplacian has absolutely continuous spectrum on the positive halfline $[0,\infty)$, and no eigenvalues or singular continuous spectrum.  If $V(x)$ can be approximated by bounded, compactly supported functions in a suitable norm (for example $V \in L_{d/2}(\mathbb{R}^d)$ suffices when $d \geq 3$), then $H$ is a relatively compact perturbation of the Laplacian and may have countably many eigenvalues with a possible accumulation point at zero.  If $V$ is real-valued, then $H$ is a self-adjoint operator whose eigenvalues must all be real numbers as well.

It is not immediately obvious how the continuous spectrum of $H$ relates to that of the Laplacian, and whether any eigenvalues are embedded within it.  An argument due to Agmon~\cite{Agmon} proceeds as follows: Suppose $\psi$ is a formal solution of the eigenvalue equation $(-\Delta - \lambda)\psi = -V\psi$ for some $\lambda > 0$.  Then
\begin{equation} \label{AgmonBootstrap}
\psi = - \lim_{\epsilon \to 0^+}(-\Delta - (\lambda+i\epsilon))^{-1} V \psi
\end{equation}
from which it follows that the imaginary parts of $\scalprod{V\psi}{\psi}$ and $-\lim\limits_{\epsilon \to 0}\scalprod{V\psi}{(-\Delta - (\lambda+i\epsilon))^{-1} V \psi}$ must agree.  The former is clearly zero since $V(x)$ is real-valued.  The latter turns out to be a multiple of $\| (V\psi)\!\hat{\phantom{i}}|_\Sigma \|_{L_2(\Sigma)}^2$, where $\Sigma$ is the sphere $\{\xi \in \mathbb{R}^d\mid |\xi|^2 = \lambda\}$.  Hence $(V\psi)\!\hat{\phantom{i}}$ vanishes on the sphere of radius $\sqrt{\lambda}$.

It is not surprising that the Fourier multiplication operator $m_\epsilon(\xi) = (|\xi|^2 - (\lambda + i\epsilon))^{-1}$ might have favorable mapping properties when applied specifically to $V\psi$, whose Fourier transform vanishes where $m_\epsilon(\xi)$ is greatest.  Bootstrapping arguments using~\eqref{AgmonBootstrap} show that $\psi \in L_2(\mathbb{R}^d)$, even if it was not assumed {\it a priori} to belong to that space.

Viewed another way, the eigenvalue problem $(-\Delta - \lambda)\psi = -V\psi$ is an inhomogeneous partial differential equation were the principal symbol is elliptic.  The Fredholm condition for existence of solutions is that $-V\psi$ should be orthogonal to the null-space of the adjoint operator $(-\Delta - \lambda)^*$.  As we will argue later in Subsection~\ref{s21}, this nullspace consists of all distributions whose Fourier transform acts as a linear functional on $C^\infty(\Sigma)$.  Thus $V\psi$ satisfies the Fredholm condition precisely if $(V\psi)\!\hat{\phantom{i}} |_{\Sigma} = 0$.

The analysis in~\cite{Agmon} is carried out in polynomially weighted $L_2(\mathbb{R}^d)$ and applies to a wide family of elliptic differential operators $H = P(i\nabla) + V$.  The main non-degeneracy condition is that the gradient of $P$ does not vanish on the level set $\{\xi\mid P(\xi) = \lambda\}$.  Similar arguments in~\cite{Goldberg_Schlag}, \cite{Ionescu_Schlag} and~\cite{Goldberg} are carried out (for $P(\xi) = |\xi|^2$) in $L_q(\mathbb{R}^d)$ and related Sobolev spaces for various ranges of $q$. Curvature of the level sets of $P$ is a crucial feature in these works, as it is in the present paper.

\paragraph{\bf Sobolev type inequalities.}
We start with the classical Sobolev Embedding Theorem
\begin{equation*}
\|f\|_{L_{q}} \lesssim \|\nabla f\|_{L_p}, \quad f \in C_0^{\infty}(\mathbb{R}^d),\quad q = \frac{dp}{d-p},\quad 1 \leq p < d, d \geq 2.
\end{equation*}
For~$p > 1$, it follows from the Hardy--Littlewood--Sobolev inequality. In the limiting case~$p=1$, the Hardy--Littlewood--Sobolev inequality fails, however, as it was proved by Gagliardo and Nirenberg, the Sobolev Embedding holds. This happens because the space
\begin{equation*}
\dot{W}_1^1(\mathbb{R}^d) = \clos_{L_1}\Big(\Big\{\nabla f\;\Big|\, f\in C_0^{\infty}(\mathbb{R}^d)\Big\}\Big)
\end{equation*}
is strictly narrower than~$L_1$ (they are even non-isomorphic as Banach spaces). Later, it was observed that there are many similar inequalities where~$\nabla f$ may be replaced with a more complicated differential vector-valued expression (see~\cite{BourgainBrezis},~\cite{vanSchaftingen}, and the survey~\cite{vanSchaftingen2}). 

In~\cite{Stolyarov}, the second named author studied the anisotropic bilinear inequality 
\begin{equation}\label{BilinearEmbedding}
\left|\scalprod{f}{g}_{\dot{W}_2^{\alpha,\beta}(\mathbb{R}^2)}\right| \lesssim \left\|(\partial_1^k -\tau\partial_2^l)f\right\|_{L_1(\mathbb{R}^2)}\left\|(\partial_1^k -\sigma\partial_2^l)g\right\|_{L_1(\mathbb{R}^2)}, \quad f,g \in C_0^{\infty}(\mathbb{R}^2)
\end{equation}
(such type inequalities were used in~\cite{KislyakovMaximovStolyarov} for purposes of Banach space theory). Here~$\dot{W}_2^{\alpha,\beta}(\mathbb{R}^2)$ is the anisotropic Bessel potential space equipped with the norm
\begin{equation*}
\|f\|_{\dot{W}_q^{\alpha,\beta}} = \Big\|\mathcal{F}^{-1}\Big(\hat{f}(\xi,\eta)|\xi|^{\alpha}|\eta|^{\beta}\Big)\Big\|_{L_q}, 
\end{equation*}
the symbols~$\sigma$ and~$\tau$ denote complex scalars, and~$\partial_1$ and~$\partial_2$ are partial derivatives with respect to the first and the second coordinates correspondingly. It appeared that~\eqref{BilinearEmbedding} holds even in the cases where the differential polynomials on the right hand-side are not elliptic, however, this may happen only in the anisotropic case~$k\ne l$. This lead to natural conjecture that the inequality
\begin{equation}\label{SobolevBochnerRiesz}
\|f\|_{\dot{W}_q^{\alpha,\beta}} \lesssim \|(\partial_1^k - \sigma \partial_2^l)f\|_{L_p}, \quad \frac{\alpha}{k} + \frac{\beta}{l} = 1 - \Big(\frac{1}{p} - \frac{1}{q}\Big)\Big(\frac{1}{k} + \frac{1}{l}\Big), k\ne l, p > 1, q< \infty,
\end{equation}
might hold true. We are especially interested in the case where the operator on the right hand-side is non-elliptic, that is~$i^{l-k}\sigma \in \mathbb{R}$. Assume this is so. Similar to the classical proof of the Sobolev Embedding Theorem, one may express~$f$ in terms of~$(\partial_1^k - \sigma \partial_2^l)f$ using a certain integral operator. This will be a Bochner--Riesz type operator of order~$-1$ with the singularity on the curve
\begin{equation*}
\Gamma_{k,l} = \Big\{(\xi,\eta)\in\mathbb{R}^2\;\Big|\,(2\pi i \xi)^k = \sigma (2\pi i \eta)^l\Big\}.
\end{equation*}
Note that this curve is convex outside the origin. Application of the Littlewood--Paley inequality and homogeneity considerations (see~\cite{Stolyarov}) reduce~\eqref{SobolevBochnerRiesz} to the case where the spectrum of~$f$ lies in a small neighborhood of a point on~$\Gamma_{k,l}$. So, by the results of~\cite{Bak}, the inequality~\eqref{SobolevBochnerRiesz} is true if~$\frac{1}{p} - \frac{1}{q} \geq \frac23$,~$p < \frac43$, and~$q > 4$. Moreover, one may construct examples to show that the conditions~$\frac{1}{p} - \frac{1}{q}\geq \frac23$ and~$p < \frac43$ are necessary.

Note that the Fourier transform of the function~$(\partial_1^k - \sigma \partial_2^l)f$ vanishes on~$\Gamma_{k,l}$,
which is a smooth convex curve in the plane (with, possibly, a singularity at zero). Thus, we need to analyze the action of Bochner--Riesz-type operator on the space~$\Si L_p$ with~$\Sigma = \Gamma_{k,l}$. It appears that passing to a narrower space allows to get rid of the condition~$q > 4$. This work was done half year later in~\cite{Goldberg}. 
\begin{Th}[\cite{Stolyarov}+\cite{Goldberg}]
The inequality~\eqref{SobolevBochnerRiesz} holds true if~$\frac{1}{p} - \frac{1}{q} \geq \frac23$,~$1<p< \frac43$, and~$q<\infty$.
\end{Th}

\section{Study of the spaces $\Si L_p^k$}\label{S2}

\subsection{Description of the annihilator and Domar's theory}\label{s21}
Let~$\partial_{\Sigma}^l$ denote the operator of normal derivative of order~$l$ with respect to~$\Sigma$,~$\partial_{\Sigma}\colon \mathcal{S}(\mathbb{R}^d) \to C^{\infty}(\Sigma)$,
\begin{equation*}
\partial_{\Sigma}^l[\Phi](\xi) = \frac{\partial^l \Phi}{\partial \mathrm{n}_{\Sigma}^l(\xi)}(\xi),\quad \Phi \in \mathcal{S}(\mathbb{R}^d), \xi \in \Sigma. 
\end{equation*}
The symbol~$\mathrm{n}_{\Sigma}(\xi)$ denotes the normal vector to~$\Sigma$ at the point~$\xi$. In particular,~$\partial_{\Sigma}^0[\Phi]$ is simply the restriction of~$\Phi$ to~$\Sigma$. 

There are conjugate operators~$(\partial_{\Sigma}^l)^* \colon (C^{\infty}(\Sigma))^{\prime} \to \mathcal{S}^{\prime}(\mathbb{R}^d)$. We can also form an operator~$\J_{\Sigma}^l\colon \mathcal{S}(\mathbb{R}^d) \to \bigoplus_{0 \leq s \leq l} C^{\infty}(\Sigma)$ composed of pure normal derivatives:
\begin{equation*}
\mathcal{S}(\mathbb{R}^d) \ni \Phi \mapsto \J_{\Sigma}^l[\Phi] = \Big(\partial_{\Sigma}^0[\Phi],\partial_{\Sigma}^{1}[\Phi],\ldots,\partial_{\Sigma}^{l}[\Phi]\Big) \in \bigoplus_{0 \leq s \leq l} C^{\infty}(\Sigma).
\end{equation*}
This operator also has an adjoint, which maps a vector-valued distribution~$\Lambda = \{\Lambda_{s}\}_{0\leq s\leq l}$ with compact support on~$\Sigma$ to a Schwartz distribution on~$\mathbb{R}^d$.

\begin{Le}\label{Annihilator}
Let~$p \in [1,\infty)$. The annihilator of~$\Si L_p^k$ in~$L_{p'}$ can be described as
\begin{equation}\label{FormulaForAnnihilator}
\Ann_{L_{p'}}(\Si L_p^k) = \clos_{L_{p'}}\bigg(\Big\{g \in L_{p'}\,\Big|\; \exists \Lambda\in \!\!\!\!\bigoplus_{0\leq s\leq k-1}\!\!\!\!\big(C^{\infty}(\Sigma)\big)^{\prime} \quad \hbox{such that} \quad \hat{g} = \big(\J_{\Sigma}^{k-1}\big)^*[\Lambda]\Big\}\bigg).
\end{equation}
\end{Le}
\begin{Rem}
Since the distribution~$\Lambda$ has compact support,~$g$ has bounded spectrum. Clearly, if~$\Sigma$ is not compact, one may construct a function~$g$ in the annihilator of~$\Si L_p$ whose spectrum is not bounded. That is why we need to add closure on the right hand-side of~\eqref{FormulaForAnnihilator}. In the case where~$\Sigma$ is compact, this is not needed\textup:
\begin{equation*}%\label{FormulaForAnnihilator}
\Ann_{L_{p'}}(\Si L_p^k) = \Big\{g \in L_{p'}\,\Big|\; \exists \Lambda\in \!\!\!\!\bigoplus_{0\leq s\leq k-1}\!\!\!\!\big(C^{\infty}(\Sigma)\big)^{\prime} \quad \hbox{such that} \quad \hat{g} = \big(\J_{\Sigma}^{k-1}\big)^*[\Lambda]\Big\}, \quad \Sigma\ \hbox{is compact}.
\end{equation*}
The proof of Lemma~\ref{Annihilator} presented below also simplifies in the case where~$\Sigma$ is compact. The functions~$\psi$ and~$\Psi$ may be omitted in this case.
\end{Rem}
We will need a technical fact to prove Lemma~\ref{Annihilator}. It is standard, so we omit its proof.
\begin{Le}\label{ExtensionOperator}
For any bounded domain~$\Omega$, consider the subspace~$C^{\infty}(\Sigma,\Omega,l)$ of vector-valued functions in~$\bigoplus_{0 \leq s \leq l} C^{\infty}(\Sigma)$ supported in~$\Omega\cap \Sigma$. There exists a linear operator~$\Ext_{\Omega,l}\colon C^{\infty}(\Sigma,\Omega,l) \to \mathcal{S}(\mathbb{R}^d)$, which is inverse to~$\J_{\Sigma}^l$ in the sense
\begin{equation*}
\forall \varphi \in C^{\infty}(\Sigma,\Omega,l) \quad \J_{\Sigma}^l\big[\Ext_{\Omega,l}[\varphi]\big] = \varphi.
\end{equation*}
\end{Le}
\begin{proof}[Proof of Lemma~\ref{Annihilator}]
First, we note that since~$\Si L_p^k$ is a translation invariant space, the set of functions~$g$ with compact spectrum is dense in~$\Ann_{L_{p'}}(\Si L_p^k)$. Consider such a function~$g$. It suffices to construct
\begin{equation*}
\Lambda\in\bigoplus_{0\leq s\leq k-1}\big(C^{\infty}(\Sigma)\big)^{\prime}
\end{equation*}
such that~$\hat{g}=\big(\J_{\Sigma}^{k-1}\big)^*[\Lambda]$. 

Let~$\Omega$ be a bounded domain containing the spectrum of~$g$. Consider the operator~$\Ext_{\Omega,k-1}$ constructed in Lemma~\ref{ExtensionOperator} and define~$\Lambda$ (as a functional on~$\bigoplus_{0\leq s\leq k-1}C^{\infty}(\Sigma)$) by the formula
\begin{equation*}
\Lambda[\varphi] = \hat{g}\big[\Ext_{\Omega,k-1}[\psi\varphi]\Big],\quad \varphi\in \bigoplus_{0\leq s\leq k-1}C^{\infty}(\Sigma),
\end{equation*}
where~$\psi$ is a smooth function on~$\Sigma$ supported in~$\Omega$ that is equal to one in a neighborhood of~$\supp\hat{g}\cap \Sigma$.
We are required to show that~$\hat{g} = \big(\J_{\Sigma}^{k-1}\big)^*[\Lambda]$, which becomes
\begin{equation*}
\scalprod{\hat{g}}{\Phi} = \scalprod{\hat{g}}{\Ext_{\Omega,k-1}[\psi\J_{\Sigma}^{k-1}[\Phi]]}
\end{equation*}
for every~$\Phi \in \mathcal{S}(\mathbb{R}^d)$. Since~$\psi=1$ in a neighborhood of the support of~$\hat{g}$, we may write
\begin{equation*}
\scalprod{\hat{g}}{\Ext_{\Omega,k-1}[\psi\J_{\Sigma}^{k-1}[\Phi]]}= \scalprod{\hat{g}}{\Ext_{\Omega,k-1}[\J_{\Sigma}^{k-1}[\Psi\Phi]]}\quad \hbox{and} \quad \scalprod{\hat{g}}{\Phi}=\scalprod{\hat{g}}{\Psi\Phi},
\end{equation*}
where~$\Psi$ is smooth function supported in~$\Omega$ that equals one in a neighborhood of~$\supp\hat{g}$. It remains to prove
\begin{equation*}
\Scalprod{ \hat{g}}{ \Big(\id - \Ext_{\Omega,k-1}\circ \J_{\Sigma}^{k-1}\Big)[\Psi\Phi]} = 0.
\end{equation*}
Since~$\Ker \J_{\Sigma}^{k-1} \subset \Ker \hat{g}$ (recall that~$g$ annihilates~$\Si L_p^k$), it suffices to show that
\begin{equation*}
\J_{\Sigma}^{k-1}\circ\big(\id - \Ext_{\Omega,k-1}\circ\J_{\Sigma}^{k-1}\big) = 0,
\end{equation*}
which holds true since~$\J_{\Sigma}^{k-1}\circ \Ext_{\Omega,k-1} = \id$ by construction of~$\Ext_{\Omega,k-1}$.
\end{proof}
\begin{Le}\label{Smoothing_Lemma}
The set
\begin{equation*}
\Big\{g \in L_{p'}\,\Big|\; \exists \Lambda\in \!\!\!\!\bigoplus_{0\leq s\leq k-1}\!\!\!\!C_0^{\infty}(\Sigma) \quad \hbox{such that} \quad \hat{g} = \big(\J_{\Sigma}^{k-1}\big)^*[\Lambda]\Big\}
\end{equation*}
is dense in~$\Ann_{L_{p'}}(\Si L_p^k)$ if~$p > 1$. In the case~$p=1$, this set is weakly dense.
\end{Le}
\begin{proof}
The case~$p=1$ had been considered in~\cite{Domar}. We repeat the argument for the general case here. Let~$g$ be a function in the said annihilator. After applying a partition of unity, we may assume that the corresponding vector-valued distribution~$\Lambda$ provided by Lemma~\ref{Annihilator} is supported in a chart neighborhood~$V$ of a point~$\xi \in \Sigma$, as it will only be necessary to sum a finite number of such pieces. We may also suppose that~$\xi = 0$ and
\begin{equation}\label{LocalForm}
\Sigma \cap V = \{(\zeta,h(\zeta))\,|\; \zeta \in U\},
\end{equation} 
here~$U$ is a neighborhood of the origin in~$\mathbb{R}^{d-1}$ and~$h\colon \mathbb{R}^{d-1} \to \mathbb{R}$ is a smooth function such that~$h(0)=0$,~$\nabla h(0) = 0$ (see Subsection~\ref{s61} for details). By our assumptions on the principal curvatures of~$\Sigma$, the second differential~$\frac{\partial^2 h}{\partial \zeta^2}$ is non-degenerate on~$U$. Consider the operator~$S$ that makes~$\Sigma\cap V$ flat:
\begin{equation}\label{Exponent}
S[\Phi](\xi) = \Phi\big(\xi_{\bar{d}},\xi_d + h(\xi_{\bar{d}})\big), \quad \Phi \in C_0^{\infty}(V),\quad \xi \in \mathbb{R}^{d}.
\end{equation}
We use the notation~$\xi=(\xi_{\bar{d}},\xi_d)$, so~$\xi_d$ is the last coordinate of~$\xi$ and~$\xi_{\bar{d}}\in\mathbb{R}^{d-1}$ is the vector consisting of first~$d-1$ coordinates.

Let~$\psi$ be a compactly supported smooth function on~$\mathbb{R}^{d-1}$ with unit integral, let~$\psi_n(\zeta) = n^{d-1}\psi(n\zeta)$ be its dilations. The function~$\xi \mapsto \psi_n(\xi_{\bar{d}})$ is also denoted by~$\psi_n$. Consider the family of operators~$D_n$, $n \geq n_0$, $n_0$ is sufficiently large, given by the rule
\begin{equation}\label{DomarOperator}
D_n[\Phi] = \mathcal{F}^{-1}\Big[S^{-1}\big[S[\Psi\hat{\Phi}]*\psi_n\big]\Big], \quad \Phi \in\mathcal{S}(\mathbb{R}^d).
\end{equation}
Here~$\Psi \in C_0^{\infty}(V)$ is a function that equals one on the support of~$\Lambda$. It is clear that the~$D_n$ are uniformly bounded as operators on~$L_2(\mathbb{R}^d)$. Lemma~$4.2$ in~\cite{Domar} says that the~$D_n$ are also (uniformly in~$n$) bounded as operators on~$L_1$, and since the dual operators have an identical structure they are also bounded on $L_\infty$.  By interpolation, the~$D_n$ are uniformly bounded on~$L_{p'}$. Also, since~$\{\psi_n\}_{n\in\mathbb{N}}$ is an approximate identity,
\begin{equation*}
\|D_n[\Phi] - \Phi\|_{L_{p'}}\to 0 \quad  \hbox{when} \quad \Phi \in \mathcal{S}(\mathbb{R}^d).
\end{equation*}
Thus, for every~$g \in L_{p'}$,~$p' < \infty$, we have~$\|D_n[g] - g\|_{L_{p'}} \to 0$. It remains to notice that~$D_n$ maps compactly supported distributions of the form
\begin{equation*}
\mathcal{F}^{-1}\Big[\big(\J_{\Sigma}^{k-1}\big)^*[\Lambda]\Big], \quad \Lambda \in\!\!\!\bigoplus_{0\leq s\leq k-1}\!\!\!\!\big(C^{\infty}(\Sigma)\big)^{\prime},
\end{equation*}
to the ones for which~$\Lambda \in \bigoplus_{0\leq s\leq k-1} C^{\infty}_0(\Sigma)$. Thus, if~$g \in \Ann_{L_{p'}}(\,\! \Si L_p^k)$ is a function with bounded spectrum, then~$D_n[g] \in \Ann_{L_{p'}}(\,\! \Si L_p^k)$,~$D_n[g]$ is generated by smooth~$\Lambda$, and~$D_n[g] \to g$ in~$L_{p'}$.
\end{proof}
\begin{proof}[Proof of Theorem~\ref{Stability}.]
By Lemma~\ref{Annihilator}, it suffices to show that any function~$g \in \Ann_{L_{p'}}(\,\!\Si L_p^{k+1})$ can be approximated by functions in~$\Ann_{L_{p'}}(\,\!\Si L_p^{k})$ when~$k \geq \frac{d}{p} - \frac{d+1}{2}$, and~$k > \frac{d-1}{2}$ when~$p=1$. By Lemma~\ref{Smoothing_Lemma}, we may assume that
\begin{equation}\label{CanonicForm}
\hat{g} = \big(\J_{\Sigma}^{k}\big)^*[\Lambda], \qquad \Lambda\in \!\!\bigoplus_{0\leq s\leq k}\!\!C^{\infty}_0(\Sigma).
\end{equation}
It suffices to prove that~$\Lambda_{k} = 0$, where~$\Lambda = (\Lambda_0,\Lambda_1,\ldots,\Lambda_{k})$. We may suppose that~$\Lambda$ is supported in a neighborhood~$V$ of a point on~$\Sigma$. We may also assume~\eqref{LocalForm} and replace normal derivatives by derivatives with respect to~$\xi_d$:
\begin{equation*}
\hat{g} = \sum\limits_{s=0}^{k} \frac{\partial^s\tilde{\Lambda}_s}{\partial \xi_d^s},
\end{equation*}  
where~$\tilde{\Lambda}_s$ are distributions generated by complex measures on~$\Sigma\cap V$ whose densities with respect to the Lebesgue measure on~$\Sigma$ are smooth functions. Note that~$\Lambda_{k} =0$ whenever~$\tilde{\Lambda}_{k} = 0$.  Since each function~$\tilde{\Lambda}_s$ has smooth density with respect to the Lebesgue measure on~$\Sigma$, one may use the stationary phase method to compute the asymptotics of~$\mathcal{F}^{-1}[\tilde{\Lambda}_s]$ at infinity (see, e.g.~\cite{SteinBook}):
\begin{equation*}
\mathcal{F}^{-1}[\tilde{\Lambda}_s] = e(x)|x|^{-\frac{d-1}{2}} + O(|x|^{-\frac{d+1}{2}})
\end{equation*}
for all~$x$ such that~$x\!\parallel\! \mathrm{n}_{\Sigma}(\xi)$ for some~$\xi \in \Sigma$ with~$\tilde{\Lambda}_s(\xi) \ne 0$. Here~$e(x)$ is a non-zero oscillating factor with constant amplitude that depends on~$h$ and the density of~$\tilde{\Lambda}_s$. This shows that
\begin{equation*}
|\hat{g}(x)| \asymp |x|^{k - \frac{d-1}{2}}, \quad x \to \infty, \quad x\parallel \mathrm{n}_{\Sigma}(\xi), \tilde{\Lambda}_{k}(\xi) \ne 0.
\end{equation*}
On the other hand,~$\hat{g} \in L_{p'}$, which for $p>1$ requires~$p'(k - \frac{d-1}{2}) < -d$, equivalently $k < \frac{d}{p} - \frac{d+1}{2}$, contradicting our assumptions. Therefore,~$\Lambda_{k} = 0$ and, thus,~$\Ann_{L_{p'}}(\,\!\Si L_p^{k+1}) = \Ann_{L_{p'}}(\,\!\Si L_p^{k})$.  If $p = 1$ the contradiction is reached provided $k - \frac{d-1}{2} > 0$.

To show that~$\Si L_p^k = \,\!\Si L_p^{\infty}$, we note that the annihilator of the latter space consists of all~$L_{p'}$ functions whose Fourier transform is supported on~$\Sigma$, recall the Schwartz theorem that any distribution supported on~$\Sigma$ may be represented in the form~$\big(\J_{\Sigma}^{l}\big)^*[\zeta]$ for some~$l$, and use the reasoning above.
\end{proof}

\subsection{Coincedence of $\Si L_p^k$ and the spaces defined as kernels of restriction operators}\label{s23}
We relate the~$\Si L_p^k$ spaces with restriction operators. Consider a neighborhood~$V$ such that~\eqref{LocalForm} holds true. We may redefine~$V$ in such a way that
\begin{equation*}
V = \{\xi \in\mathbb{R}^d\mid |\xi_d-h(\xi_{\bar{d}})| < \delta, \ \xi_{\bar{d}}\in U\}.
\end{equation*}
This gives a natural parametrization of~$V$ by~$U\times (-\delta,\delta)$. We will need the translated copies of~$\Sigma$:
\begin{equation*}
\Sigma_r = \Sigma + (0,0,\ldots,0,r),\quad r\in (-\delta,\delta).
\end{equation*}
Note that this definition depends on the choice of~$U$.

Consider the restriction operators
\begin{equation}\label{Restriction}
\R_{\Sigma,r}[f] = \hat{f}|_{V\cap \Sigma_r}, \quad f \in \mathcal{S}(\mathbb{R}^d).
\end{equation}
\begin{Def}\label{StrongRestriction}
We say that the statement~$\R(\Sigma,p,s)$ holds true if the~$\R_{\Sigma,r}$ admit continuous extensions as~$L_p(\mathbb{R}^d) \to H^{-s}(V\cap\Sigma_r)$ operators for any choice of~$U$, and the norms of these extensions are uniform in~$r$ \textup(however, we do not require any uniformity with respect to~$U$\textup).

We say that~$\R^k(\Sigma,p,s)$ is true if~$\R^{k-1}(\Sigma,p,s)$ is true and for any choice of~$U$ the operators~$\R_{\Sigma,r}$ extend continuously from the domain
\begin{equation*}
\{f\in \mathcal{S}(\mathbb{R}^d)\mid \forall l < k \quad \nabla^l\hat{f} = 0\; \hbox{on}\;\Sigma\}
\end{equation*}
to a family of mappings~$\Si L_p^k \to H^{-s}(V\cap \Sigma_r)$ whose norms are  bounded uniformly by~$C|r|^k$.
\end{Def}  
\begin{Rem}
In the definitions above it is important to be consistent with regard to the construction of local Sobolev norms on $\Sigma_r$. When we discuss~$\R(\Sigma,p,s)$, we will define the Sobolev norm by the rule~\eqref{OurSobolev} for each particular choice of~$U$,~$h$,~and $r$.
\end{Rem}
\begin{Def}\label{TildeSpaces}
For a fixed $s \geq 0$, define the set~$\Si \tilde{L}_p^k$ by the formula
\begin{equation*}
\Si \tilde{L}_p^k = \Big\{f \in L_p(\mathbb{R}^d)\,\Big|\;\forall U \quad \big\|\R_{\Sigma,r}[f]\big\|_{H^{-s}(V\cap\Sigma_r)} = o(|r|^{k-1})\Big\}.
\end{equation*}
\end{Def}
Note that it is unclear whether~$\Si \tilde{L}_p^k$ is closed in $L_p$ or not. 

\begin{Le}\label{CoincidenceLemma}
Suppose that~$\Sigma$ has non-vanishing curvature,~$p \in [1,\infty)$, and~$\R^{k-1}(\Sigma,p,s)$ holds. Then,~$\Si \tilde{L}_p^k =\,\! \Si L_p^k$.
\end{Le}
\begin{proof}
It is clear that~$\Si L_p^k\subset  \clos(\,\! \Si \tilde{L}_p^k)$, as $\Si \tilde{L}_p^k$ contains all Schwartz functions whose Fourier transform vanishes to order $k-1$ on $\Sigma$. In fact, thanks to the uniform convergence implied by condition $\R^{k-1}(\Sigma, p, s)$ it is even true that $\Si L_p^k\subset \,\! \Si \tilde{L}_p^k$. So it suffices to show that~$\Si \tilde{L}_p^k \subset \,\!  \Si L_p^k$. Assume the contrary. By the Hahn--Banach theorem,
\begin{equation*}
\exists f \in \,\!\Si \tilde{L}_p^k, \ g \in \Ann_{L_{p'}}\Big(\,\!\Si L_p^k\Big),\qquad \scalprod{f}{g} \ne 0.
\end{equation*}
By Lemma~\ref{Smoothing_Lemma}, we may assume that~$g$ is of the form
\begin{equation}\label{FormForg}
\hat{g} = (\J_{\Sigma}^{k-1})^*[\zeta], \qquad \zeta \in \!\!\!\bigoplus_{0\leq s \leq k-1}\!\!\!C_0^{\infty}(\Sigma). 
\end{equation}
Applying the same reasoning as in the proof of Lemma~\ref{Smoothing_Lemma}, we may also assume that $\zeta$ has compact support within a chart neighborhood $V$ where $\Sigma \cap V$ is the graph of a smooth function $h:U \subset \mathbb R^{d-1} \to \mathbb R$, $h(0) = 0$, and $\nabla h(0) = 0$.  

Recall the ``flattening" operator $S$ defined in~\eqref{Exponent}.  There exists another set of functions $\tilde{\zeta} \in \bigoplus_{0 \leq s \leq k-1} C_0^\infty(U)$ such that $(S^{-1})^*(\J_\Sigma^{k-1})^*[\zeta] = (\J_{\mathbb R^{d-1}}^{k-1})^*[\tilde{\zeta}]$.  If one considers each component of $\zeta$ as an element of $C^\infty_0(U)$ via the parametrization of $\Sigma \cap V$, then the components of $\tilde{\zeta}$ are constructed from $\zeta$, its gradients (in $\mathbb R^{d-1}$) up to order $k-1$, and the partial derivatives of $h$.  

Let~$\psi$ be a compactly supported function on the unit interval whose integral equals one. Consider its dilations~$\psi_n(\xi_d) = n\psi(n\xi_d)$  and the formal convolution  in the~$d$-th variable
\begin{equation*}
(\J_{\mathbb R^{d-1}}^{k-1})^*[\tilde{\zeta}](\xi_{\bar{d}}, \,\cdot\,) *\psi_n (\xi_d)= \sum\limits_{s=0}^{k-1}\tilde{\zeta}_s(\xi_{\bar{d}})\frac{\partial^s \psi_n}{\partial \xi_d^s}(\xi_d),
\end{equation*}
which is a function in $C^\infty_0(U \times [-\frac{1}{n}, \frac{1}{n}])$.  This function is bounded pointwise by the maximum size of $|\psi_n^{(k-1)}(\xi_d)|$, which is approximately $n^k$.  All of its partial derivatives in the $\xi_{\bar{d}}$ directions are bounded by~$n^k$ as well, because those derivatives act on $\tilde{\zeta}$, not on $\psi_n$.

Now define~$g_n \in \mathcal{S}(\mathbb{R}^d)$ by the formula
\begin{align*}
g_n(\xi_{\bar{d}}, \xi_d) = S^*[(\J_{\mathbb R^{d-1}}^{k-1})^*[\tilde{\zeta}](\xi_{\bar{d}}, \,\cdot\,) *\psi_n]
&= S^*[(S^{-1})^*(\J_\Sigma^{k-1})^*[\zeta](\xi_{\bar{d}}, \,\cdot\,) * \psi_n] \\
&= S^*[(S^{-1})^*[\hat{g}](\xi_{\bar{d}},\,\cdot\,)* \psi_n ] .
\end{align*}
It should be clear that convolution in the $\xi_d$ direction commutes with operators $S^*$ and its inverse, thus the construction simplifies to
\begin{equation*}
\hat{g}_n(\xi_{\bar{d}}, \xi_d) = \hat{g}(\xi_{\bar{d}}, \,\cdot\,)*\psi_n (\xi_d), \quad \text{or} \quad g_n(x) = g(x)\check{\psi}(n x_d).
\end{equation*}  
It follows that~$g_n \to g$ in~$L_{p'}$ (in the case~$p'=\infty$ we have weak-* convergence relative to $L_1$ instead), which means that,~$\scalprod{f}{g_n} \rightarrow \scalprod{f}{g}$.
On the other hand,
\begin{equation*}
\|\R_{\Sigma, r}[g_n] \|_{H^s(U)} \lesssim \begin{cases}
n^{k},\quad |r| \lesssim n^{-1};\\
0,\quad |r| \gtrsim n^{-1}. 
\end{cases} 
\end{equation*}
That gives a bound
\begin{equation*}
|\scalprod{f}{g_n}| = |\scalprod{\hat{f}}{\hat{g}_n}| \lesssim \int_{-\frac{1}{n}}^{\frac{1}{n}} \int\limits_{\Sigma_r}|\scalprod{\hat{f}}{\hat{g}_n}|\,d\sigma dr \lesssim
\frac{1}{n}\cdot n^{k} \cdot o(n^{1-k}) = o(1), 
\end{equation*}
forcing $\scalprod{f}{g} = \lim_{n\to\infty} \scalprod{f}{g_n} = 0$.  This contradicts the original assertion that $\scalprod{f}{g} \not= 0$.
\end{proof}

\begin{Def}\label{WeakRestriction}
We say that the statement~$\Rw^k(\Sigma,p,s)$ holds true if the mapping
\begin{equation*}
\mathcal{S}(\mathbb{R}^d)\ni f \mapsto \nabla^k f|_{\Sigma} \in C^{\infty}(\Sigma)
\end{equation*}
extends to a bounded linear operator between the spaces~$\Si\,\! L_p^{k}$ and~$H^{-s}(K \cap\Sigma)$ for any compact set~$K$.
\end{Def}
It is explained in the Remark~\ref{GeneralLocalization} below that~$\R^k(\Sigma,p,s)$ leads to~$\Rw^k(\Sigma,p,s)$. 

We end this subsection with an analog of Lemma~\ref{CoincidenceLemma} for~$\HDR$ inequalities. The proof is direct, i.e. does not use duality.
\begin{Le}\label{ApproximationForHDR}
For any function~$f \in L_p$ such that~$\phi\hat{f}|_{\Sigma} \in H^{\ell}(\Sigma)$, where~$\phi \in C_0^\infty(\mathbb{R}^d)$, there exists a sequence~$\{f_n\}_n$ of Schwartz functions such that
\begin{equation*}
\|f-f_n\|_{L_p} + \big\|\phi(\hat{f}-\hat{f}_n)\big|_{\Sigma}\big\|_{H^\ell(\Sigma)} \to 0.
\end{equation*}
\end{Le}
\begin{proof}
As usual, we may assume that~$\phi$ and~$\hat{f}$ are supported in a chart neighborhood~$V$ of a point~$\xi \in \Sigma$. We may also suppose that~$\xi = 0$ and~\eqref{LocalForm},
where~$U$ is a neighborhood of the origin in~$\mathbb{R}^{d-1}$ and~$h\colon \mathbb{R}^{d-1} \to \mathbb{R}$ is a smooth function such that~$h=0$,~$\nabla h = 0$ (see Subsection~\ref{s61} for details). By Theorem~\ref{Stability}, in the regime~$\sigma_p \leq 0$, the set of Schwartz functions whose Fourier transform vanishes on~$\Sigma$, is dense in~$L_p$, and there is nothing to prove. Let us assume~$\sigma_p > 0$.

We construct the functions~$F_n$ by the rule
\begin{equation*}
F_n(x) = f(x)\Psi\Big(\frac{x}{a_n}\Big),
\end{equation*}
where~$\Psi$ is a fixed Schwartz function with~$\Psi(0) = 1$ and bounded spectrum, and~$a_n$ is a large number such that~$\|f-F_n\|_{L_p} \leq 2^{-n}$. The functions~$F_n$ approximate~$f$ in~$L_p$ norm, however, their Fourier transforms may have infinite~$H^{\ell}(\Sigma)$ norms. There is a control on a weaker quantity, namely, Theorem~\ref{First_restriction_for_higher_derivatives_Theorem} (in the case~$k=0$) says that 
\begin{equation}\label{VeryWeakEstimate}
\|\phi(\hat{f}-\hat{F}_n)\|_{H^{-s}(\Sigma)} \lesssim 2^{-n}
\end{equation}
for sufficiently large~$s$.

Let now~$f_n = D_n[F_n]$, where~$D_n$ is the Domar operator~\eqref{DomarOperator}. We need to prove two limit identities
\begin{equation*}
\|f-f_n\|_{L_p}\to 0,\quad \hbox{and}\quad \|\hat{f} - \hat{f}_n\|_{H^\ell(\Sigma)} \to 0.
\end{equation*}
The first identity is simple since
\begin{equation*}
\|f-f_n\|_{L_p} \leq \|f - D_n f\|_{L_p} + \|D_n [f- F_n]\|_{L_p} \lesssim o(1) + 2^{-n}
\end{equation*}
by the properties of the operators~$D_n$ (see the proof of Lemma~\ref{Smoothing_Lemma}). For the second identity, we write
\begin{equation*}
\|\phi(\hat{f} - \hat{f}_n)\|_{H^\ell(\Sigma)} \leq \|\phi \mathcal{F}[f- D_n f]\|_{H^{\ell}(\Sigma)} + \big\|\phi \mathcal{F}\big[D_n[f-F_n]\big]\big\|_{H^{\ell}(\Sigma)}.
\end{equation*}
Note that~$D_n$ convolves the restriction to~$\Sigma$ of the Fourier transform of the function with~$\psi_n$ (see~\eqref{DomarOperator}). Thus, the first summand tends to zero by the approximation of identity properties (and since~$\hat{f}|_{\Sigma}\in H^\ell$), and the second summand is bounded by~$O(n^{\ell+s}2^{-n})$ by formula~\eqref{VeryWeakEstimate}.
\end{proof}

\subsection{Proofs of ``if'' part in Theorem~\ref{First_restriction_for_higher_derivatives_Theorem} and Theorem~\ref{CoincidenceTheorem}}\label{s24}
Since~$\R^k(\Sigma,p,s)$ leads to~$\Rw^k(\Sigma,p,s)$, Theorem~\ref{First_restriction_for_higher_derivatives_Theorem} follows from the lemma below.
\begin{Le}\label{R2}
The statement~$\R^k(\Sigma,p,0)$ holds true if~$p \in [1, \frac{2d+2}{d+3+4k}]$.  For every $p \in [1, \frac{2d}{d+1+2k})$ there exists $s \leq \max(0, [2k-\sigma_p] +1, 2k-\kappa_p)$ such that $\R^k(\Sigma, p, s)$ is true.  When $p=1$ and $\sigma_1 = \frac{d-1}{2} \in \mathbb{N}$, the value $s = \max(0, 2k-\frac{d-1}{2})$ suffices.   Finally, in odd dimensions $\R^{\frac{d-1}{2}}(\Sigma, 1, s)$ holds for $s > \frac{d-1}{2}$.
\end{Le}
\begin{proof}
Consider the case $s=0$.  It suffices to prove the bound
\begin{equation*}
\big\|\psi(\cdot)\hat{f}(\cdot,h(\cdot)+r)\big\|_{L^2(U)}^2 \lesssim_{\psi} |r|^{2k}\|f\|_{\Si L_p^k}^2, \quad \psi \in C_0^{\infty}(U).
\end{equation*}
By definition of~$\Si L_p^k$, we may assume that~$f$ is a Schwartz function. Then, the function~$\Theta$ given by the rule
\begin{equation}\label{FunctionTheta}
\Theta(r) = \big\|\psi(\cdot)\hat{f}(\cdot,h(\cdot)+r)\big\|_{L^2(U)}^2
\end{equation}
is smooth. We need to prove~$|\Theta(r)| \lesssim |r|^{2k}\|f\|_{\Si L_p^k}^2$, and for that, it suffices to show an inequality and several equalities. 

The inequality is
\begin{equation}\label{InequalityForTheta}
\forall r \in (-\delta,\delta)\quad \Big|\frac{\partial^{2k} \Theta}{\partial r^{2k}}(r)\Big| \lesssim \|f\|_{L_p}^2,
\end{equation}
which follows from Theorem~\ref{StubbornTheorem} (take~$\alpha = 0, \beta = 2k,\gamma = 0$,~$\Sigma_r$ in the role of~$\Sigma$, and notice that~\eqref{SISpectralShift} is satisfied automatically,~\eqref{SIKnapp} is equivalent to~$p \in [1, \frac{2d+2}{d+3+4k}]$, and~\eqref{SIGaussian} follows from~\eqref{SIKnapp} in this case).

The equalities are
\begin{equation}\label{EqualityForTheta}
\forall j \in [0..2k-1] \quad \frac{\partial^{j} \Theta}{\partial r^{j}}(0) = 0.
\end{equation}
Indeed, we use the product rule:
\begin{equation}\label{ProductRule}
\frac{\partial^{j} }{\partial r^{j}}\bigg[\big\|\psi(\cdot)\hat{f}(\cdot,h(\cdot)+r)\big\|_{L^2(U)}^2\bigg]\bigg|_{r=0} = \sum\limits_{i\leq j}C_j^i\Scalprod{\psi(\cdot)\frac{\partial^{i} \hat{f}}{\partial \xi_d^{i}}(\cdot,h(\cdot))}{\ \psi(\cdot)\frac{\partial^{j-i} \hat{f}}{\partial \xi_d^{j-i}}(\cdot,h(\cdot))}_{L^2(U)}
\end{equation}
and notice that in each scalar product on the right hand-side, one of the functions is identically zero since either~$i < k$ or~$j-i<k$.

It remains to combine~\eqref{InequalityForTheta},~\eqref{EqualityForTheta}, and the Taylor integral remainder formula to complete the proof in the case~$s=0$.

When $2k > \kappa_p$, the choice of $\alpha = 0$, $\beta = 2k$, and $\gamma=0$ is no longer available in Theorem~\ref{StubbornTheorem}. Suppose~$p > 1$. In order to use the product-rule argument above, one must set~$2\alpha + \beta = 2k$, and it is desirable to keep $\alpha$ as small as possible since $\gamma \geq \alpha$ is a prominent lower bound for $s$.  We can apply Theorem~\ref{StubbornTheorem} with
$\alpha = [2k - \sigma_p] + 1$, $\beta = 2k -2\alpha$ (note that~$\beta \geq 0$ and~$\alpha \geq 0$ here), and $s = \gamma = \max(\alpha, 2k-\kappa_p)$, then follow the above steps for
\begin{equation*}
\Theta_1(r) = \Big\|\psi(\cdot)\Big(\frac{\partial}{\partial r}\Big)^\alpha \hat{f}(\cdot,h(\cdot)+r)\Big\|_{H^{-s}(U)}^2
\end{equation*}
to conclude that $\frac{\partial^{\beta} \Theta_1}{\partial r^{\beta}}(r) \lesssim \|f\|_{L_p}^2$ for all  $r \in (-\delta, \delta)$, and every lower order derivative vanishes at $r=0$ because $f \in \!\, \Si L_p^k$.  Thus $\|\psi(\cdot) (\frac{\partial}{\partial r})^\alpha \hat{f}(\cdot, h(\cdot)+r)\|_{H^{-s}(U)} \lesssim |r|^{\beta/2} \|f\|_{\Si L_p^k}$.  Furthermore, $(\frac{\partial}{\partial r})^j \hat{f}(\cdot, h(\cdot) + r)$ is assumed to vanish at $r=0$ for each $0 \leq j < \alpha \leq k$.  The Taylor remainder formula and the Minkowski inequality show that $\|\hat{f}(\cdot, h(\cdot)+r)\| \lesssim |r|^{\alpha + \beta/2} \|f\|_{\Si L_p^k}$, and we previously set $\alpha + \frac{\beta}{2} = k$.

When $p=1$ and $k < \sigma_1 = \frac{d-1}{2} \in \mathbb{N}$, it is also permissible to apply Theorem~\ref{StubbornTheorem} with $\alpha = 2k - \frac{d-1}{2}$, $\beta = 2k-2\alpha$, and $s = \gamma = 2k - \frac{d-1}{2}$.  In the endpoint case $p=1, k=\frac{d-1}{2} \in \mathbb{N}$, Corollary~\ref{EndpointCaseWithoutWeight} below directly states that $|\Theta_1(r)| \lesssim \|f\|_{L_1}^2$ for $s > \frac{d-1}{2}$ and $\alpha = k = \frac{d-1}{2}$.

\end{proof}
\begin{proof}[Proof of Theorem~\ref{CoincidenceTheorem}.]
Clearly,~$\Si L_p^{k+1} \subset \Ker\R^{k}$ whenever $\R^k$ is suitably defined as a map from~$\Si L_p^{k+1}$ into $H^{-s}_{\loc}(\Sigma)$. To prove the reverse embedding, it suffices to show that
$\Ker\R^{k}\subset\ \!\, \Si \tilde{L}_p^{k+1}$,
since by Lemmas~\ref{R2} (with the value of $s$ specified there) and~\ref{CoincidenceLemma}, we have~$\Si L_p^{k+1} =\!\, \Si \tilde{L}_p^{k+1}$ for these choices of~$p$ and~$s$. 

We first consider the case $s=0$, $p \in [1, \frac{2d+2}{d+3+4k}]$.  By Definition~\ref{TildeSpaces}, we need to prove
\begin{equation*}
\Theta(r) = o(|r|^{2k}), \quad r\to 0, \quad f\in \Ker\R^{k}
\end{equation*}
where the function~$\Theta$ is defined by~\eqref{FunctionTheta}. By Taylor integral remainder formula and~\eqref{EqualityForTheta}, we simply need to show a slight refinement of~\eqref{InequalityForTheta}:
\begin{equation}\label{TildaProperty}
\lim_{r\to 0} \Big|\frac{\partial^{2k} \Theta}{\partial r^{2k}}(r)\Big| = 0,\quad f\in \Ker\R^{k}.
\end{equation}
Note that~\eqref{InequalityForTheta} holds for all $f \in L_p$.  By approximating any such $f$ by Schwartz functions, we see that $\frac{\partial^{2k}\Theta}{\partial r^{2k}}(r)$ is also continuous in $r$.  If $f \in \!\, \Si L_p^k$, the computation in~\eqref{ProductRule} with $j = 2k$ shows that $\frac{ \partial^{2k} \Theta}{\partial r^{2k}}(0) =  C^k_{2k} \big\|\psi \frac{\partial^k\hat{f}}{\partial \xi_d^k}\big\|_{L^2(U)}^2$, and for $f \in \Ker\R^k \subset \!\, \Si L_p^k$, the norm on the right hand-side is zero.

The remaining case is essentially the same. This time $\frac{\partial^\beta \Theta_1}{\partial r^\beta}(r)$ is continuous for all $f \in L_p$, the lower order derivatives vanish at $r=0$ for $f \in \!\,\Si L_p^k$, and finally $\frac{\partial^\beta \Theta_1}{\partial r^\beta}(0) = 0$ if $f \in \Ker \R^k$.  Thus $\frac{\partial^\beta \Theta_1}{\partial r^\beta}(r) = o(|r|^{\beta})$ and the rest of the integrations are the same as in Lemma~\ref{R2}.
\end{proof}

\subsection{Proof of Corollary~\ref{HDRDual}}	
Let $X$ be the vector space of functions $\{f \in L_p(\mathbb{R}^d)\mid \phi\hat{f} \in H^{\ell}(\Sigma)\}$ equipped with the norm $\|f\|_X = \|f\|_{L_p} + \|\phi\hat{f}\|_{H^\ell(\Sigma)}$.  This space contains all functions in the Schwartz class, and convergence with respect to the Schwartz class topology implies convergence in the norm of $X$.  Thus every bounded linear functional on $X$ belongs to the class of distributions $\mathcal{S}'(\mathbb{R}^d)$.  

Lemma~\ref{ApproximationForHDR} asserts that the Schwartz class is dense in $X$. To show completeness of $X$, observe that by the $k=0$ case of Theorem~\ref{CorollaryOfChoGuoLee} (i.e. by~\cite{ChoGuoLee}), if $f_n \to f$ in $L_p$, then there exists $s \geq \max(0, 2k-\kappa_p)$ so that $\phi\hat{f}_n \to \phi \hat{f}$ in $H^{-s}(\Sigma)$.  Every Cauchy sequence in $X$ has $\phi\hat{f}_n$ convergent to a limit in the stronger topology of $H^{\ell}(\Sigma)$, and the limit must be $\phi \hat{f}$ as well.

We may identify $f \in X$ with the ordered pair $(f, \phi\hat{f})$.  This gives an isometric embedding of $X$ into $L_p(\mathbb{R}^d) \times H^{\ell}(\Sigma)$.  Its image is closed, so the Hahn--Banach theorem implies that every linear functional $\rho \in X'$ extends to a functional on $L_p(\mathbb{R}^d) \times H^{\ell}(\Sigma)$.  Using Parseval's identity there exists $F_\rho \in L_{p'}(\mathbb{R}^d)$ and $g_\rho \in H^{-\ell}(\Sigma)$ with norms bounded by that of $\rho$ and which satisfy
\begin{equation*}
\rho(f) = \int_{\mathbb{R}^d} F_{\rho} f\, dx + \int_{\mathbb{R}^d} (\phi g_{\rho}\,d\sigma)\check{\phantom{i}}  f \, dx.
\end{equation*}

The defining property of $\HDR(\Sigma, k, s, \ell, p)$ expressed in~\eqref{HDRinequality} is that the linear map $f \mapsto \phi D^\alpha \hat{f}|_\Sigma$ is continuous from $X$ to $H^{-s}(\Sigma)$.  The dual map, taking $g \mapsto (ix)^{\alpha}(\phi g\,d\sigma)\check{\phantom{i}}$ therefore is bounded from $H^s(\Sigma)$ to $X'$, with elements of $X'$ described as above.

\begin{Rem}
Due to the use of the Hahn--Banach theorem in this argument, we do not have a construction for $F_\alpha$ and $g_\alpha$ in Corollary~\ref{HDRDual}.  In fact, it is not proved here that these two functions can be chosen to depend linearly on $g \in H^{s}(\Sigma)$. 
\end{Rem}

\section{Proof of Theorem~\ref{StubbornTheorem}}\label{S3}

\subsection{Pointwise estimates of the kernel}\label{s31}
The quadratic inequality~\eqref{StubbornInequality} is equivalent to its bilinear version
\begin{equation}\label{BilinearStubborn}
\bigg|\Big(\frac{\partial}{\partial r}\Big)^{\beta}\Scalprod{ \frac{\partial^{\alpha}\hat{f}}{\partial \xi_d^{\alpha}}\psi}{ \frac{\partial^{\alpha}\hat{g}}{\partial \xi_d^{\alpha}}\psi}_{\dot{H}^{-\gamma}(\Sigma_r)}\bigg|_{r=0}\bigg| \lesssim \|f\|_{L_p}\|g\|_{L_p}.
\end{equation}
We denote the bilinear form we estimate by~$\B$ and its kernel by~$\K$:
\begin{equation}\label{Kernel}
\B(f,g) =\Big(\frac{\partial}{\partial r}\Big)^{\beta}\Scalprod{ \frac{\partial^{\alpha}\hat{f}}{\partial \xi_d^{\alpha}}\psi}{ \frac{\partial^{\alpha}\hat{g}}{\partial \xi_d^{\alpha}}\psi}_{\dot{H}^{-\gamma}(\Sigma_r)}\bigg|_{r=0} = \iint\limits_{\mathbb{R}^{2d}} f(x)g(y)\K(x,y)\,dx\,dy.
\end{equation}

We also recall the notation
\begin{equation*}
x = (x_{\bar{d}},x_d),\quad \hbox{where}\quad x_{\bar{d}} = (x_1,x_2,\ldots,x_{d-1}), 
\end{equation*}
for~$x\in \mathbb{R}^d$.
\begin{St}\label{Pointwise}
The kernel~$\K$ defined in~\eqref{Kernel} satisfies the bound
\begin{equation}\label{PointwiseBound}
|\K(x,y)| \lesssim (1+|x_d|)^{\alpha - \gamma}(1+|y_d|)^{\alpha - \gamma}(1+|x_d-y_d|)^{\beta + \gamma - \frac{d-1}{2}}, \quad \gamma\in \Big[0,\frac{d-1}{2}\Big).
\end{equation}
\end{St}
\begin{Rem}
One can track the \textup{``}numerology\textup{''} of conditions~\eqref{SISpectralShift},~\eqref{SIGaussian}, and~\eqref{SIKnapp} from this proposition. The boundedness of~$\B$ on~$L_1\times L_1$ is equivalent to the uniform boundedness of~$\K$. The right hand-side of~\eqref{PointwiseBound} is uniformly bounded exactly when these three conditions hold for~$p=1$ \textup(they reflect the behavior of the kernel along the directions~$x_d = y_d$,~$x_d=1$, and~$x_d = -y_d$ respectively\textup).
\end{Rem}
In the case~$\gamma=0$, the inequality~\eqref{PointwiseBound} follows from the standard Van der Corput lemma, because in this case
\begin{equation*}
\K(x,y) = (-1)^{\alpha}(2\pi i)^{2\alpha + \beta}x_d^{\alpha}y_d^{\alpha}(x_d - y_d)^{\beta}\int_{U}e^{2\pi i(\scalprod{x_{\bar{d}}-y_{\bar{d}}}{\zeta} + (x_d -y_d)h(\zeta))}\big|\psi(\zeta)\big|^2\,d\zeta,
\end{equation*} 
the angular brackets denote the standard scalar product in~$\mathbb{R}^{d-1}$.

So, we assume~$\gamma > 0$ in what follows. We start with explicit formulas for the kernel~$\K$:
\begin{multline}\label{FormulaForKernel}
\K(x,y) = (2\pi i x_d)^{\alpha}(-2\pi i y_d)^{\alpha}\Big(\frac{\partial}{\partial r}\Big)^{\beta}\Scalprod{ \hat{\delta}_x \psi}{ \hat{\delta}_y \psi}_{\dot{H}^{-\gamma}(\Sigma_r)}\bigg|_{r=0} \stackrel{\scriptscriptstyle{\eqref{OurSobolev}}}{=}\\ C_{d,\gamma}(2\pi i x_d)^{\alpha}(-2\pi i y_d)^{\alpha}\Big(\frac{\partial}{\partial r}\Big)^{\beta}\bigg[\iint\limits_{U\times U}e^{2\pi i(\scalprod{\zeta}{x_{\bar{d}}}-\scalprod{\eta}{y_{\bar{d}}} + x_d (h(\zeta)+r)- y_d (h(\eta)+r))}\psi(\zeta)\overline{\psi(\eta)}|\zeta-\eta|^{2\gamma - d-1}\,d\zeta\,d\eta\bigg]\bigg|_{r=0}=\\ 
C_{d,\gamma}(-1)^{\alpha}(2\pi i)^{2\alpha + \beta}x_d^{\alpha}y_d^{\alpha}(x_d - y_d)^{\beta}\iint\limits_{U\times U}e^{2\pi i(\scalprod{\zeta}{x_{\bar{d}}}-\scalprod{\eta}{y_{\bar{d}}} + x_d h(\zeta)- y_d h(\eta))}\psi(\zeta)\overline{\psi(\eta)}|\zeta-\eta|^{2\gamma - d-1}\,d\zeta\,d\eta.
\end{multline}

We want to pass to the dyadic version of the Bessel semi-norm, namely,
\begin{equation*}
\|f\|_{\dot{H}^{-\gamma}(\Sigma_r)}^2 = C_{d,\gamma}\sum\limits_{k \geq 0}2^{k(d-1-2\gamma)}\iint\limits_{U\times U}f(\zeta,h_r(\zeta))\overline{f(\eta,h_r(\eta))}\varphi(2^{k}|\zeta - \eta|)\,d\zeta\,d\eta, \quad \gamma\in \Big(0,\frac{d-1}{2}\Big),
\end{equation*} 
based on the formula
\begin{equation}\label{DyadicPotential}
|\zeta - \eta|^{2\gamma-d+1} = \sum\limits_{k \geq 0}2^{k(d-1-2\gamma)}\varphi(2^{k}|\zeta - \eta|).
\end{equation}
The function~$\varphi \in C_0^{\infty}(\mathbb{R})$ is supported outside zero and non-negative.

We substitute formula~\eqref{DyadicPotential} into~\eqref{FormulaForKernel} and split~$\K$ into a dyadic sum:
\begin{equation*}
\K(x,y) = C_{d,\gamma}(-1)^{\alpha}(2\pi i)^{2\alpha + \beta}\sum\limits_{k \geq 0} I_k(x,y),
\end{equation*}
where
\begin{equation*}
I_k(x,y) = 2^{k(d-1-2\gamma)}x_d^{\alpha}y_d^{\alpha}(x_d - y_d)^{\beta}\iint\limits_{U\times U}e^{2\pi i(\scalprod{\zeta}{x_{\bar{d}}}-\scalprod{\eta}{y_{\bar{d}}} + x_d h(\zeta)- y_d h(\eta))}\psi(\zeta)\overline{\psi(\eta)}\varphi(2^k|\zeta - \eta|)\,d\zeta\,d\eta.
\end{equation*}

\begin{Le}\label{FirstEstimate}
For any~$x$,~$y$ and any~$k \geq 0$,
\begin{equation*}
|I_k(x,y)| \lesssim 2^{k(d-1-2\gamma)}(1+|x_d|)^{\alpha-\frac{d-1}{2}}(1+|y_d|)^{\alpha-\frac{d-1}{2}}|x_d-y_d|^{\beta}.
\end{equation*}  
\end{Le}
\begin{proof}
The integral in the formula for~$I_k$ may be thought of as the $(2d-2)$-dimensional Fourier integral:
\begin{equation*}
I_k(x,y) = 2^{k(d-1-2\gamma)}x_d^{\alpha}y_d^{\alpha}(x_d - y_d)^{\beta}\mathcal{F}_{(\zeta,\eta)\mapsto (x_{\bar{d}},y_{\bar{d}})}\bigg[e^{2\pi i(x_d h(\zeta)- y_d h(\eta))}\psi(\zeta)\overline{\psi(\eta)}\varphi(2^k|\zeta - \eta|)\bigg](x_{\bar{d}},-y_{\bar{d}}).
\end{equation*}
It suffices to prove the inequality
\begin{equation}\label{UniformEst}
\bigg\|\mathcal{F}_{(\zeta,\eta)\mapsto (x_{\bar{d}},y_{\bar{d}})}\bigg[e^{2\pi i(x_d h(\zeta)- y_d h(\eta))}\psi(\zeta)\overline{\psi(\eta)}\varphi(2^k|\zeta - \eta|)\bigg]\bigg\|_{L_{\infty}} \lesssim (1+|x_d|)^{-\frac{d-1}{2}}(1+|y_{d}|)^{-\frac{d-1}{2}}.
\end{equation} 
We represent the function we apply the Fourier transform to as a product of two functions
\begin{equation*}
(\zeta,\eta)\mapsto e^{2\pi i(x_d h(\zeta)- y_d h(\eta))}\psi(\zeta)\overline{\psi(\eta)} \quad \hbox{and}\quad (\zeta,\eta)\mapsto \varphi(2^k|\zeta - \eta|).
\end{equation*}
By the Van der Corput lemma, the Fourier transform of the first function is uniformly (in $(x_{\bar{d}},y_{\bar{d}})$) bounded by the right hand-side of~\eqref{UniformEst}. It remains to notice that the Fourier transform of the second function is a complex measure whose total variation is bounded uniformly in~$k$. This is easiest to see by making a
linear change of variables from~$(\zeta,\eta)$ to~$(\zeta, \zeta - \eta)$.
\end{proof}

Define the number~$k_0 \geq 0$ by the rule
\begin{equation}\label{k_0}
2^{2k_0} = \frac{(1+|x_d|)(1+|y_d|)}{1+|x_d-y_d|}.
\end{equation}
\begin{Le}
\label{SecondEstimate}
For any~$x$,~$y$ and any~$k \geq k_0$,
\begin{equation*}
|I_k(x,y)| \lesssim 2^{-2k\gamma}(1+|x_d|)^{\alpha}(1+|y_d|)^{\alpha}(1+|x_d-y_d|)^{\beta-\frac{d-1}{2}}.
\end{equation*}
\end{Le}
\begin{proof}
Let~$|y_d| \geq |x_d|$.
It suffices to prove the estimate
\begin{equation*}
\bigg|\iint\limits_{U\times U}e^{2\pi i(\scalprod{\zeta}{x_{\bar{d}}}-\scalprod{\eta}{y_{\bar{d}}} + x_d h(\zeta)- y_d h(\eta))}\psi(\zeta)\overline{\psi(\eta)}\varphi(2^k|\zeta - \eta|)\,d\zeta\,d\eta\bigg| \lesssim 2^{-k(d-1)}(1+|x_d - y_d|)^{-\frac{d-1}{2}}.
\end{equation*}
We introduce new variables~$(\theta,\eta) = (2^{k}(\zeta - \eta),\eta)$ and disregard oscillations in the~$\theta$ variable:
\begin{multline*}
\bigg|\iint\limits_{U\times U}e^{2\pi i(\scalprod{\zeta}{x_{\bar{d}}}-\scalprod{\eta}{y_{\bar{d}}} + x_d h(\zeta)- y_d h(\eta))}\psi(\zeta)\overline{\psi(\eta)}\varphi(2^k|\zeta - \eta|)\,d\zeta\,d\eta\bigg| = \\
2^{-k(d-1)}\bigg|\iint\limits_{\mathbb{R}^{2d-2}}e^{2\pi i(\scalprod{\eta + 2^{-k}\theta}{x_{\bar{d}}}-\scalprod{\eta}{y_{\bar{d}}} + x_d h(\eta + 2^{-k}\theta)- y_d h(\eta))}\psi(\eta + 2^{-k}\theta)\overline{\psi(\eta)}\varphi(|\theta|)\,d\theta\,d\eta\bigg| \lesssim \\
2^{-k(d-1)}\sup_{|\theta| \lesssim 1}\bigg|\int\limits_{\mathbb{R}^{d-1}}e^{2\pi i(\scalprod{\eta}{x_{\bar{d}}-y_{\bar{d}}} + x_d h(\eta + 2^{-k}\theta)- y_d h(\eta))}\psi(\eta + 2^{-k}\theta)\overline{\psi(\eta)}\,d\eta\bigg|.
\end{multline*}
It remains to prove
\begin{equation*}
\bigg|\int\limits_{\mathbb{R}^{d-1}}e^{2\pi i(\scalprod{\eta}{x_{\bar{d}}-y_{\bar{d}}} + x_d h(\eta + 2^{-k}\theta)- y_d h(\eta))}\psi(\eta + 2^{-k}\theta)\overline{\psi(\eta)}\,d\eta\bigg| \lesssim (1+|x_d - y_d|)^{-\frac{d-1}{2}},\quad |\theta| \lesssim 1.
\end{equation*}
The function~$\psi(\cdot + 2^{-k}\theta)$ is uniformly (with respect to~$k$ and~$\theta$) bounded in any Schwartz norm, so its Fourier transform is an~$L_1$-function whose norm is bounded independently of~$k$ and~$\theta$. Thus, it suffices to prove
\begin{equation}\label{Pre-Littman}
\bigg|\int\limits_{\mathbb{R}^{d-1}}e^{2\pi i(\scalprod{\eta}{x_{\bar{d}}-y_{\bar{d}}} + x_d h(\eta + 2^{-k}\theta)- y_d h(\eta))}\overline{\psi(\eta)}\,d\eta\bigg| \lesssim (1+|x_d - y_d|)^{-\frac{d-1}{2}}
\end{equation}
uniformly in~$x_{\bar{d}}, y_{\bar{d}}$.
This inequality is trivial if~$|x_d - y_d|\lesssim 1$, so we assume the quantity~$|x_d - y_d|$ is sufficiently large. We represent the non-linear part of the phase function as
\begin{equation*}
x_d h(\eta + 2^{-k}\theta)- y_d h(\eta) = (x_d-y_d) \Phi_{\theta,x_d,y_d}(\eta),
\end{equation*}
where
\begin{equation*}
\Phi_{\theta,x_d,y_d}(\eta) = h(\eta) + \frac{x_d}{x_d - y_d}\big(h(\eta + 2^{-k}\theta) - h(\eta)\big) = h(\eta) + O\Big(\frac{2^{-k}|x_d|}{|x_d - y_d|}\Big)
\end{equation*}
since~$|\theta|\lesssim 1$.
Note that
\begin{equation*}
\frac{2^{-k}|x_d|}{|x_d - y_d|} \leq \frac{2^{-k_0}|x_d|}{|x_d - y_d|}  \stackrel{\scriptscriptstyle\eqref{k_0}}{=} \frac{|x_d|\sqrt{1+|x_d-y_d|}}{|x_d - y_d|\sqrt{1+|x_d|}\sqrt{1+|y_d|}} \lesssim \frac{1}{\sqrt{1+|x_d - y_d|}} \ll 1,
\end{equation*} 
when~$|y_d| \geq |x_d|$,~$|x_d - y_d|$ is sufficiently large, and~$k \geq k_0$. In particular, the Hessians of the functions in the family~$\{\Phi_{\theta,x_d,y_d}\}_{\theta,x_d,y_d}$ take the form~$\frac{\partial^2 h}{\partial \eta^2}(\eta) + O(\frac{2^{-k}|x_d|}{|x_d-y_d|})$ and are uniformly invertible. By similar reasons, the functions in the family~$\{\Phi_{\theta,x_d,y_d}\}_{\theta,x_d,y_d}$ are uniformly bounded in any Schwartz norm. 
The version of Littman's lemma from~\cite{Domar} leads to~\eqref{Pre-Littman}.
\end{proof}
\begin{proof}[Proof of Proposition~\textup{\ref{Pointwise}}.]
We use Lemmas~\ref{FirstEstimate} and~\ref{SecondEstimate} (the case~$\gamma=0$  has already been considered, so we assume~$\gamma > 0$ here):
\begin{multline*}
|\K(x,y)| \lesssim C_{d,\gamma}\sum\limits_{k \geq 0}|I_k(x,y)| \lesssim\\ \sum\limits_{k \leq k_0} 2^{(d-1-2\gamma)k}(1+|x_d|)^{\alpha-\frac{d-1}{2}}(1+|y_d|)^{\alpha-\frac{d-1}{2}}|x_d-y_d|^{\beta} + \sum\limits_{k \geq k_0} 2^{-2k\gamma}(1+|x_d|)^{\alpha}(1+|y_d|)^{\alpha}(1+|x_d-y_d|)^{\beta-\frac{d-1}{2}}\lesssim\\
2^{(d-1-2\gamma)k_0}(1+|x_d|)^{\alpha-\frac{d-1}{2}}(1+|y_d|)^{\alpha-\frac{d-1}{2}}|x_d-y_d|^{\beta} +  2^{-2k_0\gamma}(1+|x_d|)^{\alpha}(1+|y_d|)^{\alpha}(1+|x_d-y_d|)^{\beta-\frac{d-1}{2}}\lesssim\\
(1+|x_d|)^{\alpha - \gamma}(1+|y_d|)^{\alpha - \gamma}(1+|x_d-y_d|)^{\beta + \gamma - \frac{d-1}{2}}.
\end{multline*}
\end{proof}

\subsection{Interpolation}\label{s32}
To prove the ``if'' part of Theorem~\ref{StubbornTheorem} for the case~$p > 1$, we will have to work with ``slices'' of the kernel~$\K$. For any~$x_d$ and~$y_d$, define the kernel~$\K_{x_d,y_d}\colon \mathbb{R}^{2d-2}\to \mathbb{C}$ by the formula
\begin{equation*}
\K_{x_d,y_d}(x_{\bar{d}},y_{\bar{d}}) = \K(x,y),\quad x = (x_{\bar{d}},x_d), \ y = (y_{\bar{d}},y_d).
\end{equation*}
Define the bilinear forms~$\B_{x_d,y_d}$ accordingly
\begin{equation*}
\B_{x_d,y_d}[f,g] = \iint\limits_{\mathbb{R}^{2d-2}}f(x_{\bar{d}})g(y_{\bar{d}})\K_{x_d,y_d}(x_{\bar{d}},y_{\bar{d}})\,d x_{\bar{d}}\,d y_{\bar{d}},
\end{equation*}
here~$f$ and~$g$ are functions on~$\mathbb{R}^{d-1}$.
Proposition~\ref{Pointwise} now may be restated as
\begin{equation}\label{Pointwise2}
\big\|\B_{x_d,y_d}\big\|_{L_1\times L_1} \lesssim (1+|x_d|)^{\alpha - \gamma}(1+|y_d|)^{\alpha - \gamma}(1+|x_d-y_d|)^{\beta + \gamma - \frac{d-1}{2}}, \quad \gamma\in \Big[0,\frac{d-1}{2}\Big). 
\end{equation}
\begin{Le}\label{TheOtherEndpoint}
For any~$\gamma\in [0,\frac{d-1}{2})$,
\begin{equation*}
\big\|\B_{x_d,y_d}\big\|_{L_{\frac{2d-2}{d-1+2\gamma}}\times L_{\frac{2d-2}{d-1+2\gamma}}} \lesssim (1+|x_d|)^{\alpha - \gamma}(1+|y_d|)^{\alpha - \gamma}(1+|x_d-y_d|)^{\beta}.
\end{equation*}
\end{Le}
\begin{proof}
Let~$\lambda_{x_d}$ and~$\lambda_{y_d}$ be the Lebesgue measures on the hyperplanes
\begin{equation*}
\{w \in \mathbb{R}^d\mid w_d = x_d\} \quad \hbox{and}\quad \{w \in \mathbb{R}^d\mid w_d = y_d\}. 
\end{equation*}
Then,
\begin{equation*}
\B_{x_d,y_d}[f,g] = \B[f\,d\lambda_{x_d},g\, d\lambda_{y_d}]
\end{equation*}
if we interpret~$f$ and~$g$ as functions of~$d$ variables that do not depend on~$x_d$.
With this formula in hand, we may re-express~$\B_{x_d,y_d}$:
\begin{multline}\label{RepresentationForB}
\B_{x_d,y_d}[f,g] = \Big(\frac{\partial}{\partial r}\Big)^{\beta}\Scalprod{\frac{\partial^{\alpha}[\hat{f}(\xi)e^{2\pi i x_d\xi_d}]}{\partial \xi_d^{\alpha}}\psi(\xi)}{\frac{\partial^{\alpha}[\hat{g}(\xi)e^{2\pi i y_d\xi_d}]}{\partial \xi_d^{\alpha}}\psi(\xi)}_{\dot{H}^{-\gamma}(\Sigma_r)}\bigg|_{r=0} = \\
(2\pi i x_d)^{\alpha}(-2\pi i y_d)^{\alpha}(2\pi i (x_d-y_d))^{\beta}\scalprod{\hat{f}e^{2\pi i x_dh(\cdot)}\psi}{\hat{g}e^{2\pi i y_dh(\cdot)}\psi}_{\dot{H}^{-\gamma}(\Sigma_0)} = \\(-1)^{\alpha}(2\pi i)^{2\alpha + \beta}x_d^{\alpha}y_d^{\alpha}(x_d-y_d)^{\beta}\int\limits_{\mathbb{R}^{d-1}}\Big[f*\mathcal{F}_{\zeta\mapsto z}\big[e^{2\pi i x_dh(\zeta)}\psi(\zeta)\big]\Big](z)\cdot\overline{\Big[g*\mathcal{F}_{\zeta\mapsto z}\big[e^{2\pi i y_dh(\zeta)}\psi(\zeta)\big]\Big](z)}\cdot|z|^{-2\gamma}\,dz.
\end{multline}
Therefore, it suffices to prove the bound
\begin{multline}\label{AlmostDone}
\bigg|\int\limits_{\mathbb{R}^{d-1}}\Big[f*\mathcal{F}_{\zeta\mapsto z}\big[e^{2\pi i x_dh(\zeta)}\psi(\zeta)\big]\Big](z)\cdot\overline{\Big[g*\mathcal{F}_{\zeta\mapsto z}\big[e^{2\pi i y_dh(\zeta)}\psi(\zeta)\big]\Big](z)}\cdot|z|^{-2\gamma}\,dz\bigg| \lesssim\\ (1+|x_d|)^{-\gamma}(1+|y_d|)^{-\gamma}\|f\|_{L_{\frac{2d-2}{d-1+2\gamma}}}\|g\|_{L_{\frac{2d-2}{d-1+2\gamma}}}.
\end{multline}
We postulate the inequality
\begin{equation}\label{LorentzBound}
\Big\|f*\mathcal{F}_{\zeta\mapsto z}\big[e^{2\pi i x_dh(\zeta)}\psi(\zeta)\big]\Big\|_{L_{\frac{2d-2}{d-1-2\gamma},2}} \lesssim (1+|x_d|)^{-\gamma}\|f\|_{L_{\frac{2d-2}{d-1+2\gamma}}}.
\end{equation}
The space on the left is the Lorentz space, see~\cite{BerghLofstrom} for definitions.
Inequality~\eqref{LorentzBound} immediately leads to~\eqref{AlmostDone}:
\begin{multline*}
\bigg|\int\limits_{\mathbb{R}^{d-1}}\Big[f*\mathcal{F}_{\zeta\mapsto z}\big[e^{2\pi i x_dh(\zeta)}\psi(\zeta)\big]\Big](z)\cdot\overline{\Big[g*\mathcal{F}_{\zeta\mapsto z}\big[e^{2\pi i y_dh(\zeta)}\psi(\zeta)\big]\Big](z)}\cdot|z|^{-2\gamma}\,dz\bigg| \lesssim \\
\Big\|f*\mathcal{F}_{\zeta\mapsto z}\big[e^{2\pi i x_dh(\zeta)}\psi(\zeta)\big]\Big\|_{L_{\frac{2d-2}{d-1-2\gamma},2}}\cdot\Big\|g*\mathcal{F}_{\zeta\mapsto z}\big[e^{2\pi i y_dh(\zeta)}\psi(\zeta)\big]\Big\|_{L_{\frac{2d-2}{d-1-2\gamma},2}}\cdot \||z|^{-2\gamma}\|_{L_{\frac{d-1}{2\gamma},\infty}} \stackrel{\scriptscriptstyle\eqref{LorentzBound}}{\lesssim}\\
(1+|x_d|)^{-\gamma}(1+|y_d|)^{-\gamma}\|f\|_{L_{\frac{2d-2}{d-1+2\gamma}}}\|g\|_{L_{\frac{2d-2}{d-1+2\gamma}}}.
\end{multline*}
We are required to prove~\eqref{LorentzBound}. Let us denote the operator we want to estimate by~$\CONV_{x_d}$:
\begin{equation*}
\CONV_{x_d}[f] = f*\mathcal{F}_{\zeta\mapsto z}\big[e^{2\pi i x_dh(\zeta)}\psi(\zeta)\big].
\end{equation*}
By the Plancherel theorem,
\begin{equation*}\label{Plancherel}
\|\CONV_{x_d}\|_{L_2\to L_2} \lesssim 1.
\end{equation*}
By the Van der Corput lemma,
\begin{equation*}\label{VdC}
\|\CONV_{x_d}\|_{L_1\to L_{\infty}} \lesssim (1+|x_d|)^{-\frac{d-1}{2}}.
\end{equation*}
The real interpolation formulas (see~\cite{BerghLofstrom}, \S 5.3)
\begin{equation*}
[L_1,L_2]_{\frac{2\gamma}{d-1},\frac{2d-2}{d-1+2\gamma}} = L_{\frac{2d-2}{d-1+2\gamma}}; \quad [L_{\infty}, L_2]_{\frac{2\gamma}{d-1},\frac{2d-2}{d-1+2\gamma}} = L_{\frac{2d-2}{d-1-2\gamma},\frac{2d-2}{d-1+2\gamma}} 
\end{equation*}
lead to the inequality
\begin{equation*}
\|\CONV_{x_d}\|_{L_{\frac{2d-2}{d-1+2\gamma}}\to L_{\frac{2d-2}{d-1-2\gamma},2}} \lesssim \|\CONV_{x_d}\|_{L_{\frac{2d-2}{d-1+2\gamma}}\to L_{\frac{2d-2}{d-1-2\gamma},\frac{2d-2}{d-1+2\gamma}}} \lesssim (1+|x_d|)^{-\gamma},
\end{equation*}
which is exactly~\eqref{LorentzBound}.
\end{proof}

Interpolation between~\eqref{Pointwise2} and Lemma~\ref{TheOtherEndpoint} leads to the inequality
\begin{equation}\label{Pre-FinalInequality}
\|\B_{x_d,y_d}\|_{L_p\times L_p} \lesssim (1+|x_d|)^{\alpha-\gamma}(1+|y_d|)^{\alpha-\gamma}(1+|x_d-y_d|)^{\beta + \gamma - \frac{d-1}{p} + \frac{d-1}{2}}
\end{equation}
for~$p \in [1,\frac{2d-2}{d-1+2\gamma}]$. Let us restrict our attention to this case for a while. To finish the proof of Theorem~\ref{StubbornTheorem}, we invoke a version of the Stein--Weiss inequality (Theorem~\ref{SteinWeissModified} in Subsection~\ref{s62} below, see Remark~\ref{Addition} there as well):
\begin{multline*}
\big|\B[f,g]\big| = \bigg|\iint\limits_{\mathbb{R}\times \mathbb{R}} \B_{x_d,y_d}\big[f(\cdot,x_d),g(\cdot,y_d)\big]\,dx_d\,dy_d\bigg| \stackrel{\scriptscriptstyle\eqref{Pre-FinalInequality}}{\lesssim}\\ \iint\limits_{\mathbb{R}\times\mathbb{R}} (1+|x_d|)^{\alpha-\gamma}(1+|y_d|)^{\alpha-\gamma}(1+|x_d-y_d|)^{\beta + \gamma - \frac{d-1}{p} + \frac{d-1}{2}} \|f(\cdot,x_d)\|_{L_p(\mathbb{R}^{d-1})}\|g(\cdot,y_d)\|_{L_p(\mathbb{R}^{d-1})}\,dx_d\,dy_d \stackrel{\hbox{\tiny Th.}\scriptscriptstyle\ref{SteinWeissModified}}{\lesssim}\\ \|f\|_{L_p}\|g\|_{L_p}
\end{multline*}
provided~$a = \gamma-\alpha$ and~$b = -\beta - \gamma - \frac{d-1}{2} + \frac{d-1}{p}$ satisfy the requirements of Theorem~\ref{SteinWeissModified}. The inequality~\eqref{SISpectralShift} leads to~$a \geq 0$, the requirement~\eqref{SIGaussian} leads to~$a+b \geq 1-\frac{1}{p}$ (with the same exclusion of the endpoint case if~$p> 1$), and~\eqref{SIKnapp} gives~$2a+b \geq 2-\frac{2}{p}$. The case~$b=1$ and~$p=2$ is impossible ($\beta+\gamma$ is negative in this case). The ``if'' part of Theorem~\ref{StubbornTheorem} is proved in the case~$p \in [1,\frac{2d-2}{d-1+2\gamma}]$.

To deal with the remaining case, we start from the estimate
\begin{equation*}
\big\|\B_{x_d,y_d}\big\|_{L_{2}\times L_{2}} \lesssim (1+|x_d|)^{\alpha}(1+|y_d|)^{\alpha}(1+|x_d-y_d|)^{\beta},
\end{equation*}
which follows from the representation~\eqref{RepresentationForB}; we use the trivial inequality
\begin{equation*}
\Big\|f*\mathcal{F}_{\zeta\mapsto z}\big[e^{2\pi i x_dh(\zeta)}a(\zeta)\big]\Big\|_{L_{\infty}} \lesssim \|f\|_{L_2}\big\|e^{2\pi i x_dh(\zeta)}\psi(\zeta)\big\|_{L_2}\lesssim \|f\|_{L_2}.
\end{equation*}
We interpolate this bound with Lemma~\ref{TheOtherEndpoint}:
\begin{equation*}
\big\|\B_{x_d,y_d}\big\|_{L_{p}\times L_{p}} \lesssim \Big(1+|x_d|\Big)^{\alpha - (d-1)(\frac{1}{p} - \frac12)}\Big(1+|y_d|\Big)^{\alpha - (d-1)(\frac{1}{p} - \frac12)}(1+|x_d-y_d|)^\beta,\quad p \in \Big[\frac{2d-2}{d-1+2\gamma},2\Big].
\end{equation*}
We invoke Theorem~\ref{SteinWeissModified} with~$a = (d-1)(\frac{1}{p} - \frac12) - \alpha$ and~$b = -\beta$. Since~$b \leq 0$, the condition~$a+b >1-\frac{1}{p}$ is stronger than~$2a+b \geq 2-\frac{2}{p}$. The condition~$a+b > 1-\frac{1}{p}$ is exactly~\eqref{SIGaussian}. The condition~$a > 0$ also follows from it:
\begin{equation*}
\alpha \leq \alpha + \beta < \frac{d}{p} - \frac{d+1}{2} = \frac{d-1}{p} - \frac{d-1}{2} + \frac{1}{p} - 1 \leq \frac{d-1}{p} - \frac{d-1}{2}.
\end{equation*}
The ``if'' part of Theorem~\ref{StubbornTheorem} is now proved.

\begin{Rem}
Note that we did not use that~$\alpha$ or~$\beta$ are integer provided we define our bilinear form by~\eqref{FormulaForKernel}.
\end{Rem}

\subsection{Strichartz estimates}\label{s33}
With the same method as in the previous section, we can get a collection of sharp (up to the endpoint) Strichartz estimates. For that we need the mixed norm spaces~$L_r(L_p)$:
\begin{equation}\label{MixedNorm}
\big\|g\big\|_{L_r(L_p)} = \big\|\big\|g\big\|_{L_p(x)}\big\|_{L_r(t)} = \bigg(\int\limits_{\mathbb{R}}\Big(\int\limits_{\mathbb{R}^{d-1}}|g(x,t)|^{p}\,dx\Big)^{\frac{r}{p}}\,dt\bigg)^{\frac{1}{r}},\quad g\colon \mathbb{R}^d \to \mathbb{R}.
\end{equation} 
We also use Theorem~\ref{SteinWeissp>2} here in order to work with the cases~$r > 2$ as well. This provides some new information even in the case~$\alpha = \beta = 0$ considered in~\cite{ChoGuoLee}. The cases~$r > 2$ were excluded in that paper and it is not clear whether the methods of~\cite{ChoGuoLee} work in this situation. 
\begin{Th}\label{StrichartzTheorem}
The inequality
\begin{equation*}
\Big|\B(f,g)\Big|\lesssim \|f\|_{L_r(L_p)}\|g\|_{L_r(L_p)}
\end{equation*}
holds true if
\begin{enumerate}
\item $r \in [1,2]$ and
\begin{itemize}
\item $p \le 2$\textup;
\item $\gamma\geq \alpha$\textup;
\item $\alpha + \beta < \frac{d-1}{p} + \frac{1}{r} - \frac{d+1}{2}$, with equality permitted if $r=1$\textup;
\item $2\alpha +\beta -\gamma <  \frac{d-1}{p} + \frac{2}{r} - \frac{d+3}{2}$, with equality also permitted if $r<2$ or $\gamma > \alpha$\textup;
\end{itemize}
\item $r \in (2,\infty]$ and
\begin{itemize}
\item $p \le 2$\textup;
\item $\gamma-\alpha > \frac12 - \frac{1}{r}$\textup;
\item $\alpha + \beta < \frac{d-1}{p} + \frac{1}{r} - \frac{d+1}{2}$\textup;
\item $2\alpha +\beta -\gamma < \frac{d-1}{p} + \frac{2}{r} - \frac{d+3}{2}$\textup.
\end{itemize}
\end{enumerate}
\end{Th}

The proof is a direct application of Theorems~\ref{SteinWeissModified} and~\ref{SteinWeissp>2}. Consider the case~$p \in [1,\frac{2d-2}{d-1+2\gamma}].$  Set~$a = \gamma - \alpha$, $b = -\beta - \gamma -\frac{d-1}{2} + \frac{d-1}{p}$, and $p = r$ (that is, the value of $p$ in those theorems is replaced by $r$). We note that the conditions of Theorem~\ref{SteinWeissModified} can be summarized as $a \geq 0$, $a + b > 1-\frac{1}{p}$ and $2a+b \geq 2 - \frac{2}{p}$.  When $p=1$, combinations with $a+b \geq 0 $ are also accepted, and when $p=2$ the case $a=0$, $b=1$ is excluded.

The three conditions stated in the case $r \in [1,2]$ are equivalent to $a \geq 0$, $a+ b > 1- \frac{1}{r}$, and $2a+b > 2- \frac{2}{r}$, respectively.   The three conditions stated in the case $r \in (2,\infty]$ are equivalent to the conditions in Theorem~\ref{SteinWeissp>2}, namely $a > \frac12 - \frac1r$, $a+b > 1-\frac{1}{r}$, and $2a+b > 2 - \frac{2}{r}$. 

Consider the case~$p \in [\frac{2d-2}{d-1+2\gamma},2]$ and set~$a=(d-1)(\frac{1}{p} - \frac12) - \alpha$,~$b = -\beta$, and~$p=r$ in the same sense as above. Since~$b \leq 0$, the condition~$a+b > \frac1p$ is stronger than~$2a+b \geq 2 - \frac{2}{p}$. The condition~$a+b > \frac{1}{p}$ is equivalent to~$\alpha + \beta < \frac{d-1}{p} + \frac{1}{r} - \frac{d+1}{2}$ with equality permitted if~$r=1$. In the case~$r\leq 2$, the requirement~$a \geq 0$ is rewritten as~$\alpha < \frac{d-1}{p} - \frac{d-1}{2}$. It also follows from~$\alpha + \beta \leq \frac{d-1}{p} + \frac{1}{r} - \frac{d+1}{2}$. The condition~$a > \frac12 - \frac{1}{r}$ arising in the case~$r \geq 2$ follows from the same inequality.

\section{Robust estimates}\label{S4}

\subsection{Introduction to ``numerology''}\label{s41}
\begin{Rem}
We are mostly interested in the case~$k > s$ in~\eqref{HDRinequality}. We claim that in the \textup{``}subcritical\textup{''} case~$k \leq s$, the second term on the right hand-side of this inequality is unnecessary. If~$k \leq s$ and~$\HDR(\Sigma, k,s, \ell,p)$ is true, then a simpler inequality
\begin{equation}\label{HDRWithOneTerm}
\|\phi \nabla^k \hat{f}\|_{H^{-s}(\Sigma)} \lesssim\|f\|_{L_p(\mathbb{R}^d)}
\end{equation}
also holds true. Indeed, if~$\HDR(\Sigma, k,s, \ell,p)$ is true, then~\eqref{HDRSurfaceCondition} and~\eqref{HDRKnapp} are valid. However, in this case, these conditions are also sufficient for~\eqref{HDRWithOneTerm} to be true \textup(see Theorem~\textup{\ref{StubbornTheorem}} and Figure~\textup{\ref{fig:Klessgamma}} below\textup).
\end{Rem}
\begin{figure}[h]
\includegraphics[height=9.5cm]{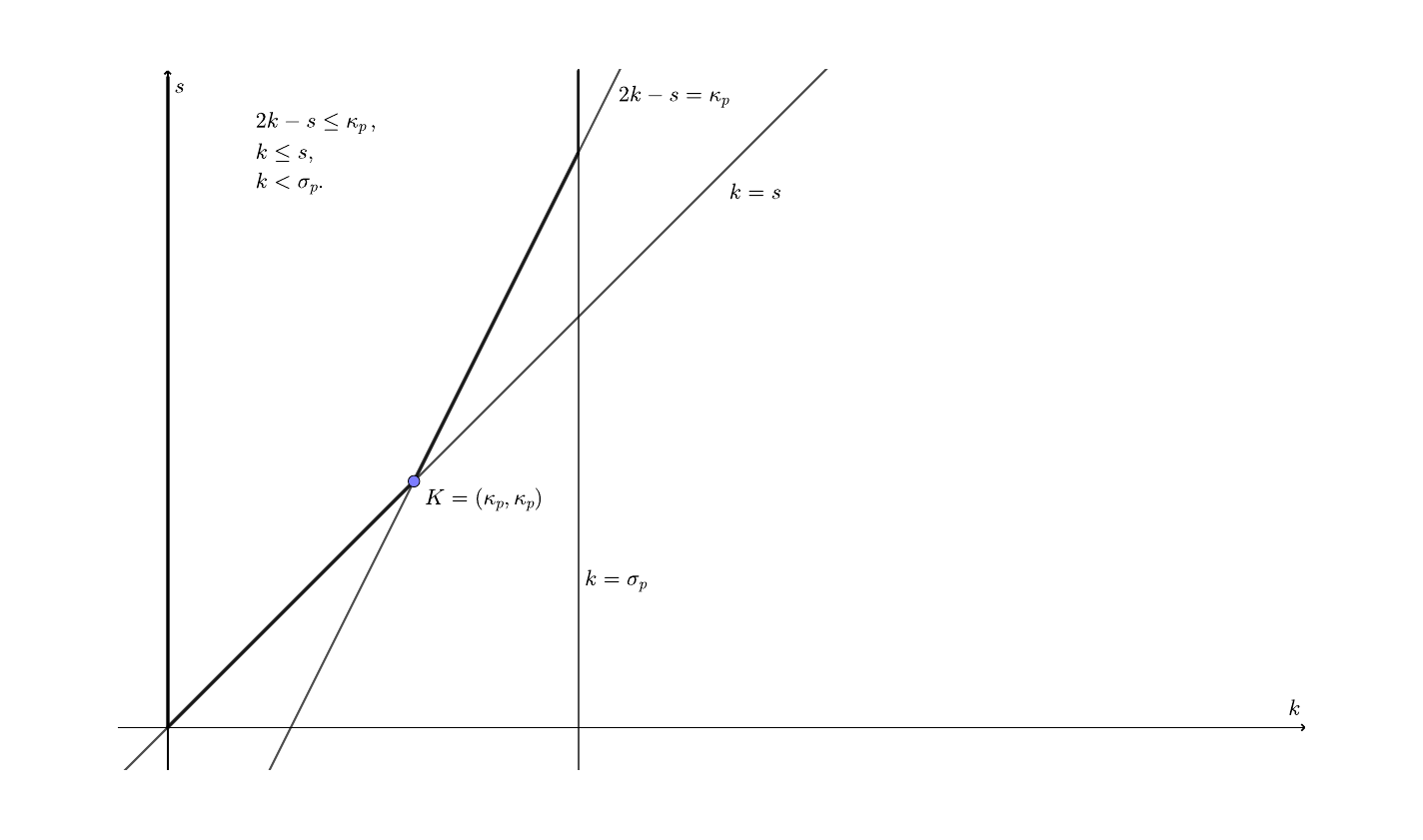}
\caption{Diagram for the case~$k\leq s$.}
\label{fig:Klessgamma}
\end{figure}
\begin{Rem}
In the \textup{``}supercritical\textup{''} case~$k > s$, the condition~\eqref{HDRSurfaceCondition} follows from~\eqref{HDRKnapp} since~$\kappa_p \leq \sigma_p$ in this case \textup(see Figure~\textup{\ref{fig:Klessgamma}} as well\textup). Note also that~\eqref{HDRShiftedKnapp} is equivalent to
\begin{equation}\label{ModifiedShiftedKnapp}
s \geq \Big(1+\frac{\ell}{\kappa_p}\Big)k - \ell.
\end{equation}
This inequality, in its turn, leads to~\eqref{HDRShiftCondition} provided~$\kappa_p \geq 0$ \textup(which is true by~\eqref{HDRKnapp}\textup). Thus, in the case~$k> s$, the conditions in Proposition~\textup{\ref{SufficientConditionsHDR}} are reduced to~\eqref{HDRSigmaLpShiftCondition}, \eqref{ModifiedShiftedKnapp}, and~\eqref{HDRKnapp}.
\end{Rem}
In our proof, the parameters~$k$ and~$s$ will be varied, however,~$\ell$ and~$p$ will be steady. It appears convenient to draw diagrams of admissible pairs~$(k,s)$. We have already drawn such a diagram for the case~$k \leq s$ (Figure~\ref{fig:Klessgamma}). For our first attempt to the ``numerology'', we neglect the integer nature of~$k$ and imagine this parameter is real positive. We have three inequalities in the subcritical case: $k\leq s$, \eqref{HDRSurfaceCondition}, and~\eqref{HDRKnapp}. The cases of equality correspond to lines on the diagram, and all three inequalities are satisfied inside the domain bounded by the bold broken line. We also note that the lines~$2k-s = \kappa_p$ and~$k=s$ intersect at the point~$(\kappa_p,\kappa_p)$, which we denote by~$K$.

Now we pass to the ``supercritical'' case~$k > s$. We need to draw two additional lines~$k=s+1$ and
\begin{equation}\label{LineKL}
s = \Big(1+\frac{\ell}{\kappa_p}\Big)k - \ell,
\end{equation}
which correspond to~\eqref{HDRSigmaLpShiftCondition} and~\eqref{HDRShiftedKnapp} respectively.
The structure of the domain of admissible parameters will depend on the mutual disposition of the these two lines and the line~$s = 2k-\kappa_p$. Before we classify the cases of disposition, we note that the line~\eqref{LineKL} passes through~$K$. There is one more nice point lying on it: the point~$L = (0,-\ell)$. We will consider the cases~$\kappa_p < 2$ and~$\kappa_p \geq 2$ separately.

\paragraph{\bf Case~$\kappa_p < 2$.} In this case, the condition~\eqref{HDRSigmaLpShiftCondition} is unnecessary, it follows from~\eqref{HDRKnapp} and~$s \geq 0$. This case, in its turn, is naturally split into subcases~$\ell \leq \kappa_p$ (see Figure~\ref{fig:Basic_KappaLessTwo_Ell_Small}, note that the broken line has a non-trivial angle at~$K$) and~$\ell > \kappa_p$ (see Figure~\ref{fig:Basic_KappaLessTwo_Ell_Large}), note that~\eqref{HDRShiftedKnapp} follows from~\eqref{HDRKnapp} when~$\ell \geq \kappa_p$.
\begin{figure}[h]
\includegraphics[height=9.5cm]{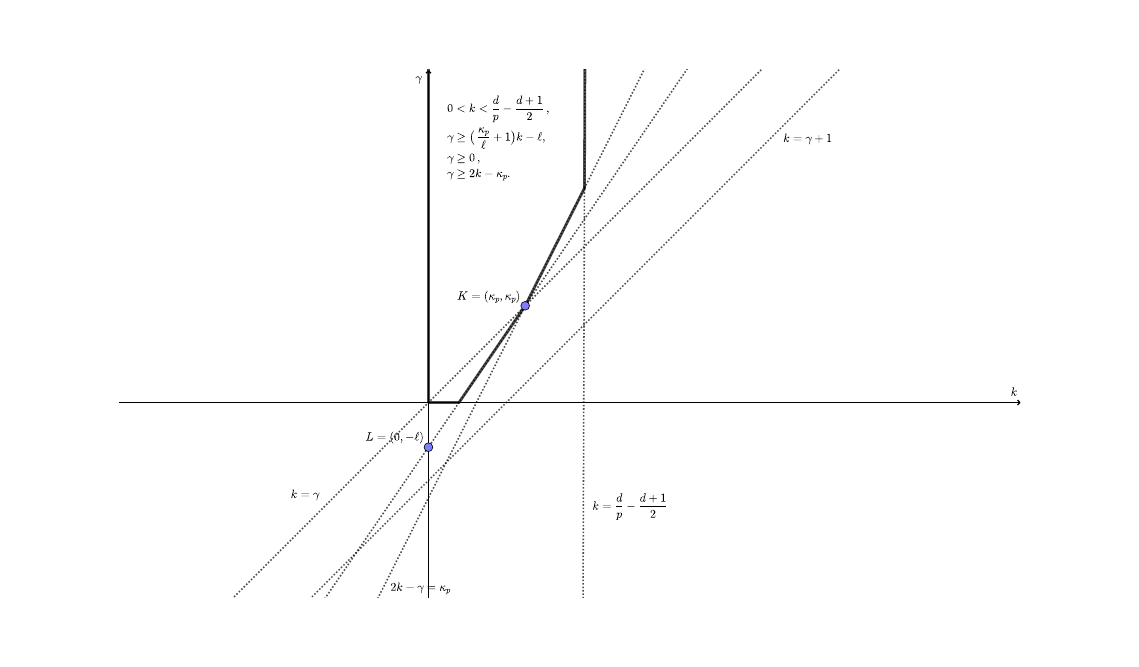}
\caption{Diagram for the case~$\ell \leq \kappa_p < 2$.}
\label{fig:Basic_KappaLessTwo_Ell_Small}
\end{figure}
\begin{figure}[h]
\includegraphics[height=9.5cm]{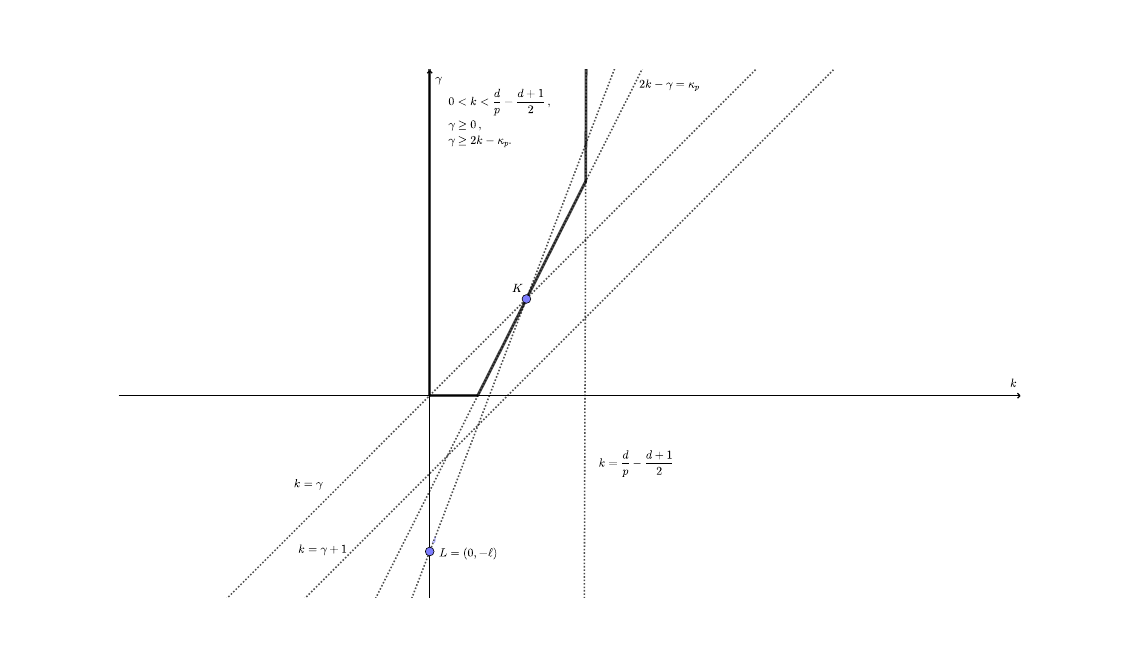}
\caption{Diagram for the case~$\kappa_p < 2$ and~$\ell > \kappa_p$.}
\label{fig:Basic_KappaLessTwo_Ell_Large}
\end{figure}

\paragraph{\bf Case~$\kappa_p \geq 2$.} In this case, there will be three subcases:~$\ell \leq \frac{\kappa_p}{\kappa_p - 1}$ (if this inequality turns into equality, then~$KL$ passes through the point~$(1,0)$),~$\frac{\kappa_p}{\kappa_p - 1} < \ell \leq \kappa_p$, and~$\ell > \kappa_p$. In the first case, the condition~\eqref{HDRSigmaLpShiftCondition} is unnecessary (see Figure~\eqref{fig:Basic_L_Small}). In the second case, all the conditions are required (see Figure~\ref{fig:Basic_L_inbetween}). In the third case, the condition~\eqref{HDRShiftedKnapp} is unnecessary (see Figure~\ref{fig:Basic_L_Large}).
\begin{figure}[h]
\includegraphics[height=9.5cm]{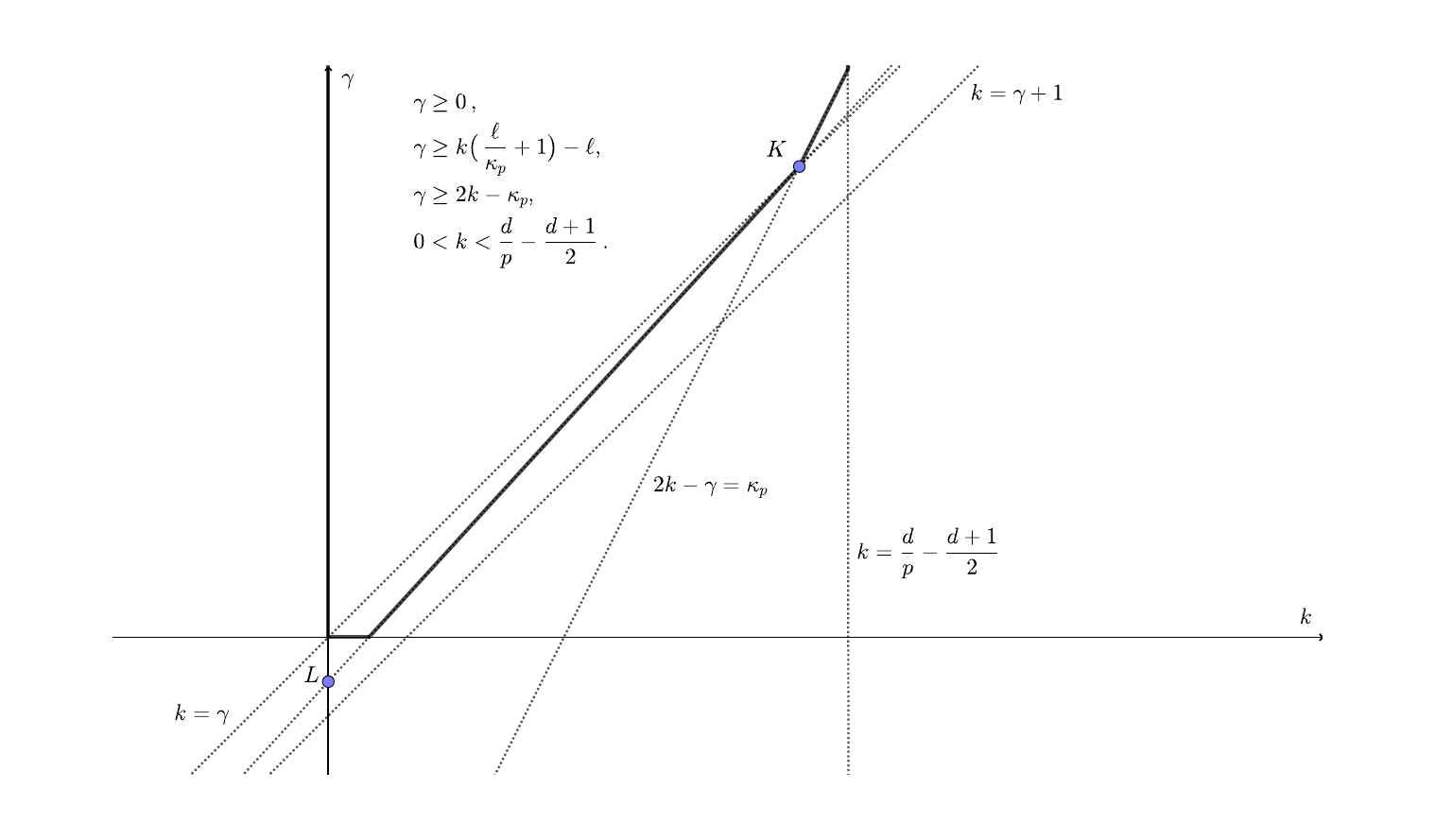}
\caption{Diagram for the case~$\kappa_p \geq 2$ and~$\ell \leq \frac{\kappa_p}{\kappa_p - 1}$.}
\label{fig:Basic_L_Small}
\end{figure}
\begin{figure}[h]
\includegraphics[height=9.5cm]{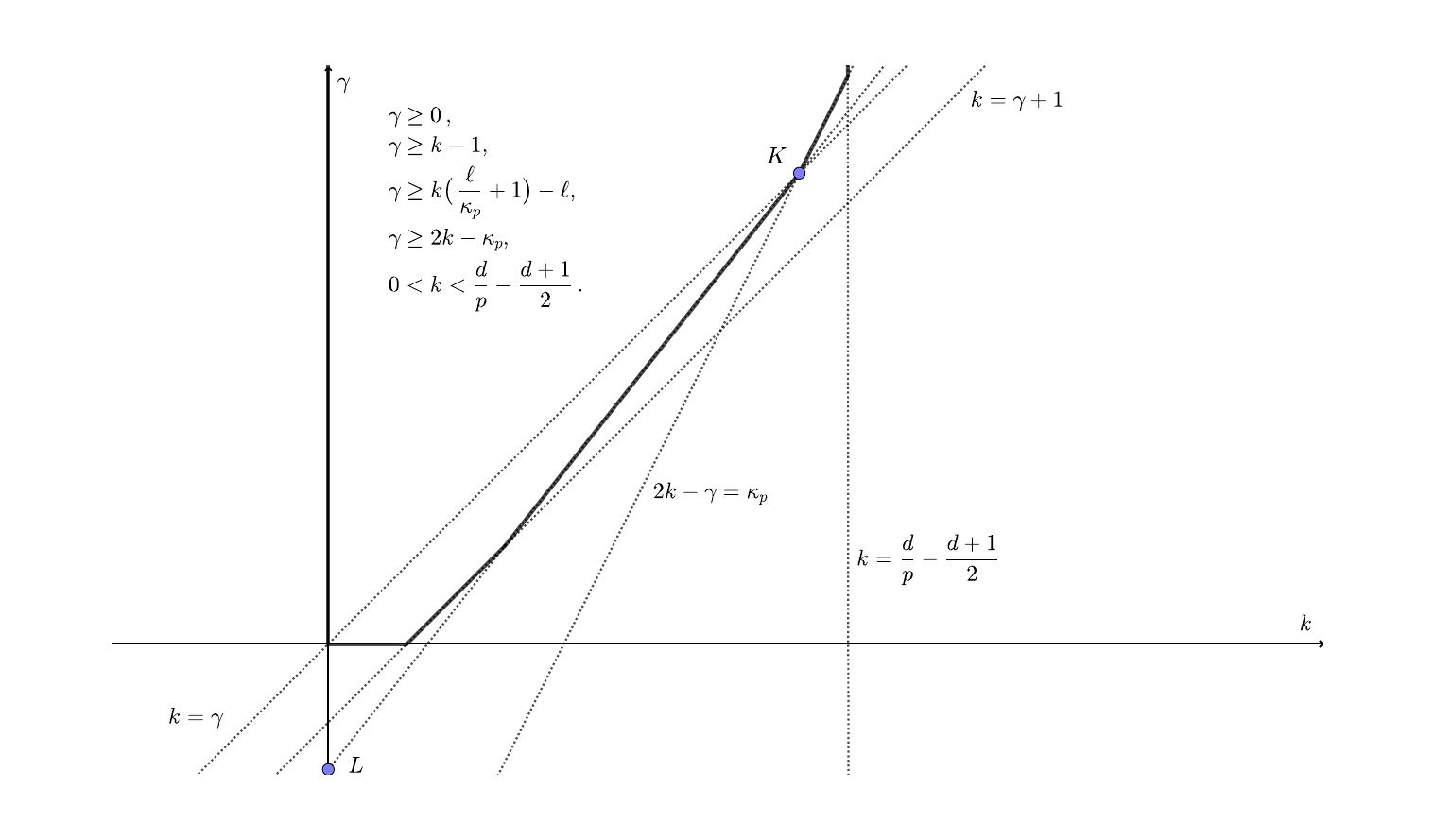}
\caption{Diagram for the case~$\kappa_p \geq 2$ and~$\frac{\kappa_p}{\kappa_p - 1} < \ell \leq \kappa_p$.}
\label{fig:Basic_L_inbetween}
\end{figure}
\begin{figure}[h!]
\includegraphics[height=9.5cm]{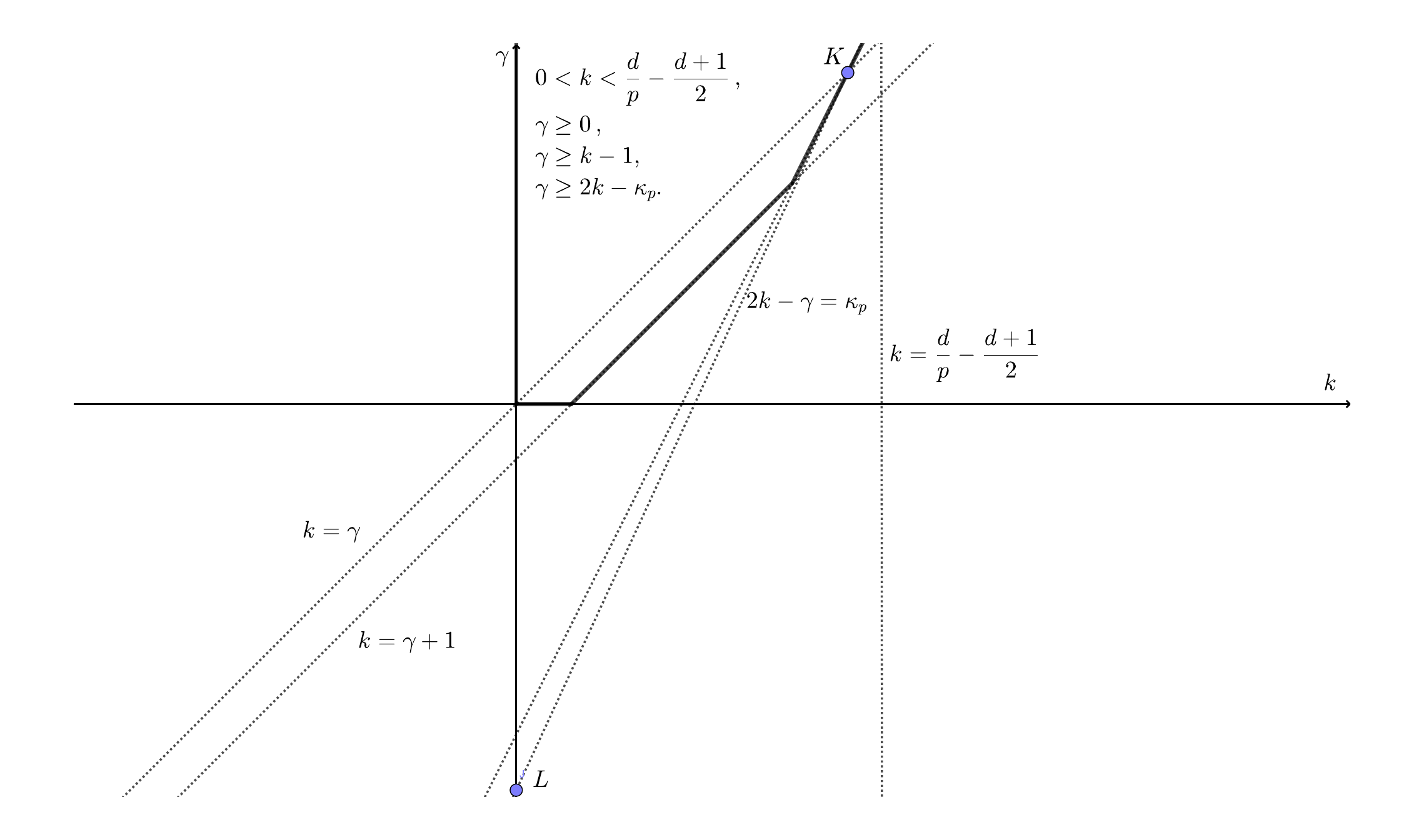}
\caption{Diagram for the case~$2 \leq \kappa_p \leq \ell$.}
\label{fig:Basic_L_Large}
\end{figure}

\subsection{Two simple examples}\label{s42}
As it was explained in Lemma~\ref{ApproximationForHDR}, we may assume that~$f$ is a Schwartz function. From now on,~$f$ is Schwartz.

\paragraph{\bf Case~$p=1$,~$k=1$,~$s = 0$, and~$\ell = 2$.}
Note that~$d \geq 5$ here and this choice of parameters corresponds to the case illustrated by Figure~\ref{fig:Basic_L_inbetween}. Let us first prove the inequality
\begin{equation*}
\Big\|\psi(\cdot)\frac{\partial\hat{f}}{\partial \xi_d}(\cdot,h(\cdot))\Big\|_{L_2} \lesssim \|f\|_{L_1} + \big\| \psi(\cdot) \hat{f}(\cdot,h(\cdot))\big\|_{H^2}
\end{equation*}
when~$d > 5$. In other words, we want to prove~$\HDRL(h,1,0,2,1)$ (see Subsection~\ref{s61} below). We start with the Newton--Leibniz formula (recall that~$\psi$ is independent of~$\xi_d$):
\begin{equation}\label{NewtonLeibniz}
\Big\|\psi(\cdot)\frac{\partial\hat{f}}{\partial\xi_d}(\cdot,h(\cdot))\Big\|_{L_2}^2  = \frac12\frac{\partial^2}{\partial r^2}\Big\|\psi(\cdot)\hat{f}(\cdot,h(\cdot)+r)\Big\|_{L_2}^2 - \Scalprod{\psi(\cdot)\frac{\partial^2\hat{f}}{\partial \xi_d^2}(\cdot,h(\cdot))}{\psi(\cdot)\hat{f}(\cdot,h(\cdot))}_{L_2}.
\end{equation}
The absolute value of the first summand on the right is bounded by~$\|f\|_{L_1}^2$ by~$\SI(h,0,2,0,1)$, which holds true provided~$d > 5$, by Theorem~\ref{StubbornTheorem}. To estimate the absolute value of the second summand, we use the Cauchy--Schwarz inequality:
\begin{equation*}
\bigg|\Scalprod{\psi(\cdot)\frac{\partial^2\hat{f}}{\partial \xi_d^2}(\cdot,h(\cdot))}{\psi(\cdot)\hat{f}(\cdot,h(\cdot))}_{L_2}\bigg| \leq \big\|\psi(\cdot)\hat{f}(\cdot,h(\cdot))\big\|_{\dot{H}^2}\Big\|\psi(\cdot)\frac{\partial^2\hat{f}}{\partial \xi_d^2}(\cdot,h(\cdot))\Big\|_{\dot{H}^{-2}}.
\end{equation*}
Note that the second multiple is bounded by~$\|f\|_{L_1}$ by~$\SI(h,2,0,2,1)$ by Theorem~\ref{StubbornTheorem} since~$d > 5$. So, we have proved
\begin{equation*}
\Big\|\psi(\cdot)\frac{\partial\hat{f}}{\partial \xi_d}(\cdot,h(\cdot))\Big\|_{L_2}^2 \lesssim \|f\|_{L_1}\Big(\|f\|_{L_1} + \big\|\psi(\cdot) \hat{f}(\cdot,h(\cdot))\big\|_{\dot{H}^2}\Big).
\end{equation*}
Clearly, one may pass to inhomogeneous Bessel space on the right.

A brief inspection of the conditions in Theorem~\ref{StubbornTheorem} shows that we have indeed proved 
\begin{equation*}
\Big\|\psi(\cdot)\frac{\partial\hat{f}}{\partial \xi_d}(\cdot,h(\cdot))\Big\|_{L_2} \lesssim \|f\|_{L_p} + \big\| \psi(\cdot)\hat{f}(\cdot,h(\cdot))\big\|_{H^2}
\end{equation*}
if~$p \leq \frac{2d+2}{d+7}$ and~$d > 5$ (the first inequality is~$\kappa_p \geq 2$, which means that~$(1,0)$ lies above the line~$KL$ on Figure~\ref{fig:Basic_L_inbetween}). In the case~$d=5$ we have the endpoint weak-type bound
\begin{equation*}
\Big\|\psi(\cdot)\frac{\partial\hat{f}}{\partial \xi_d}(\cdot,h(\cdot))\Big\|_{L_2} \lesssim \|f\|_{L_1} + \big\| \psi(\cdot) \hat{f}(\cdot,h(\cdot))\big\|_{B^{2,1}_2},
\end{equation*}
where~$B_2^{2,1}$ is the Besov space (see~\cite{BerghLofstrom} for definition and basic properties). To see this, we note that~$\SI(h,0,2,0,1)$ is true when~$d=5$, and Corollary~\ref{EndpointCaseWithoutWeight} says that
\begin{equation*}
\Big\|\psi(\cdot)\frac{\partial^2\hat{f}}{\partial \xi_d^2}(\cdot,h(\cdot))\Big\|_{B^{-2,\infty}_2} \lesssim \|f\|_{L_1}, \quad d \geq 5.
\end{equation*}
Note that the exponent~$\ell=2$ is sharp in this inequality when~$d=5$ by~\eqref{HDRShiftedKnapp}.

\paragraph{\bf Case~$p=1$,~$k=1$,~$s = 0$, and~$\ell = \frac32$.}
Note that~$d \geq 7$ here and this choice of parameters corresponds to the case illustrated by Figure~\ref{fig:Basic_L_inbetween}. We apply the Newton--Leibniz formula,~$\SI(h,0,2,0,1)$ (which is true in our case), and the Cauchy-Schwarz inequality again:
\begin{equation*}
\Big\|\psi(\cdot)\frac{\partial\hat{f}}{\partial \xi_d}(\cdot,h(\cdot))\Big\|_{L_2}^2  \lesssim  \|f\|_{L_1}^2 + \big\|\psi(\cdot)\hat{f}(\cdot,h(\cdot))\big\|_{\dot{H}^{\frac32}}\Big\|\psi(\cdot)\frac{\partial^2\hat{f}}{\partial \xi_d^2}(\cdot,h(\cdot))\Big\|_{\dot{H}^{-\frac32}}.
\end{equation*}
To estimate the second summand, we use yet another the Newton--Leibniz formula:
\begin{equation*}
\Big\|\psi(\cdot)\frac{\partial^2\hat{f}}{\partial \xi_d^2}(\cdot,h(\cdot))\Big\|^2_{\dot{H}^{-\frac32}} = \frac12\frac{\partial^2}{\partial r^2}\Big\|\psi(\cdot)\frac{\partial \hat{f}}{\partial \xi_d}(\cdot,h(\cdot))\Big\|_{\dot{H}^{-\frac32}}^2 - \Scalprod{\psi(\cdot)\frac{\partial^3\hat{f}}{\partial \xi_d^3}(\cdot,h(\cdot))}{\psi(\cdot)\frac{\partial\hat{f}}{\partial \xi_d}(\cdot,h(\cdot))}_{\dot{H}^{-\frac32}}.
\end{equation*}
The first summand may be estimated by~$\|f\|_{L_1}^2$ if~$\SI(h,1,2,\frac32,1)$ holds true. We estimate the second summand with the help of the Cauchy--Schwarz inequality:
\begin{equation*}
\bigg|\Scalprod{\psi(\cdot)\frac{\partial^3\hat{f}}{\partial \xi_d^3}(\cdot,h(\cdot))}{\psi(\cdot)\frac{\partial\hat{f}}{\partial \xi_d}(\cdot,h(\cdot))}_{\dot{H}^{-\frac32}}\bigg| \leq \Big\|\psi(\cdot)\frac{\partial^3\hat{f}}{\partial \xi_d}(\cdot,h(\cdot))\Big\|_{\dot{H}^{-3}}\Big\|\psi(\cdot)\frac{\partial\hat{f}}{\partial \xi_d}(\cdot,h(\cdot))\Big\|_{L_2}.
\end{equation*}
This is bounded by~$\|f\|_{L_1}\Big\|\psi\frac{\partial\hat{f}}{\partial \xi_d}\Big\|_{L_2}$ provided we have~$\SI(h,3,0,3,1)$ (which we have if~$d > 7$ due to Theorem~\ref{StubbornTheorem}). So, if we have~$\SI(h,1,2,\frac32,1)$, then
\begin{equation*}
\Big\|\psi(\cdot)\frac{\partial\hat{f}}{\partial \xi_d}(\cdot,h(\cdot))\Big\|_{L_2}^2 \lesssim \|f\|_{L_1}^2 + \big\|\psi(\cdot)\hat{f}(\cdot,h(\cdot))\big\|_{\dot{H}^{\frac32}}\Big(\|f\|_{L_p} + \big\|\psi(\cdot)\hat{f}(\cdot,h(\cdot))\big\|_{\dot{H}^{\frac32}}\Big\|\psi(\cdot)\frac{\partial\hat{f}}{\partial \xi_d}(\cdot,h(\cdot))\Big\|_{L_2}\Big).
\end{equation*}
Since~$\SI(h,1,2,\frac32,1)$ holds true when~$d > 7$, this leads to the inequality
\begin{equation*}
\Big\|\psi(\cdot)\frac{\partial\hat{f}}{\partial \xi_d}(\cdot,h(\cdot))\Big\|_{L_2} \lesssim \|f\|_{L_1} + \big\|\psi(\cdot)\hat{f}(\cdot,h(\cdot))\big\|_{H^{\frac32}}.
\end{equation*}
Similar to the previous case, there is a result
\begin{equation*}
\Big\|\psi(\cdot)\frac{\partial\hat{f}}{\partial \xi_d}(\cdot,h(\cdot))\Big\|_{L_2} \lesssim \|f\|_{L_1} + \big\|\psi(\cdot)\hat{f}(\cdot,h(\cdot))\big\|_{H^{\ell}}
\end{equation*}
for~$d=7$ and~$\ell > \frac32$, which is sharp up to the endpoint with respect to~\eqref{HDRShiftedKnapp}.

\subsection{Convexity properties of the function $N$}\label{s43}
As we have seen in the examples, it is useful to consider the expression
\begin{equation*}
N(k,s) = \Big\|\psi(\cdot)\frac{\partial^k\hat{f}}{\partial \xi_d^k}(\cdot,h(\cdot))\Big\|_{\dot{H}^{-s}}^2
\end{equation*}
as a function of the parameters~$k$ and~$s$. We always assume~$k$ is a non-negative integer and~$s$ is a non-negative real. Since we will be working with points in the~$(k,s)$-plane, we will give names to some regions there.
\begin{Def}
Let~$d,p,\ell$, and~$h$ be fixed. The domain
\begin{equation*}
\Big\{(k,s) \in \mathbb{R}^2\;\Big|\, k\in (0,\sigma_p), s \geq 0, s \geq k-1, 2k-s \leq \kappa_p\Big\}
\end{equation*}
is called the friendly region. The domain where~$s \geq k$ is called the subcritical region. The set of all points~$(k,s)$ such that~$\HDRL(h,k,s,\ell,p)$ holds true is called the~$\HDR$-domain.
\end{Def}

If~$X$ is an arbitrary point in the~$(k,s)$ plane,~$k_X$ will usually denote its~$k$-coordinate, and~$s_X$ will denote its~$s$-coordinate.
\begin{Rem}
The~$\HDR$-domain lies inside the friendly domain \textup(by Proposition~\ref{SufficientConditionsHDR}\textup).
\end{Rem}
\begin{Le}\label{Convexity}
For any~$k$ and~$s$, there exists a constant~$C$ such that the inequality
\begin{equation*}
N(k,s) \leq C\|f\|_{L_p}^2 + \sqrt{N(k_1,s_1)N(k_2,s_2)},\quad 2k = k_1+k_2, 2s = s_1 + s_2, 
\end{equation*}
is true provided~$(k,s)$ lies in the friendly region and~$0 \leq k_j<\sigma_p$ for~$j=1,2$.
\end{Le}
We will need an ``algebraic'' lemma that will link the quantities~$N(k,s)$,~$N(k_1,s_1)$, and~$N(k_2,s_2)$ together. 
\begin{Le}\label{AlgebraicLemma}
For any~$k, k_1, k_2\in\mathbb{Z}_+$ such that~$2k = k_1+k_2$, there exist coefficients~$c_1,c_2,\ldots,c_{|k_1-k|}$ such that
\begin{multline}\label{AlgebraOne}
(-1)^{|k-k_1|}\Big\|\psi(\cdot)\frac{\partial^k\hat{f}}{\partial \xi_d^k}(\cdot,h(\cdot))\Big\|_{\dot{H}^{-s}}^2 - \Re\Scalprod{\psi(\cdot)\frac{\partial^{k_1}\hat{f}}{\partial \xi_d^{k_1}}(\cdot,h(\cdot))}{\psi(\cdot)\frac{\partial^{k_2}\hat{f}}{\partial\xi_d^{k_2}}(\cdot,h(\cdot))}_{\dot{H}^{-s}} =\\ \sum\limits_{j=1}^{|k - k_1|}c_j\frac{\partial^{2j}}{\partial r^{2j}}\Big\|\psi(\cdot)\frac{\partial^{k-j}\hat{f}}{\partial\xi_d^{k-j}}(\cdot,h(\cdot)+r)\Big\|_{\dot{H}^{-s}}^2
\end{multline}
for any function~$f$ and any~$s$.
\end{Le}
\begin{proof}
First, by the Newton--Leibniz formula,
\begin{multline*}
\frac{\partial^{2j}}{\partial r^{2j}}\Big\|\psi(\cdot)\frac{\partial^{k-j}\hat{f}}{\partial \xi_d^{k-j}}(\cdot,h(\cdot)+r)\Big\|_{\dot{H}^{-s}}^2 = C_{2j}^j\Scalprod{\psi(\cdot)\frac{\partial^{k}\hat{f}}{\partial \xi_d^k}(\cdot,h(\cdot))}{\psi(\cdot)\frac{\partial^{k}\hat{f}}{\partial\xi_d^k}(\cdot,h(\cdot))}_{\dot{H}^{-s}} +\\ 2\sum\limits_{i=0}^{j-1} C_{2j}^i \Scalprod{\psi(\cdot)\frac{\partial^{k+j-i}\hat{f}}{\partial\xi_d^{k+j-i}}(\cdot,h(\cdot))}{\psi(\cdot)\frac{\partial^{k+i-j}\hat{f}}{\partial\xi_d^{k+i-j}}(\cdot,h(\cdot))}_{\dot{H}^{-s}}.
\end{multline*}
Thus, it is clear that
\begin{equation*}
\Re\Scalprod{\psi(\cdot)\frac{\partial^{k_1}\hat{f}}{\partial \xi_d^{k_1}}(\cdot,h(\cdot))}{\psi(\cdot)\frac{\partial^{k_2}\hat{f}}{\partial\xi_d^{k_2}}(\cdot,h(\cdot))}_{\dot{H}^{-s}}
\end{equation*}
is a linear combination of all the other terms in the identity~\eqref{AlgebraOne}. The only non-trivial question is why does the term
\begin{equation*}
\Big\|\psi(\cdot)\frac{\partial^k\hat{f}}{\partial\xi_d}(\cdot,h(\cdot))\Big\|_{\dot{H}^{-s}}^2
\end{equation*}
have coefficient~$(-1)^{|k-k_1|}$. 
For this we observe that the binomial coefficients that appear in~\eqref{AlgebraOne} once the Newton--Leibniz formula is applied are the same ones that arise in the trigonometric identity
\begin{equation*}
(-1)^{|k-k_1|} - \cos(2|k-k_1|\theta) = \sum_{j=1}^{|k-k_1|} c_j [2\cos \theta]^{2j}.
\end{equation*}
The result follows by evaluating the trigonometric sum at $\theta = \frac{\pi}{2}$.
\end{proof}
\begin{proof}[Proof of Lemma~\textup{\ref{Convexity}}.]
Lemma~\ref{AlgebraicLemma} says that
\begin{multline*}
\Big\|\psi(\cdot)\frac{\partial^k\hat{f}}{\partial\xi_d^k}(\cdot,h(\cdot))\Big\|_{\dot{H}^{-s}}^2 \leq\\  \bigg|\Scalprod{\psi(\cdot)\frac{\partial^{k_1}\hat{f}}{\partial\xi_d^{k_1}}(\cdot,h(\cdot))}{\psi(\cdot)\frac{\partial^{k_2}\hat{f}}{\partial\xi_d^{k_2}}(\cdot,h(\cdot))}_{\dot{H}^{-s}}\bigg| + C\sum\limits_{j=1}^{|k - k_1|}\bigg|\frac{\partial^{2j}}{\partial r^{2j}}\Big\|\psi(\cdot)\frac{\partial^{k-j}\hat{f}}{\partial\xi_d^{k-j}}(\cdot,h(\cdot)+r)\Big\|_{\dot{H}^{-s}}^2\bigg|.
\end{multline*}
By the Cauchy--Schwarz inequality, the first summand on the right can be estimated by~$\sqrt{N(k_1,s_1)N(k_2,s_2)}$. All the remaining terms are bounded by~$C\|f\|_{L_p}^2$, provided~$\SI(h,k-j,2j,s,p)$ holds true for any~$j=1,2,\ldots, |k-k_1|$. By Theorem~\ref{StubbornTheorem}, this holds exactly when
\begin{equation*}
\begin{aligned}
k-j \leq s;\\
k+j < \sigma_p;\\
2k - s\leq \kappa_p,
\end{aligned}
\qquad\qquad j = 1,2,\ldots, |k-k_1|.
\end{equation*}
The first list of conditions turns into~$k-1\leq s$. So, the first and the third conditions are fulfilled inside the friendly region. The second list is reduced to~$k_j < \sigma_p$,~$j=1,2$. 
\end{proof}
\begin{Cor}\label{UsualConvexity}
For any~$k$ and~$s$, there exists a constant~$C$ such that the inequality
\begin{equation*}
N(k,s) \leq C\|f\|_{L_p}^2 + \frac{N(k_1,s_1)+N(k_2,s_2)}{2},\quad 2k = k_1+k_2,\ 2s = s_1 + s_2, 
\end{equation*}
is true provided~$(k,s)$ lies in the friendly region and~$0 \leq k_j<\sigma_p$ for~$j=1,2$.
\end{Cor}
\begin{Le}\label{ConvexSequence}
Let~$a\colon \{0\}\cup [M..N]\to \mathbb{R}$ be a finite sequence and let~$2M \leq N$. Assume that
\begin{equation*}
\forall k \in [M+1..N-1]\qquad a_k\leq 1 + \frac{a_{k+1} + a_{k-1}}{2},
\end{equation*}
$a_M \leq 1 + \frac{a_{0} + a_{2M}}{2}$, and~$a_0 \leq 1$,~$a_N \leq 1$. Then,~$a_M \lesssim 1$.
\end{Le}
\begin{proof}
Consider the sequence~$\{b_k\}_k$,~$b_k = a_k + k^2$. Its terms satisfy the inequalities
\begin{equation*}
\forall k \in [M+1..N-1]\qquad b_k\leq \frac{b_{k+1} + b_{k-1}}{2}
\end{equation*}
and~$b_M \leq \frac{b_0+b_{2M}}{2}$. In particular,~$\{b_k\}_k$ is convex on~$[M..N]$. We also subtract the linear function~$k\frac{b_N-b_0}{N}+b_0$ from it:
\begin{equation*}
c_k = b_k - \Big(k\frac{b_N-b_0}{N}+b_0\Big), \quad k \in \{0\}\cup [M..N].
\end{equation*}
The sequence~$\{c_k\}_k$ is convex on~$[M..N]$, equals zero at the endpoints~$0$ and~$N$, and also satisfies the inequality~$2c_{M} \leq c_{2M}$. Thus,~$c_{M} \leq 0$ (otherwise,~$c_{2M} \geq 2c_{M}\geq c_{M}$, which contradicts the convexity of~$c$ on the interval~$[M..N]$). Therefore,~$b_{M} \leq \frac{k}{N}(b_N - b_0)$, and finally,~$a_M \leq \frac{M(N^2+1)}{N}$.
\end{proof}
\begin{Rem}\label{OtherIndices}
In fact, we have proved that~$a_k \lesssim 1$ for any~$k \in [M..N]$. 
\end{Rem}
\begin{Rem}\label{HomogeneousConvexSequence}
Using the homogeneity, one can replace the assumptions of Lemma~\textup{\ref{ConvexSequence}} by 
\begin{equation*}
\forall k \in [M+1..N-1]\qquad a_k\leq C + \frac{a_{k+1} + a_{k-1}}{2},
\end{equation*}
$a_M \leq C + \frac{a_{0} + a_{2M}}{2}$, and~$a_0 \leq C$,~$a_N \leq C$ for some positive constant~$C$. Then,~$a_M \lesssim C$.
\end{Rem}
\begin{Cor}\label{ConvexHull}
The~$\HDR$-domain is convex in the sense that if~$(k,s)$ is a convex combination of~$(k_1,s_1)$ and~$(k_2,s_2)$ \textup(we assume~$k,k_1,k_2 \in \mathbb{Z}_+$\textup), and the latter two points belong to the~$\HDR$-domain, then the former point lies in it as well.
\end{Cor}
\begin{proof}
Consider the line passing through our three points. Let~$X_0,X_2,\ldots, X_N$ be all the points with integer first coordinates lying on the segment connecting~$(k_1,s_1)$ and~$(k_2,s_2)$ (we enumerate the points in such a way that the~$k$-coordinate increases with the index). Consider also the sequence
\begin{equation*}
a_j = N(k_{X_j},s_{X_j}),\quad j = 0,1,\ldots, N.
\end{equation*}
By Corollary~\ref{UsualConvexity}, this sequence satisfies the inequality
\begin{equation}\label{ConvexityRelation}
a_{j} \leq C\|f\|_{L_p}^2 + \frac{a_{j-1} + a_{j+1}}{2},\quad j = 1,2,\ldots, N-1.
\end{equation}
By the assumption,~$a_0, a_N \lesssim \|f\|_{L_p}^2 + \|\psi(\cdot)\hat{f}(\cdot,h(\cdot))\|^2_{H^{\ell}}$. Thus, by Lemma~\ref{ConvexSequence} with~$M=0$ (in the light of Remark~\ref{HomogeneousConvexSequence}),~$a_j$ is bounded by~$C(\|f\|_{L_p}^2 + \|\psi\hat{f}\|_{H^{\ell}(\Sigma)}^2)$. In particular,~$(k,s)$ belongs to the~$\HDR$-domain.
\end{proof}
\begin{Cor}\label{ConnectingLWithSubcriticalDomain}
Let~$X$ be a point with natural~$k$-coordinate lying in the intersection of friendly and subcritical domains. Suppose that the point~$Y$ lies on the segment~$LX$, has natural first coordinate~$k_Y$, and lies in the friendly domain. If~$2k_Y \leq k_X$, then~$Y$ lies belongs to the~$\HDR$-domain.
\end{Cor}
\begin{proof}
The proof of this corollary is very much similar to the proof of the previous one. We consider all the points on the segment~$LX$ that have integer first coordinates and lie inside the friendly domain. Suppose the leftmost of them has first coordinate~$M$, let us call our points~$Y_M, Y_{M+1},\ldots,Y_N$ (so,~$Y_N = X$). We also add the point~$Y_0 = L$ to our sequence and consider the numbers
\begin{equation*}
a_j = N(k_{Y_j},s_{Y_j}), \quad j = 0, M, M+1, M+2,\ldots, N.
\end{equation*}
These numbers satisfy the inequality~\eqref{ConvexityRelation} for~$j \in [M+1..N-1]$. Moreover, Corollary~\ref{UsualConvexity} provides the inequality
\begin{equation*}
a_{M} \leq C\|f\|_{L_p}^2 + \frac{a_{0} + a_{2M}}{2}.
\end{equation*}
Note that~$2M \leq N$ since~$N = k_{X}$ and~$k_Y \in [M..N]$ since~$Y$ lies in the friendly domain and~$k_X \geq 2k_Y$. At the endpoint~$M$, we have the inequality
\begin{equation*}
N(k_X,s_X) \lesssim \|f\|_{L_p}^2
\end{equation*}
since~$X$ lies in the subcritical part of the friendly domain (this inequality is the case~$\beta = 0$ in Theorem~\ref{StubbornTheorem}). Thus,~$a_N \lesssim \|f\|_{L_p}^2$. Clearly,~$a_0 \leq \|\psi(\cdot)\hat{f}(\cdot,h(\cdot))\|^2_{H^{\ell}}$. So, Lemma~\ref{ConvexSequence} says all the points~$Y_M,Y_{M+1},\ldots, Y_N$ belong to the~$\HDR$-domain. In particular,~$Y$ does.
\end{proof}
To describe a convex set, it suffices to describe its extremal points. This will be the way to describe our results: we prove~$\HDRL$ for some points in the~$(k,s)$-plane, and then invoke Corollary~\ref{ConvexHull}. We return for a while to the two examples already considered and see what part of the~$\HDR$-domains we are able to access in these cases.

In the case~$p=1$,~$k=1$,~$s = 0$, and~$\ell = 2$, we applied Corollary~\ref{UsualConvexity} to estimate~$N(1,0)$ in terms of~$N(0,-2)$ and~$N(2,2)$. The first point coincides with~$L$ in this case, and the second lies in the subcritical region. If~$d > 5$, it also lies in the friendly region, so, both~$N(0,-2)$ and~$N(2,2)$ are bounded by the right hand-side of~\eqref{HDRinequality} and we get~$\HDRL(h,1,0,2,p)$. In other words, we applied Corollary~\ref{ConnectingLWithSubcriticalDomain} to~$X = (2,2)$ and~$Y = (1,0)$.

In the case~$p=1$,~$k=1$,~$s = 0$, and~$\ell = \frac32$, we consider the sequence of points~$(0,-\frac32)$,~$(1,0)$,~$(2,\frac32)$, and~$(3,3)$ lying on a line. All our points (except for~$L = (0,-\frac32)$) lie in the friendly domain if~$k_j < \frac{d-1}{2}$ provided~$d> 7$. Thus, we obtain~$\HDRL(h,1,0,\frac32,p)$ by applying Corollary~\ref{ConnectingLWithSubcriticalDomain} with~$X = (3,3)$ and~$Y = (1,0)$.

So, our general strategy will be to apply Corollary~\ref{ConnectingLWithSubcriticalDomain} to the points~$X$ close to the point~$K = (\kappa_p,\kappa_p)$. This will enable us to obtain ``almost extremal points'' of the $\HDR$ region, after that, we will apply Corollary~\ref{ConvexHull} to pass to convex hulls. Before we pass to the cases, we explain the obstructions that prevent us from proving the sufficiency of the conditions in Proposition~\ref{SufficientConditionsHDR}. They are of two types. First, we are able to work with points whose first coordinates are integers only. However, in the general case, the extremal points of the domain of admissible parameters  need not necessarily have integer first coordinates. So, we cannot prove (and even formulate)~$\HDR$ for them. This makes the convex hull we obtain smaller (we are able to reach only some ``integer'' points close to the extremal points) than it should be. The second obstruction is more severe. The problem comes from the inequality~$2k_Y \leq k_X$ in Corollary~\ref{ConnectingLWithSubcriticalDomain}. That restricts our ``extremal points'' from having too large~$k$-coordinate, roughly speaking, their~$k$-coordinates should satisfy~$2k<\sigma_p$, if we want to apply Corollary~\ref{ConnectingLWithSubcriticalDomain}. This will result in a considerable gap between our results and the conditions listed in Proposition~\ref{SufficientConditionsHDR} in the case when~$\ell \geq 2$.

Now we pass to the cases.

\subsection{Statement of results by cases}\label{s44}
\begin{figure}[h]
\includegraphics[height=9.5cm]{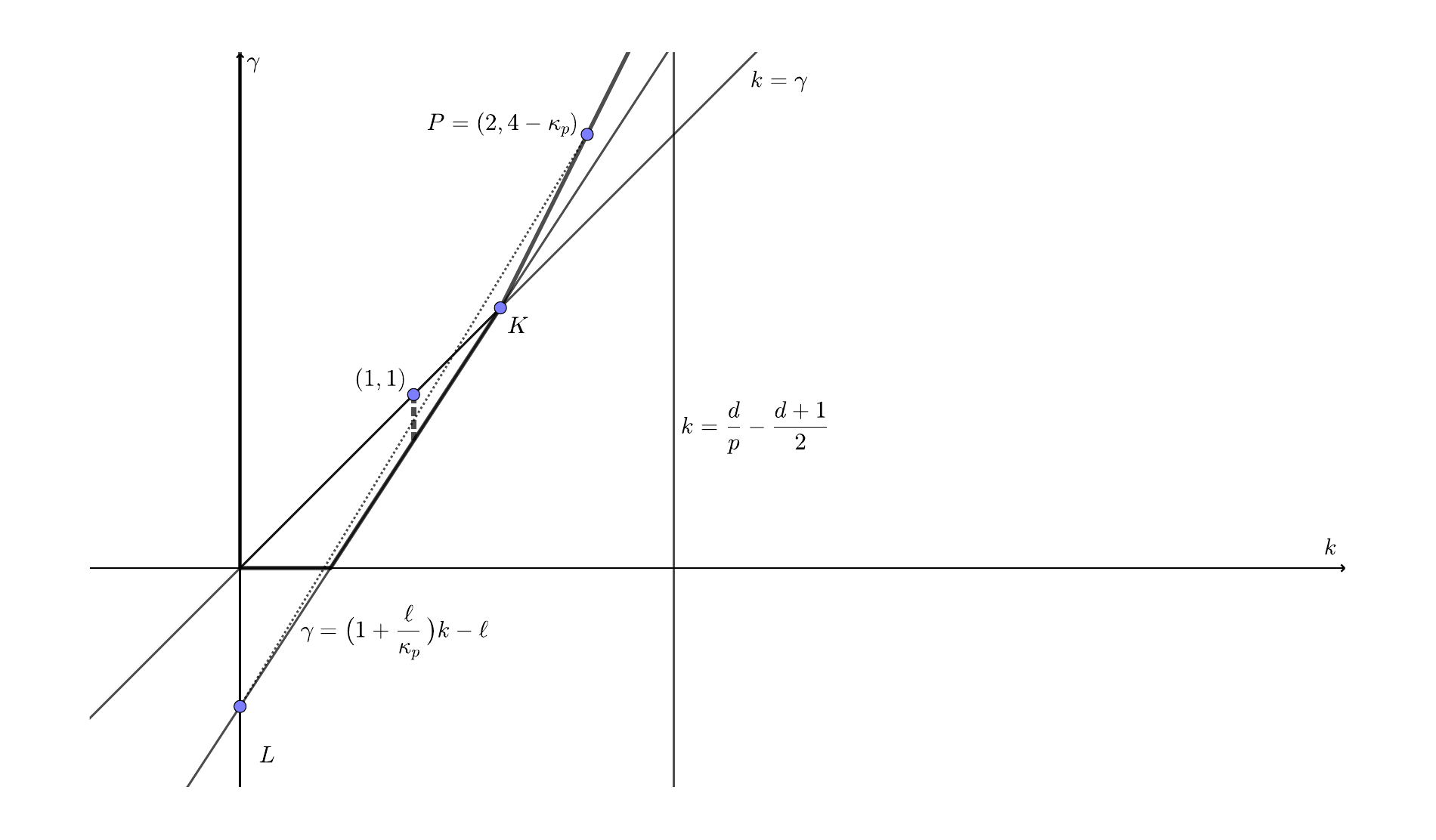}
\caption{What we can reach in the case~$\ell \leq \kappa_p < 2$.}
\label{fig:KappaLessTwoEllLessKappa}
\end{figure}
\paragraph{\bf Case~$\ell \leq \kappa_p < 2$.} Our reasonings are illustrated by Figure~\ref{fig:KappaLessTwoEllLessKappa}. Clearly, here we are interested in the case~$k=1$ only (because if~$k \geq 2$ and~$(k,s)$ lies in the~$\HDR$ domain, then~$k \leq s$ automatically). We consider the point~$P = (2,4-\kappa_p)$ and assume~$P$ lies in the friendly region, that is,~$2 < \sigma_p$. We draw a segment that connects~$P$ with~$L$ (it is the slant punctured segment on Figure~\ref{fig:KappaLessTwoEllLessKappa}).  It crosses the line~$k=1$ at the point~$(1,2-\frac{\kappa_p + \ell}{2})$. We apply Corollary~\ref{ConnectingLWithSubcriticalDomain} to the points~$P$ as~$X$ and~$(1,2-\frac{\kappa_p + \ell}{2})$ as~$Y$ and obtain the theorem below.
\begin{Th}
Let~$\ell \leq \kappa_p < 2$, let~$2 < \sigma_p$. Then,~$\HDRL(h,1,s,\ell,p)$ holds true provided
\begin{equation*}
s \geq \min (1,2 - \frac{\kappa_p + \ell}{2}).
\end{equation*}
\end{Th}

\begin{figure}[h]
\includegraphics[height=9.5cm]{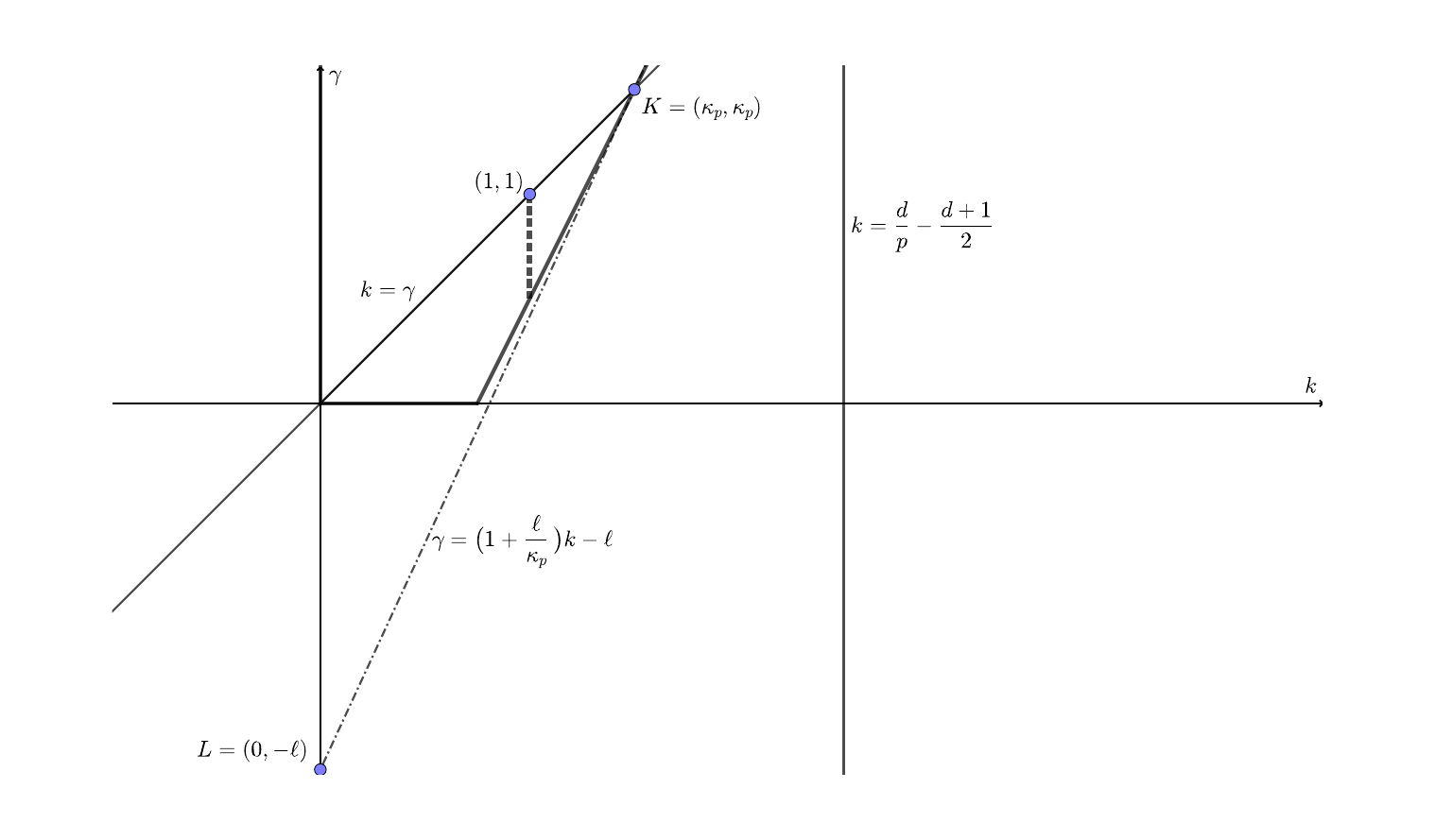}
\caption{What we can reach in the case~$\kappa_p < 2$ and~$\ell > \kappa_p$.}
\label{fig:KappaLessTwoEllGreaterKappa}
\end{figure}

\paragraph{\bf Case~$\kappa_p < 2$,~$\ell > \kappa_p$.} Our reasonings are illustrated by Figure~\ref{fig:KappaLessTwoEllGreaterKappa}. This case is simpler than the previous one. We only need~$2 < \sigma_p$ here. In this case, if~$(1,s)$ lies on the vertical punctured segment, then it is an average of~$L$ and a point inside the intersection of the friendly domain with the subcritical domain. Thus, Corollary~\ref{ConnectingLWithSubcriticalDomain} leads to the theorem below.
\begin{Th}
Let~$\kappa_p < 2$,~$\kappa_p < \ell$, let~$2 < \sigma_p$. Then,~$\HDRL(h,1,s,\ell,p)$ holds true provided
\begin{equation*}
s \geq \min (1,2 - \kappa_p).
\end{equation*}
\end{Th}

\begin{figure}[h]
\includegraphics[height=9.5cm]{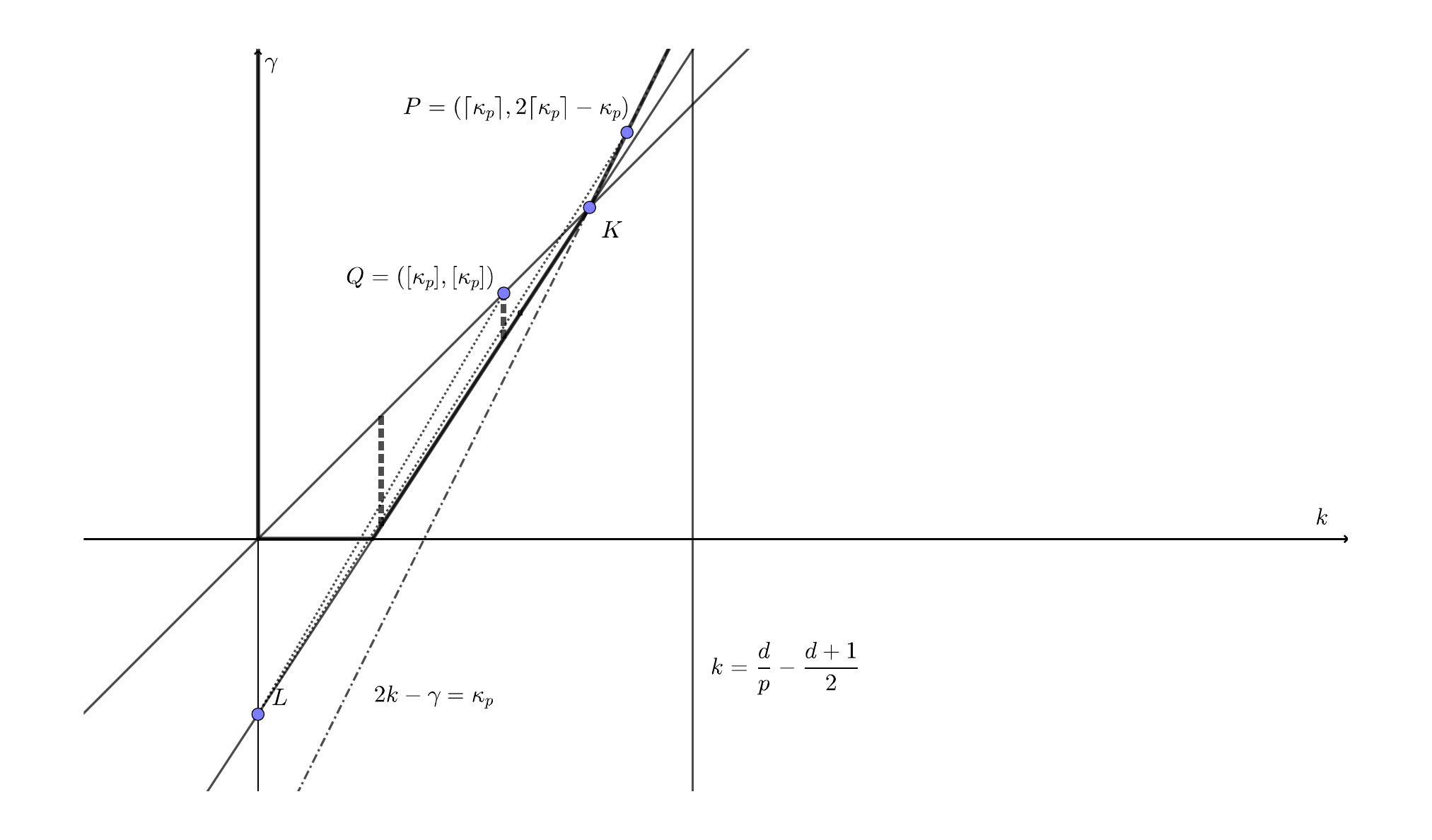}
\caption{What we can reach in the case~$2\leq \kappa_p$ and~$\ell \leq \frac{\kappa_p}{\kappa_p-1}$.}
\label{fig:Basic_L_Small_Proof}
\end{figure}
\paragraph{\bf Case~$2 \leq \kappa_p$ and~$\ell \leq \frac{\kappa_p}{\kappa_p-1}$.} Our reasonings are illustrated by Figure~\ref{fig:Basic_L_Small_Proof}. We introduce two auxiliary points~$P$ and~$Q$:
\begin{equation*}
P = (\lceil \kappa_p \rceil, 2\lceil\kappa_p\rceil - \kappa_p);\quad Q = ([\kappa_p],[\kappa_p]).
\end{equation*}
We have used two types of the notion ``integer part of a number", see formula~\eqref{IntegerPart}.

We connect the point~$L$ to~$P$ and~$Q$. Since the point~$Q$ lies in the intersection of friendly and subcritical regions, Corollary~\ref{ConnectingLWithSubcriticalDomain} applied to~$Q$ in the role of~$X$ says that~$\HDRL(h,k,s_k,\ell,p)$ is true for all pairs~$(k,s_k)$ such that~$(k,s_k) \in LQ$, in other words
\begin{equation*}
s_k = -\ell\frac{[\kappa_p] - k}{[\kappa_p]} + k.
\end{equation*}
Clearly, the same assertion is true for larger~$s$ when~$k$ is fixed. The situation with the point~$P$ is slightly more complicated: it may lie outside the friendly region if its~$k$-coordinate is too large. If it is not so (i.e. $\lceil \kappa_p \rceil < \sigma_p$), then we may apply Corollary~\ref{ConnectingLWithSubcriticalDomain} to the point~$P$ in the role of~$X$ and achieve~$\HDRL(h,k,s_k,\ell,p)$ is true for all pairs~$(k,s_k)$ such that~$(k,s_k) \in LP$, in other words
\begin{equation*}
s_k = -\ell\frac{\lceil\kappa_p\rceil - k}{\lceil\kappa_p\rceil} + 2k - \frac{\kappa_p k}{\lceil\kappa_p\rceil}.
\end{equation*}
We summarize our results. 
\begin{Th} \label{ellSmallTheorem}
Let~$2 \leq \kappa_p$ and let~$\ell \leq \frac{\kappa_p}{\kappa_p - 1}$. If~$\lceil \kappa_p\rceil \geq \sigma_p$, then~$\HDRL(h,k,s,\ell,p)$ holds true if
\begin{equation*}
s \geq -\ell\frac{[\kappa_p] - k}{[\kappa_p]} + k, \quad k \leq [\kappa_p].
\end{equation*}
If~$\lceil \kappa_p\rceil < \sigma_p$, then
~$\HDRL(h,k,s,\ell,p)$ holds true if
\begin{equation}\label{EquationForLQ}
s \geq \min\Big(-\ell\frac{[\kappa_p] - k}{[\kappa_p]} + k, -\ell\frac{\lceil\kappa_p\rceil - k}{\lceil\kappa_p\rceil} + 2k - \frac{\kappa_p k}{\lceil\kappa_p\rceil}\Big),\quad k \leq [\kappa_p].
\end{equation}
\end{Th}
\begin{Rem}
If~$k = \lceil\kappa_p\rceil$ and~$\HDRL(h,k,s,\ell,p)$ holds true, then~$s \geq k$.
\end{Rem}

\begin{figure}[h]
\includegraphics[height=9.5cm]{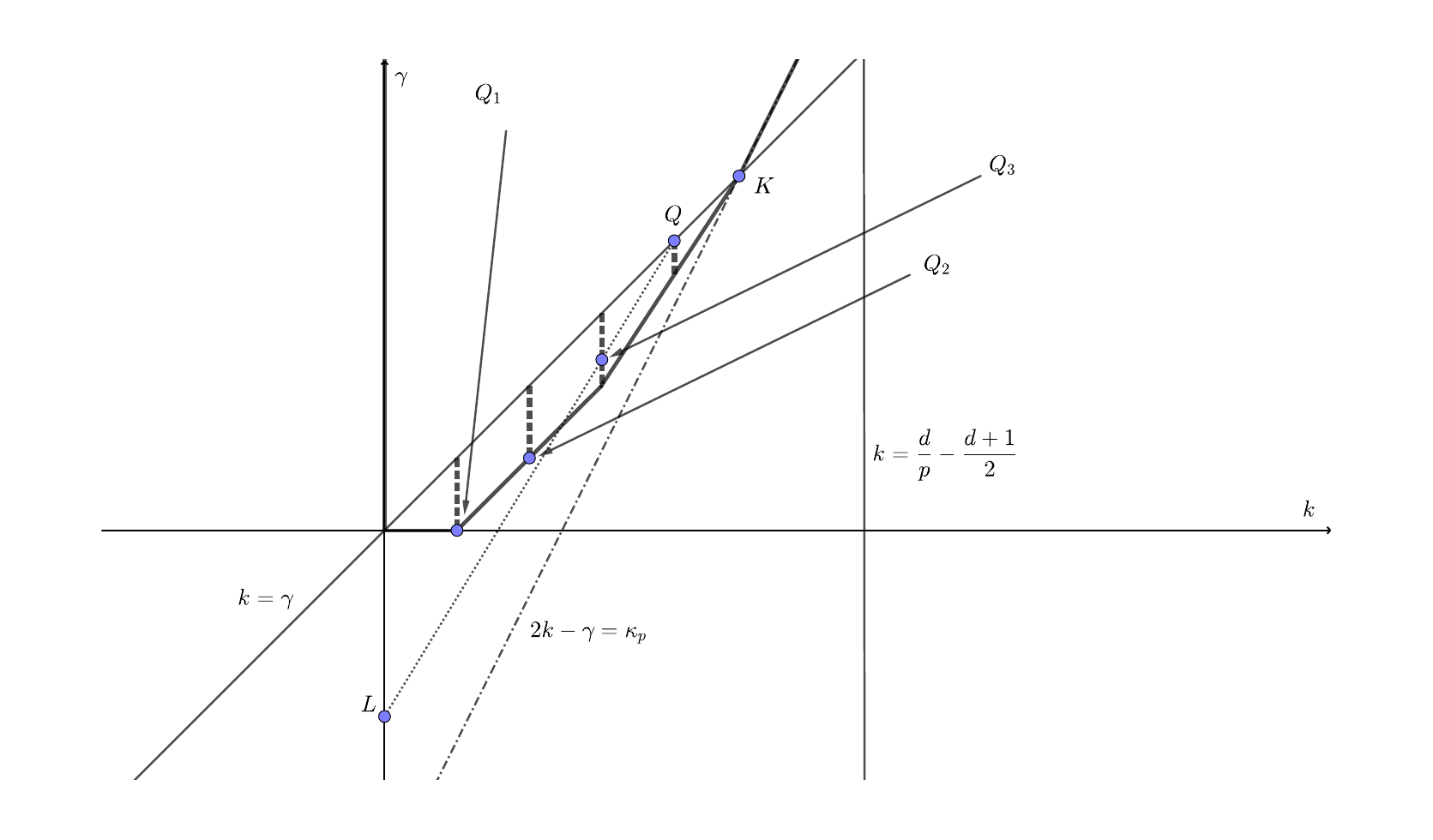}
\caption{Construction of the points~$Q_j$.}
\label{fig:Basic_L_inbetween_Proof}
\end{figure}

\paragraph{\bf Case~$2\leq \kappa_p$,~$\frac{\kappa_p}{\kappa_p - 1} \leq \ell$.}
This case will be split into many subcases. We will need to construct two sequences of points generated by~$P$ and~$Q$. 

The points~$Q_j$,~$j = 1,2,\ldots, [\kappa_p]$, are generated by~$Q$. Namely,
\begin{equation*}
Q_j = \begin{cases}
LQ \cap \{(j,s)\mid s \in \mathbb{R}\}, \quad \hbox{if this point lies above the line~$s = k-1$};\\
(j,j-1), \quad \hbox{in the other case.}
\end{cases}
\end{equation*}
The point~$Q_j$ may be described as the lowest possible point on the line~$\{(j,s)\mid s \in \mathbb{R}\}$ that lies above the segment~$LQ$ and belongs to the friendly domain. See Figure~\ref{fig:Basic_L_inbetween_Proof}.

\begin{Le}\label{StructureOfQ}
For any~$j < [\frac{\ell-1}{\ell}\kappa_p]$, we have~$Q_j = (j,j-1)$. For~$j \geq \lceil\frac{\ell-1}{\ell}\kappa_p\rceil$, all points~$Q_j$ lie on the line~$LQ$.
\end{Le}
\begin{proof}
The equation of the line~$LQ$ is
\begin{equation*}
s = -\ell\frac{[\kappa_p] - k}{[\kappa_p]} + k.
\end{equation*}
To prove the first half of the lemma, it suffices to verify the inequality
\begin{equation*}
j-1 \geq -\ell\frac{[\kappa_p] - j}{[\kappa_p]} + j
\end{equation*}
when~$j \leq [\frac{\ell-1}{\ell}\kappa_p] - 1$. This may be rewritten as
\begin{equation*}
\ell j \leq (\ell - 1)[\kappa_p].
\end{equation*}
Clearly, it suffices to prove this inequality for the largest possible~$j = [\frac{\ell-1}{\ell}\kappa_p] - 1$. In this case, we arrive at
\begin{equation*}
\ell\Big[\frac{\ell-1}{\ell}\kappa_p\Big] \leq (\ell - 1)[\kappa_p] + \ell.
\end{equation*}
We estimate the left hand-side with~$(\ell - 1)\kappa_p$, which, in its turn, does not exceed~$(\ell - 1)[\kappa_p] + (\ell - 1)$. The first assertion of the lemma is proved.

Similar to the previous reasoning, it suffices to verify the inequality
\begin{equation*}
j-1 \leq -\ell\frac{[\kappa_p] - j}{[\kappa_p]} + j
\end{equation*}
when~$j \geq \lceil\frac{\ell-1}{\ell}\kappa_p\rceil$, to prove the second assertion of the lemma. This may be rewritten as
\begin{equation*}
 j \geq \frac{\ell - 1}{\ell}[\kappa_p],
\end{equation*}
which follows from~$j \geq \frac{\ell-1}{\ell}\kappa_p$, which is weaker than our assumption~$j \geq \lceil\frac{\ell-1}{\ell}\kappa_p\rceil$. So, we have proved the second half of the lemma.
\end{proof}

The lemma says that, among all the points~$Q_j$, only those with the indices~$[\frac{\ell-1}{\ell}\kappa_p]-1$,~$[\frac{\ell-1}{\ell}\kappa_p]$, and~$\lceil\frac{\ell-1}{\ell}\kappa_p\rceil$, may be the extremal points of the accessible domain.

The points~$P_j$,~$j = 1,2,\ldots, [\kappa_p]-1$, are generated by~$P$ in a similar manner:
\begin{equation*}
P_j = \begin{cases}
PL \cap \{(j,s)\mid s \in \mathbb{R}\}, \quad \hbox{if this point lies above the line~$s = k-1$};\\
(j,j-1), \quad \hbox{in the other case.}
\end{cases}
\end{equation*}
We also consider the point~$P_{[\kappa_p]}$ separately:
\begin{equation*}
P_{[\kappa_p]} = \begin{cases}
PL \cap \{(j,s)\mid s \in \mathbb{R}\}, \quad \hbox{if this point lies above the line~$s = 2k-\kappa_p$};\\
([\kappa_p],2[\kappa_p] - \kappa_p), \quad \hbox{in the other case.}
\end{cases}
\end{equation*}
\begin{Rem}
The point~$P_{[\kappa_p]}$ lies on~$LP$ if and only if~$\kappa_p \geq \ell$.
\end{Rem}
Unfortunately, there is no analog of the first assertion of Lemma~\ref{StructureOfQ} for the points~$P_j$. Here we can only say that for small~$j$ the points~$P_j$ lie on the line~$s = k-1$ and then at some moment they jump to the line~$LP$. However, this ``moment'' can happen much earlier than~$\frac{\ell-1}{\ell}\kappa_p$. We can only bound it from above.
\begin{Le}\label{StructureOfP}
For~$j \geq \lceil\frac{\ell-1}{\ell}\kappa_p\rceil$, all points~$P_j$ lie on the line~$LP$.
\end{Le}
\begin{proof}
Consider the case~$\ell \leq \kappa_p$ first. The equation of the line~$LP$ is
\begin{equation*}
s = -\ell\frac{\lceil\kappa_p\rceil-k}{\lceil\kappa_p\rceil} + \frac{k}{\lceil\kappa_p\rceil}(2\lceil\kappa_p\rceil - \kappa_p).
\end{equation*}
So, we need to verify the inequality
\begin{equation*}
j-1 \leq -\ell\frac{\lceil\kappa_p\rceil-j}{\lceil\kappa_p\rceil} + \frac{j}{\lceil\kappa_p\rceil}(2\lceil\kappa_p\rceil - \kappa_p).
\end{equation*}
This may be rewritten as
\begin{equation*}
\frac{(\ell-1)\lceil\kappa_p\rceil}{\lceil\kappa_p\rceil-\kappa_p + \ell} \leq j.
\end{equation*}
So, it suffices to prove
\begin{equation*}
\frac{(\ell-1)\lceil\kappa_p\rceil}{\lceil\kappa_p\rceil-\kappa_p + \ell} \leq \Big\lceil\frac{\ell-1}{\ell}\kappa_p\Big\rceil.
\end{equation*}
We will prove a stronger inequality
\begin{equation*}
\frac{(\ell-1)\lceil\kappa_p\rceil}{\lceil\kappa_p\rceil-\kappa_p + \ell} \leq \frac{\ell-1}{\ell}\kappa_p,
\end{equation*}
which is equivalent to
\begin{equation*}
\ell \lceil\kappa_p\rceil \leq \kappa_p\Big(\lceil\kappa_p\rceil-\kappa_p + \ell\Big).
\end{equation*}
This may be restated as
\begin{equation*}
\ell(\lceil\kappa_p\rceil-\kappa_p) \leq \kappa_p(\lceil\kappa_p\rceil-\kappa_p),
\end{equation*}
which is true under our assumption~$\ell \leq \kappa_p$.

In the other case~$\ell > \kappa_p$, we have
\begin{equation*}
\Big\lceil\frac{\ell-1}{\ell}\kappa_p\Big\rceil \geq \lceil (\kappa_p-1)+ \rceil > [\kappa_p],
\end{equation*}
so the statement of the lemma is empty in this case (we consider the points~$P_j$ with~$j \leq [\kappa_p]$ only).
\end{proof}
\begin{Le}\label{AlmostTheorem}
If~$j$ is a number between~$1$ and~$\lceil\frac{\ell-1}{\ell}\kappa_p\rceil$ and~$2j \leq [\kappa_p]$, then~$Q_j$ belongs to the~$\HDR$-domain. If~$2j \leq \lceil \kappa_p \rceil$ and~$P$ belongs to the friendly region, then~$P_j$ belongs to the~$\HDR$ domain. 
\end{Le}
Note that~$P$ belongs to the friendly region if and only if~$\lceil \kappa_p\rceil <\sigma_p$.
\begin{proof}
We prove the second assertion, the proof of the first one is completely similar. We consider two cases:~$P_j$ lies on~$LP$ and above~$LP$. In the first case, we may apply Corollary~\ref{ConnectingLWithSubcriticalDomain} with~$P$ in the role of~$X$ and~$P_j$ in the role of~$Y$. In the second case, we may apply the same corollary with~$P_j$ in the role of~$Y$ and the point of intersection of the lines~$LP_j$ and~$\{(\lceil \kappa_p\rceil,s)\mid s \in \mathbb{R}\}$ in the role of~$X$ (the latter point lies above~$P$ since~$P_j$ lies above the segment~$LP$, and thus belongs to the friendly domain).
\end{proof}
We finally summarize our results.
\begin{Th}\label{BigTableTheorem}
Assume~$2\leq \kappa_p$,~$\frac{\kappa_p}{\kappa_p - 1} \leq \ell$. The~$\HDR$ domain contains the convex hull of points specified below. We always include the points~$(0,0)$,~$(1,0)$, and~$Q$ in our list. The other points are specified in the following table\textup{:}

\medskip
\centerline{
\begin{tabular}{|l||c|c|}
\hline
&$\lceil \kappa_p \rceil < \sigma_p$& $\lceil \kappa_p \rceil  \geq \sigma_p$\\
\hline
\hline
$2\lceil\frac{\ell-1}{\ell}\kappa_p\rceil \leq [\kappa_p]$&$Q_{\lceil \frac{\ell - 1}{\ell}\kappa_p\rceil},Q_{[\frac{\ell - 1}{\ell}\kappa_p]}, Q_{[\frac{\ell - 1}{\ell}\kappa_p]-1}, P_{\lceil \frac{\ell - 1}{\ell}\kappa_p\rceil}, P_{[\frac{\ell - 1}{\ell}\kappa_p]}$&$Q_{\lceil \frac{\ell - 1}{\ell}\kappa_p\rceil},Q_{[\frac{\ell - 1}{\ell}\kappa_p]}, Q_{[\frac{\ell - 1}{\ell}\kappa_p]-1}$\\
\hline
$2[\frac{\ell-1}{\ell}\kappa_p]\leq [\kappa_p] < 2\lceil\frac{\ell-1}{\ell}\kappa_p\rceil \leq \lceil \kappa_p\rceil$ & $Q_{[\frac{\ell - 1}{\ell}\kappa_p]}, Q_{[\frac{\ell - 1}{\ell}\kappa_p]-1}, P_{\lceil \frac{\ell - 1}{\ell}\kappa_p\rceil}, P_{[\frac{\ell - 1}{\ell}\kappa_p]}$&$Q_{[\frac{\ell - 1}{\ell}\kappa_p]}, Q_{[\frac{\ell - 1}{\ell}\kappa_p]-1}$\\
\hline
$2[\frac{\ell-1}{\ell}\kappa_p]\leq [\kappa_p] < \lceil \kappa_p\rceil < 2\lceil\frac{\ell-1}{\ell}\kappa_p\rceil$&$Q_{[\frac{\ell - 1}{\ell}\kappa_p]}, Q_{[\frac{\ell - 1}{\ell}\kappa_p]-1}, P_{[\frac{\ell - 1}{\ell}\kappa_p]}$&$Q_{[\frac{\ell - 1}{\ell}\kappa_p]}, Q_{[\frac{\ell - 1}{\ell}\kappa_p]-1}$\\
\hline
$[\kappa_p] < 2[\frac{\ell-1}{\ell}\kappa_p] \leq \lceil \kappa_p\rceil < 2\lceil\frac{\ell-1}{\ell}\kappa_p\rceil$&$Q_{[\frac{\ell - 1}{\ell}\kappa_p]-1}, P_{[\frac{\ell - 1}{\ell}\kappa_p]}$&$Q_{[\frac{\ell - 1}{\ell}\kappa_p]-1}$\\
\hline
$\lceil \kappa_p\rceil < 2[\frac{\ell-1}{\ell}\kappa_p]$&$Q_{[\frac{[\kappa_p]}{2}]},P_{[\frac{\lceil \kappa_p\rceil}{2}]}$&$Q_{[\frac{[\kappa_p]}{2}]}$\\
\hline
\end{tabular}\ .
}
\end{Th}

This theorem is a straightforward consequence of Lemma~\ref{AlmostTheorem}. What might be surprising is that only a small number of the points~$P_j$ and~$Q_j$ are extremal for our convex hull. However, this phenomenon has very simple explanation. The points~$P_j$ and~$Q_j$ lie in the union of three lines:~$LP$,~$LQ$, and~$\{s = k-1\}$ (except, possibly, for~$P_{[\kappa_p]}$, which does not lie on these three lines only in the case described by the last row in the table; in this case,~$P_{[\kappa_p]}$ is not reached by our techniques). Moreover, Lemmas~\ref{StructureOfQ} and~\ref{StructureOfP} allow us to choose specific points from each of these lines.  We have not managed to choose the minimal possible list in each of the cases (and sometimes we even choose the same point twice with different names), however, our lists are bounded by at most five points. 

We note that the cases~$\ell < \kappa_p$ and~$\ell \geq \kappa_p$ are the same for our result (our answer in these cases are given by the last row in the table above, at least when~$\kappa_p \geq 3$). However, the forms of the~$\HDR$-domain suggested by Proposition~\ref{SufficientConditionsHDR} differ in these cases (see Figures~\ref{fig:Basic_L_inbetween} and~\ref{fig:Basic_L_Large}).

\begin{proof}[Proof of Theorem~\ref{HDRExample}]
Since~$\kappa_p$ is assumed to be a nonnegative integer, we have~$K=P=Q$. Therefore, all the points~$Q_j$ and~$P_j$ lie on the lines~$LK$ and~$s = k-1$. Since $p > 1$, we have~$\kappa_p = [\sigma_p] < \sigma_p$ as well.  

When $0 \leq \ell \leq \frac {\kappa_p}{\kappa_p - 1}$, Theorem~\ref{ellSmallTheorem} immediately implies that for points $(k,s)$ whose first coordinate is a positive integer,  $\HDR(\Sigma, k, s, \ell, p)$ holds  whenever conditions~\eqref{HDRSurfaceCondition} and~\eqref{HDRShiftedKnapp} are satisfied; \eqref{EquationForLQ} turns into~\eqref{HDRShiftedKnapp}.

When $\ell > \frac{\kappa_p}{\kappa_p - 1}$ it remains to apply Theorem~\ref{BigTableTheorem} and decode its results.

For the case~$2\lceil\frac{\ell-1}{\ell}\kappa_p\rceil \leq \kappa_p$, the points $Q_{[\frac{\ell - 1}{\ell}\kappa_p]-1}$, $Q_{[\frac{\ell - 1}{\ell}\kappa_p]}$, and $Q_{\lceil \frac{\ell - 1}{\ell}\kappa_p\rceil}$ straddle the intersection of the lines $s=k-1$ and $LK$.  The convex hull of these three $Q_j$ together with points $K$ and $(1,0)$ contains every point along $s = \max(k-1, k - \ell \frac{(\kappa_p - k)}{\kappa_p})$ with $k = 1, 2, \ldots , \kappa_p$.  Thus $\HDR(\Sigma, k, s, \ell, p)$ holds provided~\eqref{HDRSigmaLpShiftCondition},~\eqref{HDRSurfaceCondition} and~\eqref{HDRShiftedKnapp} are all satisfied.

In the case~$2\lceil\frac{\ell-1}{\ell}\kappa_p\rceil > \kappa_p$ we have~$2\lceil\frac{\ell-1}{\ell}\kappa_p\rceil > \lceil\kappa_p\rceil$. Thus, this case is described in the intersection of the last row and first column in the table above. We observe that the points $P_{[\frac{\kappa_p}{2}]}$ and $Q_{[\frac{\kappa_p}{2}]}$ coincide at the location $([\frac{\kappa_p}{2}],[\frac{\kappa_p}{2}]-1)$.  Then convex combinations of $P_{[\frac{\kappa_p}{2}]}$ and $(0,1)$ form a segment of the line $k= s+1$, and convex combinations of~$P_{[\frac{\kappa_p}{2}]}$ and $K$ form a segment of the line $s \geq k - \frac{\kappa_p -k}{\kappa_p - [\frac{\kappa_p}{2}]}$.  It follows that $\HDR(\Sigma, k, s, \ell, p)$ holds provided~\eqref{HDRSigmaLpShiftCondition},~\eqref{HDRSurfaceCondition}, and~\eqref{UglyCondition} are satisfied.
\end{proof}

\section{Sharpness}\label{S5}
In this section, we consider the case where~$\Sigma$ is the paraboloid~$\{\xi_d = |\xi_{\bar{d}}|^2\}$ as a representative example. We also assume that~$\varphi\in C_0^{\infty}(\mathbb{R}^{d-1})$ is non-negative.

\subsection{Surface measure conditions}\label{s51}
Let~$\chi$ be a smooth function of one variable supported in~$[1,2]$ such that~$\hat{\chi}^{(k)}(0) = 1$.  Consider the functions~$f_n$ defined as
\begin{equation*}
\hat{f}_n(\xi) = \varphi(\xi_{\bar{d}})\hat{\chi}\big(2^n(\xi_d - |\xi_{\bar{d}}|^2)\big).
\end{equation*}
The function~$f_n$ can be written explicitly:
\begin{equation*}
f_n(x) = 2^{-n}\mathcal{F}_{\xi\to x}\Big[\varphi d\mathbb{P}\Big]\chi(2^{-n}x_d ),
\end{equation*}
here~$d\mathbb{P}$ is the Lebesgue measure on the paraboloid~$\Sigma$.
It is easy to observe two formulas:
\begin{equation*}
\frac{\partial^k\hat{f}_n}{\partial \xi_d^k}(\xi) = 2^{nk} \varphi(\xi_{\bar{d}}), \quad \xi \in \Sigma;\qquad
\|f_n\|_{L_p} \asymp 2^{n\sigma_p}.
\end{equation*}
We will also need the functions
\begin{equation*}
f_{\le N} = \sum\limits_{n =0}^N (1+n)^{-1} 2^{-nk}f_n.
\end{equation*} 

\paragraph{\bf Sharpness of~\eqref{RestrictionWithSobolevSurface}.} Assume~$\phi = 1$ on the support of~$\varphi$. We plug~$f_{\le N}$ into~\eqref{RestrictionWithSobolev}. The left hand-side is bounded away from zero by
\begin{equation*}
\Big\|\frac{\partial^k \hat{f}_{\le N}}{\partial \xi_d^k}\Big\|_{H^{-s}(\Sigma)} = \|\varphi\|_{H^{-s}(\Sigma)}\sum\limits_{0 \le n \le N}(1+n)^{-1}  \asymp \log(N).
\end{equation*}
As for the~$L_p$-norm, we note that the functions~$f_n$ have disjoint supports, so,
\begin{equation*}
\|f_{\le N}\|_{L_p} = \Big(\sum\limits_{n\le N}\|f_{n}\|_{L_p}^p\Big)^{\frac{1}{p}} \asymp \Big(\sum\limits_{n \le N} (1+n)^{-p}\, 2^{n p(\sigma_p - k)}\Big)^{\frac{1}{p}}.
\end{equation*}
Since the left hand-side of~\eqref{RestrictionWithSobolev} tends to infinity as~$N\to \infty$, the right hand-side cannot be uniformly bounded. This means~\eqref{RestrictionWithSobolevSurface} holds true if~$p > 1$. In the case~$p=1$, we get~$k \leq \frac{d-1}{2}$ instead.

\paragraph{\bf Necessity of~$p < \frac{2d}{d+1+2k}$ in Theorem~\ref{First_restriction_for_higher_derivatives_Theorem}.} Add the requirements~$\hat{\chi}^{(j)}(0) = 0$ for all~$j < k$. Then,~$f_n$ and~$f_{\le N}$ belong to~$\Si L_p^k$ and the same reasoning gives the necessity of~$k < \sigma_p$ in~\eqref{Restriction_for_higher_derivatives} and~\eqref{Sobolev_Restriction_for_higher_derivatives}, which is exactly~$p < \frac{2d}{d+1+2k}$. As usual, the cases~$k = \frac{d-1}{2}$ are permitted if~$p=1$.

\paragraph{\bf Necessity of~\eqref{HDRSurfaceCondition}.} We plug exactly the same functions~$f_{\le N}$ into~\eqref{HDRinequality}. The~$H^{-s}$-norm on the left hand-side and the~$L_p$-norm on the right hand-side behave in the same manner as previously. Since we have assumed~$\hat{\chi}(0) =0$, there is no summand~$\|\hat{f}\|_{H^\ell(\Sigma)}$ on the right hand-side.

\paragraph{\bf Necessity of~\eqref{SIGaussian}.} As it was mentioned earlier, the quadratic inequality~\eqref{StubbornInequality} is equivalent to its bilinear version~\eqref{BilinearStubborn}. We work with the latter expression here. The functions~$f$ and~$g$ will be constructed from the functions~$f_n$ in a slightly different manner from before. To define~$g$, we take~$\chi$ that satisfies~$\hat{\chi}^{(\alpha)}(0) = 1$ and set~$g = f_0$. For the function~$f$, we require~$\hat{\chi}^{(j)}(0) = 0$ for all~$j < k = \alpha + \beta$ and~$\hat{\chi}^{(\alpha+\beta)}(0) = 1$, and set~$f = f_{\le N}$ (with~$k=\alpha+\beta$). We plug these functions~$f$ and~$g$ into~\eqref{BilinearStubborn} and use the Newton--Leibniz formula (we assume~$\psi=1$ on the support of~$\varphi$)
\begin{equation*}
\bigg|\Big(\frac{\partial}{\partial r}\Big)^{\beta}\Scalprod{ \frac{\partial^{\alpha}\hat{f}}{\partial \xi_d^{\alpha}}\psi}{ \frac{\partial^{\alpha}\hat{g}}{\partial \xi_d^{\alpha}}\psi}_{\dot{H}^{-\gamma}(\Sigma_r)}\bigg|_{r=0}\bigg| = \bigg|\Scalprod{ \frac{\partial^{\alpha+\beta}\hat{f}}{\partial \xi_d^{\alpha+\beta}}\psi}{ \frac{\partial^{\alpha}\hat{g}}{\partial \xi_d^{\alpha}}\psi}_{\dot{H}^{-\gamma}(\Sigma)}\bigg| = \|\varphi\|_{\dot{H}^{-\gamma}(\Sigma)}^2\sum\limits_{n=0}^N\frac{1}{1+n}\asymp \log N.
\end{equation*}
On the right hand-side, we have
\begin{equation*}
\|g\|_{L_p} \asymp 1, \quad \|f_{\le N}\|_{L_p} \asymp \Big(\sum\limits_{n \le N} (1+n)^{-p}\, 2^{n p(\sigma_p - \alpha - \beta)}\Big)^{\frac{1}{p}}.
\end{equation*}
So, the necessity of~\eqref{SIGaussian} is proved.

\paragraph{\bf Necessity of~$\alpha + \beta < \frac{d-1}{p} + \frac{1}{r} - \frac{d+1}{2}$ in Theorem~\ref{StrichartzTheorem}}  This is proved in the same manner as in the previous paragraph. One should only replace the formula for the~$L_p$-norm of~$f_n$ with
\begin{equation*}
\|f_n\|_{L_r(L_p)} \asymp 2^{n(\frac{d-1}{p} + \frac{1}{r} - \frac{d+1}{2})}.
\end{equation*}

\subsection{Knapp examples}\label{s52}
We start with a Schwartz function~$f$ with compactly supported Fourier transform and define the functions~$f_n$ by the formula
\begin{equation}\label{ParabolicRescaling}
f_n(x) = n^{-\frac{d-1}{2}}f\Big(\frac{x_1}{n},\frac{x_2}{n},\ldots,\frac{x_{d-1}}{n},\frac{x_d}{n^2}\Big).
\end{equation}
By homogeneity,
\begin{align*}
\Big\|\frac{\partial^k \hat{f}_n}{\partial \xi^k_d}\Big\|_{\dot{H}^{-s}(\Sigma)} = n^{2k-s}\Big\|\frac{\partial^k \hat{f}}{\partial \xi^k_d}\Big\|_{\dot{H}^{-s}(\Sigma)};\\
\|f_n\|_{L_p} = n^{\kappa_p}\|f\|_{L_p}.
\end{align*}
with the caveat that the homogeneous Sobolev norm may already be infinite if $s \geq \frac{d-1}{2}$.

\paragraph{\bf Necessity of~\eqref{RestrictionWithSobolevKnapp}} This can be obtained by simply plugging~$f_n$ into~\eqref{RestrictionWithSobolev} and assuming~$\phi = 1$ in a neighborhood of the origin.

\paragraph{\bf Necessity of the condition~$2k-s \leq \kappa_p$ in Theorem~\ref{First_restriction_for_higher_derivatives_Theorem}.} We take~$f \in\!\, \Si L_p^k$ and note that~$f_n \in\!\, \Si L_p^k$ as well (recall that~$\Sigma$ is the paraboloid). It remains to plug~$f_n$ into~\eqref{Sobolev_Restriction_for_higher_derivatives} with the same assumption about~$\phi$.

\paragraph{\bf Necessity of~\eqref{HDRKnapp}.} Here we plug~$f_n$ generated by~$f\in\!\, \Si L_p$ into~\eqref{HDRinequality} and note that~$\|\hat{f}_n\|_{H^\ell(\Sigma)}=0$.

\paragraph{\bf Necessity of~\eqref{SIKnapp}} This follows from the formula
\begin{equation*}
\frac{\partial^\beta}{\partial r^\beta}\Big\|\frac{\partial^\alpha \hat{f}_n}{\partial \xi^\alpha_d}\Big\|^2_{\dot{H}^{-\gamma}(\Sigma_r)}\Big|_{r=0}=n^{4\alpha + 2\beta-2\gamma}\frac{\partial^\beta}{\partial r^\beta}\Big\|\frac{\partial^\alpha \hat{f}}{\partial \xi^\alpha_d}\Big\|^2_{\dot{H}^{-\gamma}(\Sigma_r)}\Big|_{r=0}.
\end{equation*}

\paragraph{\bf Necessity of~$2\alpha + \beta -\gamma \leq \frac{d-1}{p} + \frac{2}{r} - \frac{d+3}{2}$ in Theorem~\ref{StrichartzTheorem}} One can prove this in the same manner as in the previous paragraph. One should only replace the formula for the~$L_p$-norm of~$f_n$ with
\begin{equation*}
\|f_n\|_{L_r(L_p)} \asymp 2^{n(\frac{d-1}{p} + \frac{2}{r} - \frac{d+3}{2})}.
\end{equation*}

\subsection{Pure shifts}\label{s53}
We start with a Schwartz function~$f$ and consider its shifts in the~$x_d$ direction:
\begin{equation}\label{ShiftFormula}
f_n(x) = f(x_1,x_2,\ldots,x_{d-1},x_d - n).
\end{equation}
We also assume
\begin{equation*}
\frac{\partial^j \hat{f}}{\partial \xi_{d}^j} = 0\quad \hbox{on $\Sigma$ for }\ j=1,2,\ldots,k
\end{equation*}
and~$\hat{f} = 1$ on the support of~$\phi$. Then~$\hat{f}_n(\xi) = f(\xi)e^{2\pi i n\xi_{d}}$ and
\begin{align*}
\Big\|\phi\frac{\partial^k \hat{f}_n}{\partial \xi_{d}^k}\Big\|_{H^{-s}(\Sigma)} = (2\pi n)^k\big\|\phi e^{2\pi i n|\cdot|^2}\big\|_{H^{-s}} \gtrsim  n^k \Big(n^{-(d-1)}\int\limits_{\mathbb{R}^{d-1}} \Big|\check{\phi}*e^{-2\pi i \frac{|z|^2}{4n}}\Big|^2(z)\cdot(1+|z|)^{-2s}\,dz\Big)^{\frac12}\\
 \gtrsim n^k \Big(n^{-(d-1)}\int\limits_{|z|\lesssim n} (1+|z|)^{-2s}\,dz\Big)^{\frac12} \gtrsim n^{k-s}.
\end{align*}

\paragraph{\bf Necessity of~\eqref{RestrictionWithSobolevShift}} This follows from the fact that~$\|f_n\|_{L_p}$ does not depend on~$n$.

\paragraph{\bf Necessity of~\eqref{HDRShiftCondition}.} Note that~$\|\hat{f}_n\|_{H^{\ell}(\Sigma)}$ does not exceed~$cn^{\ell}$ (this estimate reduces to the product rule in the case~$\ell \in \mathbb{Z}_+$; the general case follows from the case~$\ell \in \mathbb{Z}_+$ by the Cauchy--Schwarz inequality). 
Comparing the left and right parts of~\eqref{HDRinequality}, we get~\eqref{HDRShiftCondition}. 

\paragraph{\bf Necessity of~\eqref{HDRSigmaLpShiftCondition}.} We consider the functions~$f_n$ generated by the rule~\eqref{ShiftFormula} from a function~$f\in\,\! \Si L_p$ such that~$\frac{\partial \hat{f}}{\partial \xi_d} = \varphi$ and~$\varphi = 1$ on the support of~$\phi$, and all higher order (up to order~$k$) derivatives of~$\hat{f}$ vanish on~$\Sigma$. Then,~$f_n \in\,\! \Si L_p$ as well, so, there is no term~$\|\hat{f}_n\|_{H^\ell}$ on the right hand-side of~\eqref{HDRinequality}. However, on left hand-side, we cannot have~$n^{k-s}$, but only have growth~$n^{k-s-1}$ since
\begin{equation*}
\frac{\partial^{k}\hat{f}_n}{\partial \xi_d^k}\Big|_{\Sigma}(\xi_{\bar{d}}) = (2\pi i n)^{k-1}\varphi(\xi_{\bar{d}})e^{2\pi i n|\xi_{\bar{d}}|^2}.
\end{equation*}
Thus,~$n^{k-s-1}$ should be bounded if~\eqref{HDRinequality} holds, which is exactly~\eqref{HDRSigmaLpShiftCondition}.

\paragraph{\bf Necessity of~\eqref{SISpectralShift}.} Consider a Schwartz function~$f$ of~$d$ variables such that for any~$j \in [0..\alpha + \beta]$ we have
\begin{equation}\label{ConditionForSSExample}
\frac{\partial^j \hat{f}}{\partial \xi_d^j} = 1 \quad \hbox{on}\ \Sigma\cap V.
\end{equation} 
Let~$f_n$ be generated by~\eqref{ShiftFormula} from~$f$. We plug~$f_n$ into~\eqref{StubbornInequality}. We first compute the ``interior'' derivative:
\begin{equation*}
\frac{\partial^{\alpha} \hat{f}_n}{\partial \xi_d^{\alpha}}(\xi)\psi(\xi) = \frac{\partial^{\alpha} \big[\hat{f}e^{2\pi i n\xi_d}\big]}{\partial \xi_d^{\alpha}}(\xi)\psi(\xi) = e^{2\pi i n\xi_d}\psi(\xi)\sum\limits_{j=0}^{\alpha}C_{\alpha}^j\frac{\partial^j \hat{f}}{\partial \xi_d^j}(\xi)(2\pi i n)^{\alpha-j}.  
\end{equation*}
Therefore,
\begin{multline}\label{VeryVeryBigFormula}
\Big(\frac{\partial}{\partial r}\Big)^{\beta}\Big\|\frac{\partial^{\alpha} \hat{f}_n}{\partial \xi_d^{\alpha}}\psi\Big\|^2_{\dot{H}^{-\gamma}(\Sigma_r)}\bigg|_{r=0}=
\Big(\frac{\partial}{\partial r}\Big)^{\beta}\Big\|e^{2\pi i n(|\xi_{\bar{d}}|^2+r)}\psi(\xi_{\bar{d}})\sum\limits_{j=0}^{\alpha}C_{\alpha}^j\frac{\partial^j \hat{f}}{\partial \xi_d^j}(\xi_{\bar{d}},|\xi_{\bar{d}}|^2+r)(2\pi i n)^{\alpha-j}\Big\|^2_{\dot{H}^{-\gamma}}\bigg|_{r=0} = \\
\Big(\frac{\partial}{\partial r}\Big)^{\beta}\Big\|e^{2\pi i n|\xi_{\bar{d}}|^2}\psi(\xi_{\bar{d}})\sum\limits_{j=0}^{\alpha}C_{\alpha}^j\frac{\partial^j \hat{f}}{\partial \xi_d^j}(\xi_{\bar{d}},|\xi_{\bar{d}}|^2+r)(2\pi i n)^{\alpha-j}\Big\|^2_{\dot{H}^{-\gamma}}\bigg|_{r=0} = \\
\sum\limits_{k=0}^{\beta}C_{\beta}^k\bigg\langle e^{2\pi i n|\xi_{\bar{d}}|^2}\psi(\xi_{\bar{d}})\sum\limits_{j=0}^{\alpha}C_{\alpha}^j\frac{\partial^{j+k} \hat{f}}{\partial \xi_d^{j+k}}(\xi_{\bar{d}},|\xi_{\bar{d}}|^2)(2\pi i n)^{\alpha-j},\\
e^{2\pi i n|\xi_{\bar{d}}|^2}\psi(\xi_{\bar{d}})\sum\limits_{j=0}^{\alpha}C_{\alpha}^j\frac{\partial^{j+\beta-k} \hat{f}}{\partial \xi_d^{j+\beta-k}}(\xi_{\bar{d}},|\xi_{\bar{d}}|^2)(2\pi i n)^{\alpha-j}\bigg\rangle_{\dot{H}^{-\gamma}} \stackrel{\scriptscriptstyle\eqref{ConditionForSSExample}}{=}\\
\sum\limits_{k=0}^{\beta}C_{\beta}^k\bigg\langle e^{2\pi i n|\xi_{\bar{d}}|^2}\psi(\xi_{\bar{d}})\sum\limits_{j=0}^{\alpha}C_{\alpha}^j(2\pi i n)^{\alpha-j},e^{2\pi i n|\xi_{\bar{d}}|^2}\psi(\xi_{\bar{d}})\sum\limits_{j=0}^{\alpha}C_{\alpha}^j(2\pi i n)^{\alpha-j}\bigg\rangle_{\dot{H}^{-\gamma}} = \\
2^{\beta}(1+2\pi in)^{2\alpha} \Big\|e^{2\pi i n|\cdot|^2}\psi(\cdot)\Big\|^2_{\dot{H}^{-\gamma}}=
2^{\beta}(1+2\pi in)^{2\alpha}\int\limits_{\mathbb{R}^{d-1}} \Big|\big[\hat{\psi}*\big(n^{-\frac{d-1}{2}} e^{2\pi i \frac{|\zeta|^2}{4n}}\big)\big](z)\Big|^2|z|^{-2\gamma}\,dz \\ \gtrsim n^{2\alpha-d+1}\!\!\!\int\limits_{|z|\lesssim n}|z|^{-2\gamma}\,dz \gtrsim n^{2\alpha-2\gamma}.
\end{multline}
Thus, the left hand-side of~\eqref{StubbornInequality} grows at least as fast as~$n^{2\alpha-2\gamma}$, whereas the right hand-side does not change. This proves the necessity of the condition~\eqref{SISpectralShift}.

\paragraph{\bf Necessity of condition~$\gamma \geq \alpha$ in Theorem~\ref{StrichartzTheorem}} This is obtained by completely the same method in the case~$r \leq 2$. For the case~$r \geq 2$, we can only prove the necessity of the non-strict inequality~$\gamma - \alpha \geq \frac12 - \frac1r$. For that we slightly modify the construction above. We consider the function
\begin{equation*}
F_n = \sum\limits_{j=n}^{2n}\eps_jf_{Aj},
\end{equation*}
where the functions~$f_j$ are generated by~\eqref{ShiftFormula},~$A$ is a sufficiently large number, and~$\eps_j$ are randomly chosen signs. Then,
\begin{equation*}
\mathbb{E}\Big(\frac{\partial}{\partial r}\Big)^{\beta}\Big\|\frac{\partial^{\alpha} \hat{F}_n}{\partial \xi_d^{\alpha}}\psi\Big\|^2_{\dot{H}^{-\gamma}(\Sigma_r)}\bigg|_{r=0} = \sum\limits_{j=n}^{2n} \Big(\frac{\partial}{\partial r}\Big)^{\beta}\Big\|\frac{\partial^{\alpha} \hat{f}_{Aj}}{\partial \xi_d^{\alpha}}\psi\Big\|^2_{\dot{H}^{-\gamma}(\Sigma_r)}\bigg|_{r=0}\stackrel{\scriptscriptstyle{\eqref{VeryVeryBigFormula}}}{\gtrsim} n^{2\alpha - 2\gamma + 1}.
\end{equation*}
On the other hand, disregarding the choice of the signs~$\eps_j$,
\begin{equation*}
\|F_n\|_{L_r(L_p)} \asymp n^{\frac1r}
\end{equation*}
provided~$A$ is sufficiently large (this number is needed to diminish the influence of Schwartz tails on this almost orthogonality). It remains to choose~$\eps_j$ with the largest possible quantity on the right hand-side and compare the two sides.

\paragraph{\bf Necessity of condition~$p \leq 2$ in Theorem~\ref{StrichartzTheorem}} This can be obtained by a construction similar to the one described in the previous paragraph, except with functions $f_n$ shifted in the~$x_1$ direction instead of the~$x_d$ direction.

\subsection{Shifted Knapp example}\label{s54}
We need to modify the classical Knapp construction to get the necessity of~\eqref{HDRShiftedKnapp}. We take some sequence~$\{D_n\}_n$ and modify the functions~$f_n$ generated by~\eqref{ParabolicRescaling}. Now we also shift them:
\begin{equation*}
f_n = n^{-d-1}f\Big(\frac{x_1}{n},\frac{x_2}{n},\ldots,\frac{x_{d-1}}{n},\frac{x_d - D_n}{n^2}\Big).
\end{equation*}
We require~$D_n \gg n^2$ and do not require the vanishing~$f\in \!\,\Si L_p$. The~$L_p$-norms are influenced by scaling but do not depend on the size of the shifts:
\begin{equation*}
\|f_n\|_{L_p} \asymp n^{(d+1)(\frac{1}{p} - 1)}.
\end{equation*} 
Let~$\hat{f}(\zeta,|\zeta|^2)$ be~$g(\zeta)$, here~$g$ is a smooth function, let us assume it is compactly supported and has non-zero integral. Then,
\begin{equation*}
\|\hat{f}_n\|_{\dot{H}^\ell(\Sigma)} = \Big\|g(n\cdot)e^{2\pi i D_n|\cdot|^2}\Big\|_{\dot{H}^\ell} = n^{\ell - \frac{d-1}{2}}\Big\|g(\cdot)e^{2\pi i \frac{D_n}{n^2}|\cdot|^2}\Big\|_{\dot{H}^\ell} \lesssim n^{\ell - \frac{d-1}{2}}\Big(\frac{D_n}{n^2}\Big)^\ell.
\end{equation*}
The latter estimate can be proved via the product rule for the case~$\ell \in \mathbb{Z}_+$ and reduced to this case with the help of the Cauchy--Schwarz inequality. Similarly,
\begin{multline*}
\Big\|\frac{\partial^k\hat{f}_n}{\partial \xi_d^k}\Big\|_{\dot{H}^{-s}(\Sigma)} \asymp D_n^k \Big\|g(n\cdot)e^{2\pi i D_n|\cdot|^2}\Big\|_{\dot{H}^{-s}} = D_n^k n^{-s - \frac{d-1}{2}}\Big\|g(\cdot)e^{2\pi i \frac{D_n}{n^2}|\cdot|^2}\Big\|_{\dot{H}^{-s}} \gtrsim\\ D_n^k n^{-s - \frac{d-1}{2}}\Big(\frac{D_n}{n}\Big)^{-\frac{d-1}{2}} \Big(\int\limits_{|z| \lesssim \frac{D_n}{n^2}}|z|^{-2s}\Big)^\frac12 \gtrsim 
D_n^k n^{-s - \frac{d-1}{2}} \Big(\frac{D_n}{n^2}\Big)^{-s}.
\end{multline*}
So, if~\eqref{HDRinequality} is true, then
\begin{equation}\label{InequalityForNumbers}
D_n^{k-s} n^{s -\frac{d-1}{2}} \lesssim D_n^{\ell} n^{-\ell - \frac{d-1}{2}} + n^{(d+1)(\frac{1}{p}-1)}
\end{equation}
whenever~$D_n \gg n^2$. We recall~$k -s \leq \ell$ by~\eqref{HDRShiftCondition} (the necessity of which is already proved), so, the first term on the right dominates the left hand-side when~$D_n$ is sufficiently large. We want to make~$D_n$ as small as possible in such a way that the left hand-side is still greater than the second summand on the right. Let
\begin{equation*}
D_n = n^{\frac{\kappa_p-s}{k-s}}\log n.
\end{equation*}
Note that such a choice of~$D_n$ guarantees~$D_n \gg n^2$ by~\eqref{HDRKnapp} and the assumption~$k > s$. Plugging it back to~\eqref{InequalityForNumbers}, we get
\begin{equation*}
n^{\kappa_p - s}n^{s - \frac{d-1}{2}}(\log n)^{k-s} \lesssim n^{\frac{(\kappa_p -s)\ell}{k-s} - \ell -\frac{d-1}{2}}(\log n)^{\ell},
\end{equation*}
which, after a tiny portion of algebra and~\eqref{HDRShiftCondition}, leads to
\begin{equation*}
\frac{k \ell}{s + \ell - k}\leq \kappa_p,
\end{equation*}
which is~\eqref{HDRShiftedKnapp}.

\section{Additional lemmas and supplementary material}\label{S6}

\subsection{Localization argument}\label{s61}
We need to localize the~$\HDR$ inequalities and also replace the gradient with a single directional derivative. Namely, we want to reduce~$\HDR(\Sigma,k,s,\ell,p)$ to a collection of statements~$\HDRL(h,k,s,\ell,p)$ defined below. A similar principle works for inequalities of the type~\eqref{RestrictionWithSobolev},~\eqref{Restriction_for_higher_derivatives},~\eqref{Sobolev_Restriction_for_higher_derivatives} and the proof is completely identical. 
\begin{Def} \label{def6.1}
Let the numbers~$k,s,\ell,p$ be of the same nature as in Definition~\textup{\ref{HDR}}. Let~$U$ be a neighborhood of the origin in~$\mathbb{R}^{d-1}$, let~$h\colon U \to \mathbb{R}$ be a smooth function such that~$h(0) = 0$,~$\nabla h(0) = 0$, and the determinant of the Hessian of~$h$ at the origin does not vanish. Further, we assume~\eqref{HessianSmoothness}.
We say that the statement~$\HDRL(h,k,s,\ell,p)$ holds true if the inequality
\begin{equation*}
\Big\|\psi(\cdot)\frac{\partial^{k}\hat{f}}{\partial \xi_d^k}(\cdot,h(\cdot))\Big\|_{H^{-s}(\mathbb{R}^{d-1})}\lesssim_\psi \|f\|_{L_p} + \big\|\psi(\cdot)\hat{f}\big(\cdot,h(\cdot)\big)\big\|_{H^\ell(\mathbb{R}^{d-1})}
\end{equation*}
holds true for any smooth function~$\psi$ supported in~$U$.
\end{Def}

\begin{Le}\label{LocalizationLemma}
The statement~$\HDR(\Sigma,k,s,\ell,p)$ is true provided the statement~$\HDRL(h,k,s,\ell,p)$ is true for any~$h$ satisfying the conditions of Definition~\ref{def6.1}.
\end{Le}
\begin{proof}
We need to prove~\eqref{HDRinequality} with a fixed compactly supported smooth function~$\phi$. We find a smooth partition of unity~$\{\Phi_n\}_n$ on~$\Sigma$, each function~$\Phi_n$ supported in a small ball~$V_n$ and each~$V_n$ lies in a chart neighborhood of a certain point~$\xi_n \in \Sigma$. For each~$n$ fixed, we identify~$\xi_n$ with the origin of~$\mathbb{R}^d$, the tangent plane~$T_{\xi_n}\Sigma$ with~$\mathbb{R}^{d-1}$, and get a graph representation for~$\Sigma\cap V_n$:
\begin{equation*}
\Sigma \cap V_n = \{(\zeta,h_n(\zeta))\mid \zeta \in U_n\},
\end{equation*}
where~$U_n$ is a neighborhood of the origin in~$\mathbb{R}^{d-1}$. If the partition of unity is sufficiently fine, then the function~$h_n$ satisfies~\eqref{HessianSmoothness}. We estimate the left hand-side of~\eqref{HDRinequality} by the triangle inequality
\begin{equation*}
\|\phi \nabla^k \hat{f}\|_{H^{-s}(\Sigma)} \leq \sum_n \|\phi\Phi_n \nabla^k \hat{f}\|_{H^{-s}(\Sigma)}.
\end{equation*} 
Note that the sum on the right is, in fact, finite. We fix~$n$. We are going to use the following algebraic fact: there exists a finite collection of vectors~$v_n$ in~$\mathbb{R}^d$ such that any homogeneous polynomial of degree~$k$ is a linear combination of the monomials~$\scalprod{\cdot}{v_n}^k$; moreover, such vectors~$v_n$ may be chosen arbitrarily close to any fixed vector. Since the determinant of the Hessian of~$h_n$ is non-zero, the normals~$\boldsymbol{\normal}_\zeta$ to~$\Sigma$ at the points~$(\zeta,h_n(\zeta))$ cover a neighborhood of the vector~$(0,0,\ldots,1)$ in~$S^{d-1}$ (the unit sphere in~$\mathbb{R}^d$). Thus, we may choose finitely many points~$\zeta_j$ in a sufficiently small neighborhood of the origin such that
\begin{align}
\forall \alpha \in \mathbb{Z}^d_+ \hbox{ such that } |\alpha| = k \qquad \frac{\partial^\alpha}{\partial \xi^\alpha} \hbox{ is a linear combination of }  \Big\{\frac{\partial^k}{\partial \normal_{\zeta_j}^k}\Big\}_j;\\
\label{TransversalityAssumption} \normal_{\zeta_j}\nparallel T_{\zeta}\Sigma \hbox{ for any~$j$ and any } \zeta \in U_n.
\end{align}
This allows us to write the estimate
\begin{equation}\label{SumOverj}
\|\phi\Phi_n \nabla^k \hat{f}\|_{H^{-s}(\Sigma)} \lesssim \sum_j\Big\|\phi\Phi_n \frac{\partial^k\hat{f}}{\partial \normal_{\zeta_j}^k}\Big\|_{H^{-s}(\Sigma)}.
\end{equation}
Now we restrict our attention to each point~$\zeta_j$ individually. We adjust our coordinates to this point: now~$\zeta_j$ is the origin, we also identify~$T_{\zeta_j}\Sigma$ with~$\mathbb{R}^{d-1}$. The summand corresponding to~$j$ on the right hand-side of the previous inequality   transforms into
\begin{equation*}
\Big\|\Psi \frac{\partial^k \hat{f}}{\partial \xi_d^k}\Big\|_{H^{-s}(\Sigma)},
\end{equation*}
where~$\Psi$ is a certain smooth function supported in~$V_n$. By the assumption~\eqref{TransversalityAssumption},
\begin{equation*}
\Sigma \cap V_n = \Big\{\big(\zeta,h_{n,j}(\zeta)\big)\,\Big|\; \zeta \in U_{n,j}\Big\},
\end{equation*}
where~$U_{n,j}$ is a neighborhood of the origin in~$\mathbb{R}^{d-1}$, and~$h_{n,j}$ satisfies~\eqref{HessianSmoothness} (with the constant~$\frac14$ instead of~$\frac{1}{10}$ possibly). Take a smooth non-negative function~$\psi$ that is supported in~$U_{n,j}$ and is bounded away from zero on the projection of the support of~$\Psi$ to~$\mathbb{R}^{d-1}$. Then, clearly,
\begin{equation*}
\Big\|\Psi \frac{\partial^k \hat{f}}{\partial \xi_d^k}\Big\|_{H^{-s}(\Sigma)} \lesssim \Big\|\psi(\cdot) \frac{\partial^k \hat{f}}{\partial \xi_d^k}(\cdot,h(\cdot))\Big\|_{H^{-s}(\mathbb R^{d-1})}.
\end{equation*}
We also note that the norms
\begin{equation*}
\|g\|_{H^{-s}(\Sigma)}\quad \hbox{and}\quad \|g\big(\cdot,h_{n,j}(\cdot)\big)\|_{H^{-s}(\mathbb{R}^{d-1})}
\end{equation*}
are comparable for functions~$g$ supported on~$\Sigma\cap V_n$. Thus, by~$\HDRL(h_{n,j},k,s,\ell,p)$, we may bound each summand in~\eqref{SumOverj} by
\begin{equation*}
\|f\|_{L_p} + \|\psi(\cdot) \hat{f}\psi(\cdot,h(\cdot))\|_{H^\ell(\mathbb{R}^{d-1})} \lesssim\|f\|_{L_p} + \|\Psi \hat{f}\|_{H^\ell(\Sigma)}.
\end{equation*}
It remains to note that we have a finite number of summands both over~$j$ and~$n$.
\end{proof}
\begin{Rem}\label{GeneralLocalization}
Consider Banach spaces~$X_1, X_2,\ldots, X_m$ of functions on~$\Sigma$ such that multiplication operators
\begin{equation*}
\varphi \mapsto \psi\varphi, \quad \varphi \in X_m,
\end{equation*}
are bounded on~$X_m$ whenever~$\psi \in C_0^{\infty}(\Sigma)$. The inequality
\begin{equation*}
\big\|\phi (\nabla^k \hat{f})\big|_{\Sigma}\big\|_{H^{-s}(\Sigma)} \lesssim_{\phi}\Big(\|f\|_{L_p(\mathbb{R}^d)} + \sum\limits_{j=1}^m\|\Phi \hat{f}\|_{X_j}\Big)
\end{equation*}
may be reduced to local form
\begin{equation*}
\Big\|\psi(\cdot)\frac{\partial^{k}\hat{f}}{\partial \xi_d^k}(\cdot,h(\cdot))\Big\|_{H^{-s}(\mathbb{R}^{d-1})}\lesssim_\psi \|f\|_{L_p} + \sum\limits_{j=1}^m\|\hat{f}\|_{X_j}, \quad\supp\psi \subset U,
\end{equation*}
and~$U$ satisfies the usual assumptions, with the same argument as in the proof of Lemma~\ref{LocalizationLemma}. In particular, the case~$X_j = \{0\}$ allows to reduce~$\Rw^k(\Sigma,p,s)$ to~$\R^k(\Sigma,p,s)$ \textup(see Definitions~\ref{StrongRestriction} and~\ref{WeakRestriction}\textup).
\end{Rem}

\subsection{A version of the Stein--Weiss inequality}\label{s62}

\subsubsection{Case~$p \in [1,2]$}
Let~$L_p(w)$ be the weighted Lebesgue space:
\begin{equation*}
f \in L_p(w)\quad \Leftrightarrow\quad fw \in L_p.
\end{equation*}
Let also~$\Conv_b$ be the operator of convolution with the function~$(1+|x|)^{-b}$. In this section, we work with functions on~$\mathbb{R}$.
\begin{Th}\label{SteinWeissModified}
Let~$a\geq 0$, let~$p \in [1,2]$. The operator~$\Conv_b$ maps the space~$L_p((1+|x|)^a)$ to its dual space~$L_{p'}((1+|x|)^{-a})$ if 
\textup{\begin{enumerate}
\item $b\leq 0$ \emph{and} \begin{itemize}
\item $p=1$ \emph{and} $a+b \geq 0$\textup;
\item $p > 1$ \emph{and}~$a + b > 1 - \frac{1}{p}$\textup;
\end{itemize}
\item $b \in (0,1)$ \emph{and}~$2a+b \geq 2 - \frac{2}{p}$\textup;
\item $b=1$ \emph{and} \begin{itemize}
\item~$p < 2$\textup;
\item~$p=2$ \emph{and}~$a > 0$\textup;
\end{itemize}
\item $b > 1$.
\end{enumerate}}
\end{Th}		
Theorem~\ref{SteinWeissModified} is a variation on the classical Stein--Weiss inequality from~\cite{SteinWeiss}. In the classical setting, the convolutional kernel and weights are homogeneous.
\begin{Rem}
The conditions listed in Theorem~\textup{\ref{SteinWeissModified}} are also necessary.
\end{Rem}
\begin{Rem}\label{Addition}
The boundedness of~$\Conv_b$ as an operator between~$L_p((1+|x|)^a)$ and~$L_{p'}((1+|x|)^{-a})$ is equivalent to the~$L_p\to L_{p'}$ boundedness of the integral operator~$\T_{a,b}$ with the kernel
\begin{equation*}
K_{a,b}(x,y) = (1+|x|)^{-a}(1+|y|)^{-a}(1+|x-y|)^{-b}.
\end{equation*} 
\end{Rem}
\begin{Rem}
One can restate Theorem~\textup{\ref{SteinWeissModified}} like this. The operator~$\T_{a,b}$ maps~$L_p((1+|x|)^a)$ to~$L_{p'}((1+|x|)^{-a})$ if~$a+b > 1-\frac1p$ and~$2a+b \geq 2-\frac2p$ with two exceptional cases. If~$p=1$, then the first inequality might turn into equality, and if~$a=0$,~$b=1$, and~$p=2$, then~$\T_{a,b}$ is not continuous.
\end{Rem}

\begin{proof}[Proof of Theorem~\textup{\ref{SteinWeissModified}}]

We will study the cases~$p=1$ and~$p=2$ and then use interpolation (note that we can plug complex~$a$ and~$b$ in the formula for~$K_{a,b}$ and passing to complex parameters~$a$ and~$b$ does not make the kernel worse since for real~$a$ and~$b$ it is positive).

\paragraph{\bf Case~$p=1$.}
In the space~$L_1$, each element of the unit ball is a convex combination of point masses. Thus, it suffices to prove the uniform boundedness of~$\T_{a,b}$ on measures~$\delta_{z}$,~$z\in \mathbb{R}$. Clearly,
\begin{equation*}
\|\delta_z\|_{L_1((1+|x|)^{a})} = (1+|z|)^{a}
\end{equation*}
(formally, a~$\delta$-measure does not belong to~$L_1$, however, we may work with the larger space of measures instead). Moreover,
\begin{equation*}
\big[\delta_z*(1+|\cdot|)^{-b}\big](x) = (1+|z-x|)^{-b}
\end{equation*}
and
\begin{equation*}
\big\|\delta_z*(1+|\cdot|)^{-b}\big\|_{L_{\infty}((1+|x|)^{-a})} = \sup_x \;(1+|z-x|)^{-b}(1+|x|)^{-a} \stackrel{\scriptscriptstyle{a+b \geq 0}}{\lesssim} \max\Big((1+|z|)^{-b},(1+|z|)^{-a}\Big).
\end{equation*}
Thus, the operator~$\Conv_b$ maps~$L_1((1+|x|)^a)$ to its dual if and only if~$a+b \geq 0$ (we have assumed~$a \geq 0$).

\paragraph{\bf Case~$p=2$.}
We will be applying Schur's test with the function~$\rho(x) = (1+|x|)^c$, where~$c$ is a parameter to be chosen later. Since our kernel~$K_{a,b}$ is symmetric, it suffices to verify
\begin{equation*}
\int K_{a,b}(x,y)\rho(x)\,dx \lesssim \rho(y),
\end{equation*}
which, in our case is rewritten as
\begin{equation*}
\int(1+|x|)^{c-a}(1+|x-y|)^{-b}\,dx \lesssim (1+|y|)^{c+a}.
\end{equation*}
We estimate the integral on the left by splitting it into three parts (around~$x=0$, around~$x=y$, and around infinity):
\begin{multline}\label{ThreeIntegrals}
\int(1+|x|)^{c-a}(1+|x-y|)^{-b}\,dx \lesssim \\\int\limits_{|x|\geq 2|y|}(1+|x|)^{c-a-b}\,dx + \int\limits_{|x|\leq |y|}(1+|x|)^{c-a}\,dx\cdot(1+|y|)^{-b} + \int\limits_{|x|\leq |y|}(1+|x|)^{-b}\,dx\cdot(1+|y|)^{c-a}.
\end{multline}
We restrict our choice of~$c$ to the region~$c < a+b - 1$ to ensure that the first integral converges. Assume for a while that~$b\ne 1$ and~$c-a \ne -1$. Then,
\begin{equation*}
\begin{aligned}
\int\limits_{|x|\geq 2|y|}(1+|x|)^{c-a-b}\,dx &\asymp (1+|y|)^{c-a-b+1};\\
\int\limits_{|x|\leq |y|}(1+|x|)^{c-a}\,dx &\asymp (1+|y|)^{\max(c-a+1,0)};\\
\int\limits_{|x|\leq |y|}(1+|x|)^{-b}\,dx &\asymp (1+|y|)^{\max(-b+1,0)}.
\end{aligned}
\end{equation*}
Thus, we need to prove the inequalities
\begin{equation*}
\begin{aligned}
c-a-b+1 &\leq c+a;\\
\max(c-a+1,0) - b &\leq c+a;\\
\max(-b+1,0) + c- a &\leq c+a.
\end{aligned}
\end{equation*}
The first and the third inequalities do not depend on~$c$ and follow from~$2a+b \geq 1$ and~$a \geq 0$. As for the second one, we take~$c = a+b - 1 -\eps$, where~$\eps$ is sufficiently small number, and the second inequality becomes
\begin{equation}\label{SecondINequality}
\max(b-\eps,0)\leq 2(a+b) - 1-\eps.
\end{equation}

If~$b > 0$, then this inequality follows from~$2a+b \geq 1$, and the case~$b > 0$ is proved except for~$b=1$. 

In the case~$b=1$, the third integral in~\eqref{ThreeIntegrals} is not bounded by a constant, but grows logarithmically at infinity. If~$a> 0$, then~$2a+b > 1$, and the Schur's test is still applicable. If~$a=0$, then the operator is not continuous.

If~$b \leq0$, the inequality~\eqref{SecondINequality} follows from~$a+b > \frac12$.

\paragraph{\bf Interpolation.}
First, we exclude the case~$a=0$, which reduces to the classical Hardy--Littlewood--Sobolev inequality. Assume~$a\ne 0$ and~$p \in (1,2)$ in what follows.

We choose~$\theta$ such that~$\frac{1}{p} = \theta + \frac{1-\theta}{2}$, in other words,~$\theta = \frac{2}{p}-1$, and introduce an analytic operator-valued function
\begin{equation*}
t\mapsto \T_{a(t),b(t)}, \quad \hbox{where}\quad (a(t),b(t)) = \Big(\frac{a(1-t)}{1-\theta},\frac{b(1-t)}{1-\theta}\Big).
\end{equation*}
Here~$\Re t \in [0,1]$. Note that
\begin{equation}\label{Borders}
\begin{aligned}
\big\|\T_{a(t),b(t)}\big\|_{L_1\to L_{\infty}} \lesssim 1, \quad &\Re t = 1;\\
\big\|\T_{a(t),b(t)}\big\|_{L_2\to L_2} \lesssim 1, \quad &\Re t = 0,
\end{aligned}
\end{equation}
since the absolute value of the kernel does not depend on the imaginary part of~$t$ (and moreover,~$\frac{a+b}{1-\theta} > \frac12$ since~$a+b > 1 - \frac{1}{p}$ and~$\frac{2a+b}{1-\theta} \geq 1$ since~$2a+b \geq 2-\frac2p$; note that these inequalities are sufficient for~\eqref{Borders} since we have excluded the case~$a=0$). Thus, by interpolation of analytic families of operators (see~\cite{SteinBook}, Ch. 9, \S 1.2.5),~$\T_{a,b}$ maps~$L_p$ to~$L_{p'}$.
 \end{proof}

\subsubsection{Case~$p > 2$}
We start from the endpoint case.
\begin{Le}
The operator~$\Conv_b$ maps~$L_{\infty}((1+|x|)^a)$ to its dual~$L_1((1+|x|)^{-a})$ if 
\begin{enumerate}
\item $a+b > 1$;
\item $2a+b > 2$;
\item $a > \frac12$.
\end{enumerate}
\end{Le}
\begin{proof}
Let~$\|f\|_{L_{\infty}((1+|x|)^\alpha)}\leq 1$, in other words,
\begin{equation*}
\forall x\in \mathbb{R}\qquad |f(x)|\leq (1+|x|)^{-a},
\end{equation*}
and, thus,
\begin{equation*}
|\Conv_b[f]|(z)\leq \int(1+|x|)^{-a}(1+|z-x|)^{-b}\,dx\stackrel{\scriptscriptstyle a+b > 1}{\lesssim} (1+|z|)^{1-a-b} + (1+|z|)^{-a + \max(1-b,0)} + (1+|z|)^{-b + \max(1-a,0)}.
\end{equation*}
We have used the same principle as in~\eqref{ThreeIntegrals}. Consequently, the conditions
\begin{align*}
&2a+b > 2;\\
&2a+b > 2 \quad \hbox{and}\quad 2a > 1;\\
&2a+b > 2 \quad \hbox{and}\quad a+b > 1
\end{align*}
are sufficient for the boundedness of 
\begin{equation*}
\Big\|\Conv_b[f]\Big\|_{L_1((1+|x|)^{-a})} \lesssim \int (1+|z|)^{1-2a - b}\,dz + \int(1+|z|)^{-2a + \max(1-b,0)}\,dz + \int(1+|z|)^{-a-b + \max(1-a,0)}\,dz.
\end{equation*}
\end{proof}
\begin{Rem}
Following the same lines, we get the endpoint case inequalities
\begin{equation*}
\Conv_b\colon L_{\infty}((1+|x|)^\alpha) \to L_{1,\infty}((1+|x|)^{-\alpha})
\end{equation*}
when~$a=\frac12$ or~$2a+b = 2$. In the case~$a+b=1$, seemingly, there is no limiting inequality.
\end{Rem}
\begin{Th}\label{SteinWeissp>2}
Let~$p \in (2,\infty]$, and let
\begin{enumerate}
\item $a+b > 1 - \frac{1}{p}$;
\item $2a+b > 2-\frac{2}{p}$;
\item $a > \frac12 -\frac{1}{p}$.
\end{enumerate}
Then the operator~$\Conv_b$ maps~$L_{p}((1+|x|)^a)$ to~$L_{p'}((1+|x|)^{-a})$.
\end{Th}
\begin{proof}
We write~$\frac{1}{p} = \frac{\theta}{2}$, i.e.~$\theta = \frac{2}{p}$ and consider the analytic family of operators
\begin{equation*}
t \mapsto \T_{a(t),b},\quad a(t) = a + \frac{\theta-t}{2}.
\end{equation*}
In the case~$\Re t = 1$, the parameters~$(\Re a(t),b)$ satisfy the hypothesis of Theorem~\ref{SteinWeissModified}, so, we have the~$L_2\to L_2$ boundedness there. On the line~$\Re t=0$, we have the~$L_{\infty}\to L_1$ boundedness. The application of the interpolation lemma finishes the proof.
\end{proof}
\begin{Rem}
The endpoint cases are more intriguing here. Using real interpolation, one can prove the bound
\begin{equation*}
\Conv_b\colon L_{p,q}((1+|x|)^a) \to L_{p',q}((1+|x|)^{-a}),\quad\hbox{for any } q\in [1,\infty], p\in [2,\infty)
\end{equation*}
provided~$a+b >1-\frac{1}{p}$,~$2a+b \geq 2-\frac{2}{p}$, and~$a \geq \frac12 -\frac{1}{p}$. Note, however, that the Lorentz spaces here should be defined via the formula
\begin{equation*}
f\in L_{p,q}(w)\quad \Leftrightarrow \quad fw\in L_{p,q},
\end{equation*}
otherwise, the natural interpolation formulas for weighted spaces do not work \textup(see~\cite{Ferreyra}\textup).
\end{Rem}

\subsection{Some endpoint estimates}\label{s63}
To formulate the endpoint version of inequality~\eqref{RestrictionWithSobolev}, we need some Besov spaces (see~\cite{BerghLofstrom}). Given a function~$f$, we define the Besov~$B_2^{-\frac{d-1}{2},\infty}$ norm by the formula
\begin{equation*}
\|f\|_{B_{2}^{-\frac{d-1}{2},\infty}} = \sup_{k \geq 0}\ 2^{-\frac{d-1}{2}k}\|P_k f\|_{L_2},
\end{equation*}
where the~$P_k$,~$k\geq 1$, are the Littlewood--Paley projectors on the annuli~$B_{2^k}(0) \setminus B_{2^{k-1}}(0)$ and~$P_0$ is the spectral projector on the unit ball~$B_1(0)$ (the symbol~$B_r(x)$ denotes the~$(d-1)$-dimensional Euclidean ball of radius~$r$ centered at~$x$). Using the standard properties of Besov spaces, one may then define Besov spaces on smooth submanifolds of~$\mathbb{R}^d$ as well as on their reasonable subdomains. 
\begin{St}\label{WeightedBesov}
The inequality
\begin{equation*}
 \big\|\hat{g}(\cdot,h(\cdot))\psi(\cdot)\big\|_{B_{2}^{-\frac{d-1}{2},\infty}(\mathbb{R}^{d-1})} \lesssim_{\psi} \|g\|_{L_1((1+|x_d|)^{-\frac{d-1}{2}})}
\end{equation*}
is true for any~$h$ and~$\psi$ satisfying the standard requirements.
\end{St}
The norm in the weighted space on the right hand-side is given by the formula
\begin{equation*}
\|g\|_{L_1((1+|x_d|)^{-\frac{d-1}{2}})} = \int\limits_{\mathbb{R}^d}|g(x)|(1+|x_d|)^{-\frac{d-1}{2}}\,dx.
\end{equation*}
Similarly,~$f\in L_p((1+|x_d|)^{-\alpha})$  whenever~$f(x)(1+|x_d|)^{-\alpha} \in L_p$.
\begin{proof}
Since the delta measures are the extremal points of the unit ball in the space of measures, it suffices to prove the proposition for the case where~$g$ is a delta measure:
\begin{equation*}
\Big\|e^{2\pi i(\scalprod{x_{\bar{d}}}{\cdot} + x_d h(\cdot))}\psi(\cdot)\Big\|_{B_2^{-\frac{d-1}{2},\infty}} \lesssim (1+|x_d|)^{-\frac{d-1}{2}}.
\end{equation*}
By the Van der Corput Lemma (for $h(\cdot) = |\cdot|^2$, this is also the Schr\"odinger dispersive bound),
\begin{equation*}
\Big\|\mathcal{F}_{\zeta}\Big[e^{2\pi i(\scalprod{x_{\bar{d}}}{\zeta} + x_d h(\zeta))}\psi(\zeta)\Big]\Big\|_{L_{\infty}}\lesssim
(1+|x_d|)^{-\frac{d-1}{2}}.
\end{equation*}
Thus, we need to prove the inequality
\begin{equation*}
\sup_{k\geq 0}\bigg(2^{-k(d-1)}\!\!\!\!\!\!\!\!\!\!\!\!\int\limits_{B_{2^k}(0)\setminus B_{2^{k-1}}(0)}\!\!\!\!\!\! (1+|x_d|)^{-(d-1)}\,dy\bigg)^{\frac12}\lesssim (1+|x_d|)^{-\frac{d-1}{2}},
\end{equation*}
which is obvious.
\end{proof}
\begin{Cor}\label{EndpointCaseWithoutWeight}
Let~$d$ be odd. If we apply Proposition~\ref{WeightedBesov} with the function~$g(x) = (-2\pi ix_d)^{\frac{d-1}{2}}f(x)$, we get the local form of the endpoint case in~\eqref{RestrictionWithSobolev}:
\begin{equation*}
 \bigg\|\frac{\partial^{\frac{d-1}{2}}\hat{f}}{\partial \xi_d^{\frac{d-1}{2}}}(\cdot,h(\cdot))\psi(\cdot)\bigg\|_{B_{2}^{-\frac{d-1}{2},\infty}(\mathbb{R}^{d-1})} \lesssim_{\psi} \|f\|_{L_1}.
\end{equation*}
Using Remark~\ref{GeneralLocalization}, we may pass to the global form:
\begin{equation*}
\big\|(\phi \nabla^{\frac{d-1}{2}}\hat{f})\big|_{\Sigma}\big\|_{B^{-\frac{d-1}{2},\infty}_2(\Sigma)} \lesssim_\phi \|f\|_{L_1(\mathbb{R}^d)}.
\end{equation*}
Since~$B^{-\frac{d-1}{2},\infty}_2 \hookrightarrow H^{-s}$ for~$s > \frac{d-1}{2}$, we also have
\begin{equation*}
\big\|(\phi \nabla^{\frac{d-1}{2}}\hat{f})\big|_{\Sigma}\big\|_{H^{-s}(\Sigma)} \lesssim_\phi \|f\|_{L_1(\mathbb{R}^d)}
\end{equation*}
for~$s > \frac{d-1}{2}$.
\end{Cor}
Now we will show how to derive Theorem~\ref{CorollaryOfChoGuoLee} from the case~$k=0$ considered in~\cite{ChoGuoLee} and Proposition~\ref{WeightedBesov}. We consider the inequality
\begin{equation*}
 \big\|\hat{g}(\cdot,h(\cdot))\psi(\cdot)\big\|_{H^{-s}(\mathbb{R}^{d-1})} \lesssim_{\psi} \|g\|_{L_p((1+|x_d|)^{-k})},
\end{equation*}
which, as we have seen, is stronger than~\eqref{RestrictionWithSobolev}. Note that in such a formulation,~$k$ might be real. We know the inequality holds true in the case~$k=0$ (from~\cite{ChoGuoLee}) and is almost true when~$k=s=\frac{d-1}{2}$, $p=1$ (from Proposition~\ref{WeightedBesov}). We claim that any triple~$(k ,s,\frac{1}{p})$ that satisfies the necessary conditions of Theorem~\ref{CorollaryOfChoGuoLee} might be represented as a convex combination of the said cases:
\begin{equation*}
\Big(k, s, \frac{1}{p}\Big) = \theta_+\Big(0,s_+, \frac{1}{p_+}\Big) + \theta_-\Big(\frac{d-1}{2},\frac{d-1}{2},1\Big).
\end{equation*} 
Solving several elementary equations, we see
\begin{equation*}
\theta_- = \frac{2k}{d-1},\quad \theta_+ = \frac{d-1-2k}{d-1}, \quad p_+ = p\frac{d-1-2k}{d-1-2k p}, \quad s_+ = \frac{(s-k)(d-1)}{d-1-2k}.
\end{equation*}
We leave to the reader the verification of the conditions
\begin{equation*}
0 < \sigma_{p+}, \quad \hbox{and}\quad -s_{+}\leq \kappa_{p_+}
\end{equation*}
(the easiest way to do this is to sketch the~$3D$-domain of admissible~$(k,s,\frac{1}{p})$) and explain how we interpolate the inequality. First, we note that linear operator that maps~$g$ to~$\psi\hat{g}|_{\Sigma}$ does not depend on the varying parameters, so we may use the classical interpolation theory, specifically, the real interpolation method (see~\cite{BerghLofstrom}). For the image of our operator, we use the formula
\begin{equation*}
(H^{-s_+},B_{2}^{-\frac{d-1}{2},\infty})_{\theta_-,2} = H^{-s},
\end{equation*} 
see~\cite{BerghLofstrom}. For the domain, we need to show that
\begin{equation*}
(L_{p_+}, L_{1}((1+|x_d|)^{-\frac{d-1}{2}}))_{\theta_-,2} \supset L_p((1+|x_d|)^{-k}).
\end{equation*}
In fact,
\begin{equation*}
(L_{p_+}, L_{1}((1+|x_d|)^{-\frac{d-1}{2}}))_{\theta_-,2} = L_{p,2}((1+|x_d|)^{-k}),
\end{equation*}
where the latter space is the space of all functions~$f$ such that~$f(1+|x_d|)^{-k} \in L_{p,2}$ (see~\cite{Freitag}). It is clear that~$L_{p} = L_{p,p} \hookrightarrow L_{p,2}$ since~$p \leq 2$.


\begin{thebibliography}{99}
\bibitem{Agmon} S.~Agmon, \emph{Spectral properties of Schr\"odinger operators and scattering theory.} Ann.\ Scuola Norm.\ Sup.\ Pisa Cl.\ Sci. (4) {\bf 2} (1975), no.~2, 151--218.

\bibitem{Bak} J.-G.~Bak, \emph{Sharp estimates for the Bochner--Riesz operator of negative order in~$\mathbb{R}^2$}, Proc. Amer. Math. Soc. {\bf 125} (1997), no.~7, 1977--1986.

\bibitem{BerghLofstrom} J. Bergh, J. Lofstrom, \emph{Interpolation spaces: an introduction}, Springer-Verlag, 1976.

\bibitem{BourgainBrezis} J. Bourgain, H. Brezis, \emph{On the equation $\mathrm{div} Y = f$ and application to control of phases}, Journ. Amer. Math. Soc. {\bf 16}:2 (2002), 393--426.

\bibitem{BourgainGuth} J. Bourgain, L. Guth, \emph{Bounds on oscillatory integral operators based on multilinear estimates}, Geom. Funct. Anal. {\bf 21}:6 (2011), 1239--1295.

\bibitem{ChoGuoLee} Y. Cho, Z. Guo, S. Lee, \emph{A Sobolev estimate for the adjoint restriction operator}, Math. Ann.  {\bf 362}:3 (2015), 799--815.


\bibitem{Domar} Y. Domar, \emph{On the spectral synthesis problem for~$(n-1)$-dimensional subsets of~$\mathbb{R}^n$,~$n \geq 2$}, Arxiv Mat. {\bf 9}:1 (1970), 23--37.

\bibitem{Fefferman} C. Fefferman, \emph{Inequalities for strongly singular convolution operators}, Acta Math. {\bf 124} (1970), 9--36.

\bibitem{Ferreyra} E. V. Ferreyra, \emph{On a negative result concerning interpolation with change of measure for Lorentz spaces}, Proc. Amer. Math. Soc. {\bf 125}:5 (1997), 1413--1417.

\bibitem{Freitag} D. Freitag, \emph{Real interpolation of weighted~$L_p$-spaces}, Math. Nachr. {\bf 86} (1978), 15--18. 

\bibitem{Goldberg} M. Goldberg, \emph{The Helmholtz equation with~$L^p$ data and Bochner--Riesz multipliers},  Math. Res. Lett. {\bf 23}:6 (2016), 1665--1679.

\bibitem{Goldberg_Schlag} M. Goldberg, W. Schlag, \emph{A Limiting Absorption Principle for the Three-Dimensional Schr\"odinger Equation with $L^p$ Potentials}, Intl. Math. Res. Not. {\bf 2004}:75 (2004), 4049--4071.

%\emph{Dispersive estimates for the Schr\"odinger operator in dimensions
%one and three}, Comm. Math. Phys. {\bf 251}:1 (2004), 157--178.

\bibitem{Guth1} L. Guth, \emph{A restriction estimate using polynomial partitioning}, J. Amer. Math. Soc. {\bf 29} (2016), 371--413.

\bibitem{Guth2} L. Guth, \emph{Restriction estimates using polynomial partitioning II}, preprint 2016.  (arXiv:1603.04250)

\bibitem{Ionescu_Schlag} A.~Ionescu, W.~Schlag, \emph{Agmon--Kato--Kuroda theorems for a large class of
perturbations.},  Duke Math.\ J.~{\bf 131}  (2006),  no.~3, 397--440.

\bibitem{Kerman} R. A. Kerman, \emph{Convolution theorems with weights}, Trans. Amer. Math. Soc. {\bf 280}:1 (1983), 207--219.

\bibitem{KislyakovMaximovStolyarov} S. Kislyakov, D. Maximov, D. Stolyarov, \emph{Differential expression with mixed homogeneity and spaces of smooth functions they generate in arbitrary dimension}, Journ. Funct. Anal. {\bf 269}:10 (2015),  3220--3263.

\bibitem{MuscaluSchlag} C. Muscalu, W. Schlag, \emph{Classical and multilinear Harmonic Analysis, volume 1}, Cambridge Studies in Advanced Mathematics {\bf 137}, 2013.

\bibitem{vanSchaftingen} J. van Schaftingen, \emph{Limiting Sobolev inequalities for vector fields and
cancelling linear differential operators},  J. Eur. Math. Soc. (JEMS) {\bf 15} (2013), no. 3, 877--921.

\bibitem{vanSchaftingen2} J. van Schaftingen, \emph{Limiting Bourgain--Brezis inequalities for systems
of linear differential equations\textup: Theme and variations}, J. of fixed point th. and appl. (2014), 1--25.

\bibitem{Stein} E. M. Stein, \emph{Oscillatory integrals in Fourier analysis}, Beijing Lectures in Harmonic Analysis, pp. 307--355, E. M. Stein. (ed.), Annals of Math. Studies {\bf 112}, Princeton Univ. Press, Princeton, NJ, 1986.

\bibitem{SteinBook} E. M. Stein, \emph{Harmonic analysis: real-variable methods, almost orthogonality, and oscillatory integrals}, Princeton University Press, 1993.


\bibitem{SteinWeiss} E. M. Stein, G. Weiss, \emph{Fractional integrals on $n$-dimensional Euclidean space}, J. Math. Mech. {\bf 7} (1958), 503--514.

\bibitem{Stolyarov} D. M. Stolyarov, \emph{Bilinear embedding theorems for differential operators in~$\mathbb{R}^2$}, Zap. nauchn. sem. POMI {\bf 424} (2014), 210--235 (in Russian); English translation: Journal of Mathematical Sciences (New York) {\bf 206}:5 (2015), 792--807.

\bibitem{Tomas} P. A. Tomas, \emph{A restriction theorem for the Fourier transform}, Bull. Amer. Math. Soc. {\bf 81} (1975), 477--478.
\end{thebibliography}
\end{document}